\documentclass[a4paper, 11pt, leqno]{amsart}
\usepackage{amsmath}
\usepackage{amscd}
\usepackage{amssymb}
\usepackage{amsthm}
\usepackage{amsfonts}
\usepackage{latexsym}

\usepackage[dvipdfmx]{graphicx, xcolor}
  
\theoremstyle{plain}
\newtheorem{lemma}{{\sc Lemma}}[section]
\newtheorem{corollary}[lemma]{{\sc Corollary}}
\newtheorem{proposition}[lemma]{{\sc Proposition}}
\newtheorem{theorem}[lemma]{{\sc Theorem}}

\theoremstyle{definition}
\newtheorem{remark}[lemma]{{\sc Remark}}
\newtheorem{definition}[lemma]{{\sc Definition}}
\newtheorem{example}[lemma]{{\sc Example}}

\numberwithin{equation}{section}

\def\Gb{{\mathfrak{b}}}

\def\Gg{{\mathfrak{g}}}
\def\Gh{{\mathfrak{h}}}

\def\Gn{{\mathfrak{n}}}

\def\Gsl{{\mathfrak{sl}}}

\def\BC{{\mathbb{C}}}
\def\BF{{\mathbb{F}}}
\def\BK{{\mathbb{K}}}

\def\BQ{{\mathbb{Q}}}
\def\BZ{{\mathbb{Z}}}
\def\CA{{\mathcal A}}
\def\CB{{\mathcal B}}

\def\CF{{\mathcal F}}

\def\CI{{\mathcal I}}
\def\CJ{{\mathcal J}}

\def\CM{{\mathcal M}}

\def\CS{{\mathcal S}}

\def\alg{\mathop{\rm alg}\nolimits}

\def\Hom{\mathop{\rm Hom}\nolimits}

\def\id{\mathop{\rm id}\nolimits}

\def\Image{\mathop{\rm Im}\nolimits}
\def\Ker{\mathop{\rm Ker\hskip.5pt}\nolimits}

\def\Mod{\mathop{\rm Mod}\nolimits}

\def\hT{{\hat{T}}}
\def\dT{{\dot{T}}}
\def\he{{\hat{e}}}
\def\de{{\dot{e}}}
\def\te{{\tilde{e}}}
\def\hf{{\hat{f}}}
\def\df{{\dot{f}}}
\def\tf{{\tilde{f}}}

\def\Bi{{\boldsymbol{i}}}
\def\Bj{{\boldsymbol{j}}}

\def\Bn{{\boldsymbol{n}}}

\begin{document}
\title[Modules over quantized coordinate algebras]
{Modules over quantized coordinate algebras and PBW-bases}
\author{Toshiyuki TANISAKI}
\thanks
{
The author was partially supported by Grants-in-Aid for Scientific Research (C) 24540026 
from Japan Society for the Promotion of Science.
}
\address{
Department of Mathematics, Osaka City University, 3-3-138, Sugimoto, Sumiyoshi-ku, Osaka, 558-8585 Japan}
\email{tanisaki@sci.osaka-cu.ac.jp}
\subjclass[2010]{20G05, 17B37}
\date{}
\begin{abstract}
Around 1990 Soibelman constructed certain irreducible modules over the quantized coordinate algebra.
A. Kuniba, M. Okado, Y. Yamada  \cite{KOY} recently found that the relation among natural bases of Soibelman's irreducible module can be described using the relation among the PBW-type bases of the positive part of the quantized enveloping algebra, and proved this fact using case-by-case analysis in rank two cases.
In this paper we will give a realization of Soibelman's module as an induced module, and give a unified proof of the above result of \cite{KOY}.
We also verify Conjecture 1 of \cite{KOY} about certain operators on Soibelman's module.
\end{abstract}
\maketitle

\section{Introduction}
\subsection{}
Let $G$ be a connected simply-connected simple algebraic group over the complex number field $\BC$ with Lie algebra $\Gg$.
The coordinate algebra $\BC[G]$ of $G$ is a Hopf algebra which is dual to the enveloping algebra $U(\Gg)$ of $\Gg$.
So we can naturally define a $q$-analogue $\BC_q[G]$ of $\BC[G]$ as the Hopf algebra dual to the quantized enveloping algebra $U_q(\Gg)$.
This paper is concerned with the representation theory of the quantized coordinate algebra $\BC_q[G]$.

Since the ordinary coordinate algebra $\BC[G]$ is commutative, its irreducible modules are all one-dimensional and are in one-to-one correspondence with the points of $G$; however, the quantized coordinate algebra $\BC_q[G]$ is non-commutative, and its  representation theory is much more complicated.
In fact, Soibelman \cite{S} already pointed out around 1990 that there are not so many one-dimensional $\BC_q[G]$-modules and that there really exist infinite dimensional irreducible $\BC_q[G]$-modules.

Let us recall Soibelman's result more precisely.
He considered the situation where the parameter $q$ is a positive real number with $q\ne1$.
In this case $\BC_q[G]$ is endowed with a structure of $*$-algebra, and we have the  notion of unitarizable $\BC_q[G]$-modules.
Soibelman showed that one-dimensional unitarizable $\BC_q[G]$-modules are in one-to-one correspondence with the points of the maximal compact subgroup 
$H_{\rm cpt}$ of the maximal torus $H$ of $G$.
Denote the one-dimensional $\BC_q[G]$-module corresponding to $h\in H_{\rm cpt}$ by $\BC_h$.
On the other hand infinite-dimensional irreducible unitarizable $\BC_q[G]$-modules are constructed as follows.
In the case $G=SL_2$ Vaksman and Soibelman \cite{VS} constructed an irreducible unitarizable $\BC_q[SL_2]$-modules $\CF$ with basis $\{m_n\}_{n\in\BZ, n\geqq0}$ using an explicit description of $\BC_q[SL_2]$.
For general $G$ denote by $I$ the index set of simple roots.
For each $i\in I$ we have a natural Hopf algebra homomorphism $\pi_i:\BC_q[G]\to\BC_{q_i}[SL_2]$, where $q_i$ is some power of $q$.
Via $\pi_i$ we can regard $\CF$ as a $\BC_q[G]$-module.
Denote this $\BC_q[G]$-module by $\CF_i$.
Let $W$ be the Weyl group of $G$.
For $w\in W$ we denote the length of $w$ by $\ell(w)$.
Take $w\in W$ 
and its reduced expression $w=s_{i_1}\cdots s_{i_{\ell(w)}}$\; $(i_r\in I)$  as a product of simple reflections.
Soibelman proved that the tensor product $\CF_{i_1}\otimes\cdots\otimes\CF_{i_{\ell(w)}}$ is a unitarizable irreducible $\BC_q[G]$-module.
Moreover, he showed that $\CF_{i_1}\otimes\cdots\otimes\CF_{i_{\ell(w)}}$ depends only on $w$.
So we can denote this $\BC_q[G]$-module by $\CF_w$.
It is also verified in \cite{S} that any irreducible unitarizable $\BC_q[G]$-module is isomorphic to the tensor product $\CF_w\otimes \BC_h$ for $w\in W$, $h\in H_{\rm cpt}$.

As for further development of the theory of $\BC_q[G]$-modules we refer to Joseph \cite{J}, Yakimov \cite{Y}.

Quite recently the above work of Soibelman has been taken up again by 
Kuniba, Okado, Yamada \cite {KOY}.
Let $w_0\in W$ be the longest element.
Note that for each reduced expression $w_0=s_{i_1}\cdots s_{i_{\ell(w_0)}}$ of $w_0$ we have a basis 
\[
\CB_{i_1,\dots, i_{\ell(w_0)}}=
\{m_{n_1}\otimes\cdots \otimes m_{n_{\ell(w_0)}}\mid n_1,\dots, n_{\ell(w_0)}\geqq0\}
\]
of $\CF_{w_0}=\CF_{i_1}\otimes\cdots\otimes\CF_{i_{\ell(w_0)}}$
parametrized by the set of $\ell(w_0)$-tuples $(n_1,\dots,n_{\ell(w_0)})$ of non-negative integers.
On the other hand, by Lusztig's result,
for each reduced expression $w_0=s_{i_1}\cdots s_{i_{\ell(w_0)}}$ of $w_0$ we have a PBW-type basis 
$\CB'_{i_1,\dots, i_{\ell(w_0)}}$
of the positive part $U_q(\Gn^+)$ of $U_q(\Gg)$ parametrized by the set of $\ell(w_0)$-tuples of non-negative integers.
Kuniba, Okado, Yamada observed in \cite {KOY} that for two reduced expressions 
$w_0=s_{i_1}\cdots s_{i_{\ell(w_0))}}=s_{j_1}\cdots s_{j_{\ell(w_0)}}$ of $w_0$
the transition matrix between $\CB_{i_1,\dots, i_{\ell(w_0)}}$ and 
$\CB_{j_1,\dots, j_{\ell(w_0)}}$
coincides with the
transition matrix between $\CB'_{i_1,\dots, i_{\ell(w_0)}}$ and 
$\CB'_{j_1,\dots, j_{\ell(w_0)}}$ up to  a normalization factor.
They proved this fact partly using a case-by-case argument in rank two cases.

In the present paper we give a new approach to the results of Soibelman \cite{S} and Kuniba, Okado, Yamada \cite{KOY}.
We work over the rational function field $\BF=\BQ(q)$; however, our arguments also hold in a more general situation (see Section \ref{sec:comment} below).
Let $\Gg=\Gn^+\oplus\Gh\oplus\Gn^-$ be the triangular decomposition of $\Gg$.
Let $N^\pm$ and $B^\pm$ be the subgroups of $G$ corresponding to $\Gn^\pm$ and $\Gh\oplus\Gn^{\pm}$ respectively.
For each $w\in W$ we define a $\BC_q[G]$-module $\overline{\CM}_w$ as the induced module from a one-dimensional representation of a certain subalgebra $\BC_q[(N^-\cap wN^+w^{-1})\backslash G]$ of $\BC_q[G]$.
We will show that $\overline{\CM}_w$ is an irreducible $\BC_q[G]$-module and that for each reduced expression $w=s_{i_1}\cdots s_{i_{\ell(w)}}$ we have a decomposition 
$
\overline{\CM}_w\cong
{\CF}_{i_1}\otimes\cdots\otimes {\CF}_{i_{\ell(w)}}
$
into tensor product.
This gives a new proof of Soibelman's result.
We will also show that there exists a natural linear isomorphism 
\begin{equation}\label{eq:ISOM}
\overline{\CM}_w \cong
U_q(\Gn^+\cap w\Gn^-),
\end{equation}
where $U_q(\Gn^+\cap w\Gn^-)$ is a certain subalgebra of $U_q(\Gg)$ defined in terms of Lusztig's braid group action (see De Concini, Kac, Procesi \cite{DKP}, Lusztig \cite{Lbook}).
From this we obtain (in the case $w=w_0$) the result of Kuniba, Okado, Yamada described above.
As in \cite{KOY} a certain localization of $\BC_q[G]$ plays a crucial role in the proof.
More precisely, for each $w\in W$ we consider the localization $\BC_q[wN^+B^-]$ of $\BC_q[G]$, which is a  $q$-analogue of $\BC[wN^+B^-]$.
In addition to it, we use the Drinfeld pairing between the positive and negative parts of the quantized enveloping algebra in constructing the isomorphism \eqref{eq:ISOM}.
A crucial difference between Soibelman's approach and our approach is that, 
instead of the decomposition
\[
\BC_q[G]=\BC_q[G/N^+]\BC_q[G/N^-]
\]
used by Soibelman, 
we utilize the $q$-analogue of the decomposition
\[
\BC[B^-w_0B^-]
\cong
\BC[B^-w_0B^-/B^-]
\otimes
\BC[N^-\backslash B^-w_0B^-]
\]
in the case $w=w_0$, and 
\begin{align*}
\BC[wN^+B^-]
\cong&
\BC[(wN^+w^{-1}\cap N^-)]
\otimes
\BC[(wN^+w^{-1}\cap N^+)wB^-]
\\
\cong&
\BC[(wN^+w^{-1}\cap N^-)]
\otimes
\BC[(wN^+w^{-1}\cap N^-)\backslash wN^+B^-],
\end{align*}
for general $w$, 
which is more natural from geometric point of view.
As a consequence of our approach, we can also show  easily 
a conjecture of Kuniba, Okado, Yamada \cite[Conjecture 1]{KOY} concerning the action of a certain element of $\BC_q[wN^+B^-]$ on $\overline{\CM}_w$.

We finally note that our results hold true for any symmetrizable Kac-Moody algebra (see Section \ref{sec:comment} below).
We hope this fact will be useful in the investigation of 3-dimensional integrable systems, which was the original motivation of \cite{KOY}.
After writing up the first draft of this paper Yoshiyuki Kimura pointed out to me that Proposition \ref{prop:decomposition} below in the Kac-Moody case is not an obvious fact which is stated as a conjecture in Berenstein and Greenstein \cite[Conjecture 5.5]{BG}.
In the present manuscript we have included a proof of Proposition \ref{prop:decomposition} which works for the Kac-Moody case.
We heard that Kimura also proved it by a different method (see Kimura \cite{K}).

After finishing this work we heard that Yoshihisa Saito \cite{Sa} has obtained similar results by a different method.

\subsection{}
We use the following notation for Hopf algebras throughout the paper.
For a Hopf algebra $H$ over a field $\BK$ we denote its multiplication, comultiplication, counit, antipode by
$m_H:H\otimes_\BK H\to H$, $\Delta_H:H\to H\otimes_\BK H$, $\varepsilon_H:H\to\BK$, $S_H:H\to H$ respectively.
The subscript $H$ is often omitted.
For left $H$-modules $V_0,\cdots, V_m$ we regard
$V_0\otimes_\BK\cdots\otimes_\BK V_m$ as a left $H$-module via the iterated comultiplication $\Delta_m:H\to H^{\otimes m+1}$.
We will occasionally use Sweedler's notation for the comultiplication
\[
\Delta(h)=\sum_{(h)}h_{(0)}\otimes h_{(1)}\qquad(h\in H),
\]
and the iterated comultiplication
\[
\Delta_{m}(h)=\sum_{(h)_{m}}h_{(0)}\otimes \cdots \otimes h_{(m)}\qquad(h\in H).
\]

\subsection{}
I would like to thank Masato Okado and Yoshiyuki Kimura for some useful discussion.

\section{Quantized enveloping algebras}
\subsection{}
Let $G$ be a connected simply-connected simple algebraic group over the complex number field $\BC$.
We take Borel subgroups $B^+$ and $B^-$ such that $H=B^+\cap B^-$ is a maximal torus of $G$, and set $N^\pm=[B^\pm,B^\pm]$.
The Lie algebras of $G$, $B^\pm$, $H$, $N^\pm$ are denoted by
$\Gg$, $\Gb^\pm$, $\Gh$, $\Gn^\pm$ respectively.
We denote by $P$ the  character group of $H$.
Let $\Delta^+$ and $\Delta^-$ be the subsets of $P$ consisting of weights of $\Gn^+$ and $\Gn^-$ respectively, and set
$\Delta=\Delta^+\cup\Delta^-$.
Then $\Delta$ is the set of roots of $\Gg$ with respect to $\Gh$.
We denote by $\Pi=\{\alpha_i\mid i\in I\}$ the set of simple roots of $\Delta$ such that $\Delta^+$ is the set of positive roots.
Let $P^+$ be the set of dominant weights in $P$ with respect to $\Pi$, and set $P^-=-P^+$.
We set
\[
Q=\sum_{i\in I}\BZ\alpha_i,\qquad
Q^+=\sum_{i\in I}\BZ_{\geqq0}\alpha_i,
\]
where $\BZ_{\geqq0}$ denotes the set of non-negative integers.
The Weyl group $W=N_G(H)/H$ naturally acts on $P$ and $Q$.
By differentiation we will regard $P$ as a $\BZ$-lattice of $\Gh^*$ in the following.
We denote by 
\[
(\;,\;):\Gh^*\times\Gh^*\to\BC
\]
the $W$-invariant non-degenerate symmetric bilinear form such that 
$(\alpha,\alpha)=2$ for short roots $\alpha$.
For $\alpha\in\Delta$ we set
$\alpha^\vee=2\alpha/(\alpha,\alpha)$.
As a subgroup of $GL(\Gh^*)$ the Weyl group $W$ is generated by the simple reflections $s_i\;(i\in I)$ given by
$s_i(\lambda)=\lambda-(\lambda,\alpha_i^\vee)\alpha_i$\; $(\lambda\in\Gh^*)$.
We denote by $\ell:W\to\BZ_{\geqq0}$ the length function with respect to the generating set $\{s_i\mid i\in I\}$ of $W$.
The longest element of $W$ is denoted by $w_0$.
For $w\in W$ we set 
\[
\CI_w=\{(i_1,\dots, i_{\ell(w)})\in I^{\ell(w)}\mid w=
s_{i_1}\cdots s_{i_{\ell(w)}}\}.
\]
\subsection{}
For $n\in\BZ$ we set
\[
[n]_q=\frac{q^n-q^{-n}}{q-q^{-1}}\in\BZ[q,q^{-1}].
\]
For $m\in\BZ_{\geqq0}$ we set
\[
[m]_q!=[m]_q[m-1]_q\cdots[1]_q.
\]
For $m, n\in\BZ$ with $m\geqq0$
we set
\[
\begin{bmatrix}
n
\\
m
\end{bmatrix}_q
=\frac{[n]_q[n-1]_q\cdots[n-m+1]_q}{[m]_q[m-1]_q\cdots[1]_q}
\in\BZ[q,q^{-1}].
\]
For $i\in I$ we set $q_i=q^{(\alpha_i,\alpha_i)/2}$, and for $i, j\in I$ we further set $a_{ij}=(\alpha_i^\vee,\alpha_j)$.

We denote by $U=U_q(\Gg)$ the quantized enveloping algebra of $\Gg$.
Namely, it is an associative algebra over $\BF=\BQ(q)$ generated by the elements
$k_i^{\pm1}$, $e_i$, $f_i$\;$(i\in I)$
satisfying the defining relations
\begin{align*}
&k_ik_i^{-1}=k_i^{-1}k_i=1\qquad&(i\in I),
\\
&k_ik_j=k_jk_i&(i, j\in I),
\\
&k_ie_jk_i^{-1}=q_i^{a_{ij}}e_j
&(i, j\in I),
\\
&k_if_jk_i^{-1}=q_i^{-a_{ij}}e_j
&(i, j\in I),
\\
&e_if_j-f_je_i=\delta_{ij}
\frac{k_i-k_i^{-1}}{q_i-q_i^{-1}}
&(i,j\in I),
\\
&\sum_{m=0}^{1-a_{ij}}
(-1)^m
e_i^{(1-a_{ij}-m)}e_je_i^{(m)}
=0
\quad&(i, j\in I, i\ne j),
\\
&\sum_{m=0}^{1-a_{ij}}
(-1)^m
f_i^{(1-a_{ij}-m)}f_jf_i^{(m)}
=0
\quad&(i, j\in I, i\ne j),
\end{align*}
where
\[
e_i^{(m)}=\frac1{[m]_{q_i}!}e_i^m,
\qquad
f_i^{(m)}=\frac1{[m]_{q_i}!}f_i^m
\qquad(m\in\BZ_{\geqq0}).
\]
We endow $U$ with the Hopf algebra structure given by
\[
\Delta(k_i^{\pm1})=k_i^{\pm1}\otimes k_i^{\pm1},\quad
\Delta(e_i)=e_i\otimes1+k_i\otimes e_i,\quad
\Delta(f_i)=f_i\otimes k_i^{-1}+1\otimes f_i,
\]
\[
\varepsilon(k_i^{\pm1})=1,\quad\varepsilon(e_i)=\varepsilon(f_i)=0,
\]
\[
S(k_i^{\pm1})=k_i^{\mp1},\quad
S(e_i)=-k_i^{-1}e_i,\quad
S(f_i)=-f_ik_i.
\]
We define subalgebras $U^0=U_q(\Gh)$, $U^+=U_q(\Gn^+)$, $U^-=U_q(\Gn^-)$, $U^{\geqq0}=U_q(\Gb^+)$, $U^{\leqq0}=U_q(\Gb^-)$ by
\[
U^0=\langle k_i^{\pm1}\mid i\in I\rangle,\quad
U^{+}=\langle e_i\mid i\in I\rangle,
\qquad
U^{-}=\langle f_i\mid i\in I\rangle,
\]
\[
U^{\geqq0}=\langle k_i^{\pm1}, e_i\mid i\in I\rangle,
\quad
U^{\leqq0}=\langle k_i^{\pm1}, f_i\mid i\in I\rangle,
\]
respectively.
Then $U^0$, $U^{\geqq0}$, $U^{\leqq0}$ are Hopf subalgebras.
The multiplication of $U$ induces isomorphisms
\[
U\cong U^+\otimes U^0\otimes U^-\cong U^-\otimes U^0\otimes U^+,
\]
\[
U^{\geqq0}\cong U^0\otimes U^+\cong U^+\otimes U^0,
\qquad
U^{\leqq0}\cong U^0\otimes U^-\cong U^-\otimes U^0.
\]
\begin{remark}
In this paper $\otimes_\BF$ is often written as $\otimes$.
\end{remark}
For $\gamma=\sum_{i\in I}m_i\alpha_i\in Q$ we set
\[
k_\gamma=\prod_{i\in I}k_i^{m_i}\in U^0.
\]
Then we have
$U^0=\bigoplus_{\gamma\in Q}\BF k_\gamma$, and hence $U^0$ is isomorphic to the group algebra of $Q$.
For $\gamma\in Q^+$ we define 
$U^\pm_{\pm\gamma}$ by
\[
U^\pm_{\pm\gamma}=
\{u\in U^{\pm}\mid
k_iuk_i^{-1}=q_i^{\pm(\alpha_i^\vee,\gamma)}u\;(i\in I)\}.
\]
Then we have $U^\pm=\bigoplus_{\gamma\in Q^+}U^{\pm}_{\pm\gamma}$.
\subsection{}
There exists a unique bilinear map
\begin{equation}
\tau:U^{\geqq0}\times U^{\leqq0}\to\BF
\end{equation}
characterized by the properties:
\begin{align}
\label{eq:Dr1}
&(\tau\otimes\tau)(\Delta(x),y_2\otimes y_1)=\tau(x,y_1y_2)
\quad(x\in U^{\geqq0},\;y_1, y_2\in U^{\leqq0}),
\\
\label{eq:Dr2}
&(\tau\otimes\tau)(x_1\otimes x_2,\Delta(y))=\tau(x_1x_2,y)
\;\;\;(x_1, x_2\in U^{\geqq0},\;y\in U^{\leqq0}),
\\
\label{eq:Dr3}
&\tau(e_i,k_\lambda)=\tau(k_\lambda,f_i)=0
\qquad(i\in I,\; \lambda\in Q),
\\
\label{eq:Dr4}
&\tau(k_\lambda,k_\mu)=q^{(\lambda,\mu)}
\qquad(\lambda, \mu\in Q),
\\
\label{eq:Dr5}
&\tau(e_i,f_j)=\delta_{ij}\frac1{q_i-q_i^{-1}}
\qquad(i, j\in I).
\end{align}
We call it the Drinfeld pairing.
It also satisfies the following properties:
\begin{align}
\label{eq:Dr6}
&\tau(Sx,Sy)=\tau(x,y)
\qquad(x\in U^{\geqq0},y\in U^{\leqq0}),
\\
\label{eq:Dr7}
&\tau(k_\lambda x,k_\mu y)=\tau(x,y)q^{(\lambda,\mu)}
\qquad(x\in U^+, y\in U^-),
\\
\label{eq:Dr8}
&\gamma, \delta\in Q^+,\;\gamma\ne\delta\;\Longrightarrow\;
\tau|_{U^+_{\gamma}\times U^-_{-\delta}}=0,
\\
\label{eq:Dr9}
&\gamma\in Q^+\;\Longrightarrow\;
\text{
$\tau|_{U^+_{\gamma}\times U^-_{-\gamma}}$ is non-degenerate}.
\end{align}

\subsection{}
For a $U^0$-module $M$ and $\lambda\in P$ we set
\[
M_\lambda=\{m\in M\mid k_im=q_i^{(\lambda,\alpha_i^\vee)}m\;(i\in I)\}.
\]
We say that a $U^0$-module $M$ is a weight module if $M=\bigoplus_{\lambda\in P}M_\lambda$.

For a $U$-module $V$ we regard $V^*=\Hom_\BF(V,\BF)$ as a right $U$-module by
\[
\langle v^*u,v\rangle=\langle v^*,uv\rangle
\qquad
(v\in V, v^*\in V^*, u\in U).
\]
Denote by $\Mod_0(U)$ (resp.\ $\Mod^r_0(U)$) the category of finite-dimensional left (resp.\ right) $U$-modules which is a weight module as a $U^0$-module.
Here, a right $U^0$-module $M$ is regarded as a left $U^0$-module by 
\[
tm:=mt\qquad(m\in M, t\in U^0).
\]
If $V\in\Mod_0(U)$, then we have $V^*\in\Mod^r_0(U)$.
This gives an anti-equivalence $\Mod_0(U)\ni V\mapsto V^\star\in\Mod^r_0(U)$ of categories.

For $\lambda\in P^-$ we denote by $V(\lambda)$ the finite-dimensional irreducible (left) $U$-module with lowest weight $\lambda$.
Namely, $V(\lambda)$ is a finite-dimensional $U$-module generated by a non-zero element $v_\lambda\in V(\lambda)_\lambda$ satisfying $f_iv_\lambda=0\;(i\in I)$.
Then $\Mod_0(U)$ is a semisimple category with simple objects $V(\lambda)$ for $\lambda\in P^-$ (see Lusztig \cite{Lbook}).
For $\lambda\in P^-$ we set $V^*(\lambda)=(V(\lambda))^*$, and define $v^*_\lambda\in V^*(\lambda)_\lambda$ by $\langle v^*_\lambda, v_\lambda\rangle=1$.

The following well known fact will be used occasionally in this paper
(see e.g. \cite[Lemma 2.1]{T}).

\begin{proposition}
\label{prop:UV}
Let $\gamma\in Q^+$.
\begin{itemize}
\item[(i)] 
For sufficiently small $\lambda\in P^-$
the linear map
$U_{\gamma}^+\ni x\mapsto xv_\lambda\in V(\lambda)_{\lambda+\gamma}$ is bijective.
\item[(ii)] 
For sufficiently small $\lambda\in P^-$
the linear map
$U_{-\gamma}^-\ni y\mapsto v^*_\lambda y\in V^*(\lambda)_{\lambda+\gamma}$ is bijective.
\end{itemize}
\end{proposition}
\begin{remark}
In this paper the expression 
``for sufficiently small $\lambda\in P^- ...$''
means that
``
there exists some $\mu\in P^-$ such that for any $\lambda\in\mu+P^-$ ...''.
\end{remark}

\subsection{}
For $i\in I$ and $M\in\Mod_0(U)$ we denote by $\dT_i, \hT_i\in GL(M)$ the operators denoted by $T''_{i,1}$ and $T''_{i,-1}$ respectively in 
\cite{Lbook}.
We have also algebra automorphisms $\dT_i$,  $\hT_i$ of $U$ satisfying
\[
\dT_i(um)=\dT_i(u)\dT_i(m), \qquad
\hT_i(um)=\hT_i(u)\hT_i(m)
\]
for $u\in U$, $m\in M\in\Mod_0(U)$.
They are given by
\begin{align*}
\dT_i(e_j)
=&
\begin{cases}
-f_ik_i\quad&(j=i)\\
\displaystyle{\sum_{r=0}^{-a_{ij}}}(-1)^{r}q_i^{-r}e_i^{(-a_{ij}-r)}e_je_i^{(r)}
\quad&(j\ne i),
\end{cases}
\\
\dT_i(f_j)
=&
\begin{cases}
-k_i^{-1}e_i\quad&(j=i)\\
\displaystyle{\sum_{r=0}^{-a_{ij}}}(-1)^{-a_{ij}-r}q_i^{-a_{ij}-r}f_i^{(-a_{ij}-r)}f_jf_i^{(r)}
\quad&(j\ne i),
\end{cases}
\\
\hT_i(e_j)
=&
\begin{cases}
-f_ik_i^{-1}\quad&(j=i)\\
\displaystyle{\sum_{r=0}^{-a_{ij}}}(-1)^{r}q_i^{r}e_i^{(-a_{ij}-r)}e_je_i^{(r)}
\quad&(j\ne i),
\end{cases}
\\
\hT_i(f_j)
=&
\begin{cases}
-k_ie_i\quad&(j=i)\\
\displaystyle{\sum_{r=0}^{-a_{ij}}}(-1)^{-a_{ij}-r}q_i^{-(-a_{ij}-r)}f_i^{(-a_{ij}-r)}f_jf_i^{(r)}
\quad&(j\ne i),
\end{cases}
\\
\dT_i(k_\gamma)
=&
\hT_i(k_\gamma)
=
k_{s_i\gamma}.
\end{align*}
By \cite{Lbook} both
$\{\dT_i\}_{i\in I}$  and $\{\hT_i\}_{i\in I}$ satisfy the braid relation, and hence we obtain the operators
$\{\dT_w\}_{w\in W}$,  $\{\hT_w\}_{w\in W}$ given by
\[
\dT_w=\dT_{i_1}\cdots\dT_{i_{\ell(w)}},
\quad
\hT_w=\hT_{i_1}\cdots\hT_{i_{\ell(w)}}
\qquad(i_1,\cdots, i_{\ell(w)})\in\CI_w).
\]
By the description of $\dT_i$, $\hT_i$ as automorphisms of $U$ we have
\begin{equation}
\label{eq:epsT}
\varepsilon(\dT_w(u))=
\varepsilon(\hT_w(u))=
\varepsilon(u)
\qquad(w\in W, u\in U).
\end{equation}

For $w\in W$ and $M\in\Mod^r_0(U)$ we define a right action of $\dT_w$ (resp.\ $\hT_w$) on $M$ by
\[
\langle m\dT_w,m^*\rangle=\langle m,\dT_wm^*\rangle
\qquad
(\text{resp.}\;\;
\langle m\hT_w,m^*\rangle=\langle m,\hT_wm^*\rangle)
\]
for $m\in M$, $m^*\in M^*$.
We can easily check the following fact.
\begin{lemma}
\label{lem:dThT}
Let $w\in W$.
Then as algebra automorphisms of $U$ we have
$\hT_w=S^{-1}\dT_wS$.
\end{lemma}
Let 
$w\in W$ and 
$\Bi=(i_1,\cdots, i_m)\in\CI_w$.
For $r=1,\dots, m$ set
\begin{align*}
&
k_{\Bi,r}=k_{s_{i_1}\cdots s_{i_{r-1}}\alpha_{i_r}},
\\
&\de_{\Bi,r}=\dT_{i_1}\cdots\dT_{i_{r-1}}(e_{i_r}),\qquad
\df_{\Bi,r}=\dT_{i_1}\cdots\dT_{i_{r-1}}(f_{i_r}),
\\
&\te_{\Bi,r}=\dT_{i_m}^{-1}\cdots\dT_{i_{r+1}}^{-1}(e_{i_r}),\qquad
\tf_{\Bi,r}=\dT_{i_m}^{-1}\cdots\dT_{i_{r+1}}^{-1}(f_{i_r}),
\\
&\he_{\Bi,r}=\hT_{i_1}\cdots\hT_{i_{r-1}}(e_{i_r}),\qquad
\hf_{\Bi,r}=\hT_{i_1}\cdots\hT_{i_{r-1}}(f_{i_r}).
\end{align*}
By \cite{Lbook} we have
$\de_{\Bi,r}, \te_{\Bi,r}, \he_{\Bi,r}\in U^+$, 
$\df_{\Bi,r}, \tf_{\Bi,r}, \hf_{\Bi,r}\in U^-$.
For $n\in\BZ_{\geqq0}$ set
\begin{align*}
&\de_{\Bi,r}^{(n)}=\dT_{i_1}\cdots\dT_{i_{r-1}}(e_{i_r}^{(n)}),\qquad
\tf_{\Bi,r}^{(n)}=\dT_{i_m}^{-1}\cdots\dT_{i_{r+1}}^{-1}(f_{i_r}^{(n)}),\\&
\he_{\Bi,r}^{(n)}=\hT_{i_1}\cdots\hT_{i_{r-1}}(e_{i_r}^{(n)}),
\end{align*}
and for 
$\Bn=(n_1,\dots, n_m)\in(\BZ_{\geqq0})^m$ set
\begin{align*}
&\de_{\Bi}^{(\Bn)}=
\de_{\Bi,m}^{(n_m)}\cdots\de_{\Bi,1}^{(n_1)}
,\qquad
\df_{\Bi}^\Bn=\df_{\Bi,m}^{n_m}\cdots\df_{\Bi,1}^{n_1},
\\
&\te_{\Bi}^{\Bn}=
\te_{\Bi,1}^{n_1}\cdots\te_{\Bi,m}^{n_m}
,\qquad
\tf_{\Bi}^{(\Bn)}=\tf_{\Bi,1}^{(n_1)}\cdots\tf_{\Bi,m}^{(n_m)},
\\
&\he_{\Bi}^{(\Bn)}=
\he_{\Bi,m}^{(n_m)}\cdots\he_{\Bi,1}^{(n_1)}
,\qquad
\hf_{\Bi}^\Bn=\hf_{\Bi,m}^{n_m}\cdots\hf_{\Bi,1}^{n_1}.
\end{align*}
\begin{proposition}[\cite{LS}]
\label{prop:base2a}
Let $w\in W$ and $\Bi\in\CI_w$.
Then we have
\[
\tau(\he_\Bi^{(\Bn)},\hf_\Bi^{\Bn'})
=\delta_{\Bn,\Bn'}
\prod_{r=1}^{\ell(w)}
c_{q_{i_r}}(n_r),
\]
where
\[
c_q(n)=[n]!q^{-n(n-1)/2}(q-q^{-1})^{-n}.
\]
\end{proposition}
The following result will be used frequently in this paper.
\begin{proposition}
[\cite{KR}, \cite{LS}, \cite{Lbook}]
\label{prop:dT}
We have
\begin{align*}
\Delta(\dT_i)=&(\dT_i\otimes \dT_i)\exp_{q_i}((q_i-q_i^{-1})f_i\otimes e_i)\\
=&\exp_{q_i}((q_i-q_i^{-1})k_i^{-1}e_i\otimes f_ik_i)(\dT_i\otimes \dT_i),
\end{align*}
where
\[
\exp_q(x)=\sum_{n=0}^\infty\frac{q^{n(n-1)/2}}{[n]!}x^n.
\]
\end{proposition}
\begin{corollary}
\label{cor:dT}
For $w\in W$ and $\Bi=(i_1,\dots, i_m)\in\CI_w$ we have

\begin{align*}
\Delta(\dT_w)
=&
(\dT_w\otimes \dT_w)
\exp_{q_{i_1}}(X_1)
\cdots
\exp_{q_{i_m}}(X_m)
\\
=&\exp_{q_{i_1}}(Y_1)
\cdots
\exp_{q_{i_m}}(Y_m)
(\dT_w\otimes \dT_w),
\\
\Delta(\dT_w^{-1})
=&
\exp_{q_{i_m}^{-1}}(-X_m)
\cdots
\exp_{q_{i_1}^{-1}}(-X_1)
(\dT_w^{-1}\otimes \dT_w^{-1})
\\
=&
(\dT_w^{-1}\otimes \dT_w^{-1})
\exp_{q_{i_m}^{-1}}(-Y_m)
\cdots
\exp_{q_{i_1}^{-1}}(-Y_1),
\end{align*}
where
\begin{align*}
X_r=(q_{i_r}-q_{i_r}^{-1})
\tf_{\Bi,r}\otimes \te_{\Bi,r},\qquad
Y_r=(q_{i_r}-q_{i_r}^{-1})
k_{\Bi,r}^{-1}\de_{\Bi,r}\otimes \df_{\Bi,r}k_{\Bi,r}.
\end{align*}
\end{corollary}

\begin{lemma}
\label{lem:delta-U}
For $w\in W$ we have 
\begin{align*}
&\Delta(\dT_w(U^+))\subset U\otimes(\dT_w(U^+))U^0,
\quad
\Delta(\dT_w^{-1}(U^+))\subset (\dT_w^{-1}(U^+))U^0\otimes U,
\\
&\Delta(\dT_w(U^-))\subset (\dT_w(U^-))U^0\otimes U,
\quad
\Delta(\dT_w^{-1}(U^+))\subset U\otimes(\dT_w^{-1}(U^+))U^0.
\end{align*}

\end{lemma}
\begin{proof}
We only show the first formula since the proof of other formulas are similar.
To show the first formula we need to show that for
$y\in \dT_wU^+$ we have $(\dT_w^{-1}\otimes\dT_w^{-1})(\Delta(y))\in U\otimes U^{\geqq0}$.

For each $\lambda\in P^-$ take $v_{w_0\lambda}\in V(\lambda)_{w_0\lambda}\setminus\{0\}$.
Then for $u\in U$ we have $u\in U^{\geqq0}$ if and only if 
$u(M\otimes v_{w_0\lambda})\subset M\otimes v_{w_0\lambda}$ for any $\lambda\in P^-$ and $M\in\Mod_0(U)$ (see the proof of \cite[Proposition 5.11]{Jan}).
By this fact together with \cite[Proposition 5.11]{Jan} it is sufficient to show 
that 
for $M_1, M_2\in\Mod_0(U)$ and $\lambda\in P^-$ 
the element
\[
(\id\otimes\Delta)\{(\dT_w^{-1}\otimes\dT_w^{-1})(\Delta(y))\}
\in U\otimes U\otimes U
\]
sends $M_1\otimes M_2\otimes v_{w_0\lambda}$ to itself.
As an operator on the tensor product of two integrable modules we have
\[
(\dT_w^{-1}\otimes\dT_w^{-1})(\Delta(y))
=
(\dT_w^{-1}\otimes\dT_w^{-1})\circ(\Delta(y))\circ(\dT_w\otimes\dT_w).
\]
Take $\Bi=(i_1,\dots, i_m)\in\CI_w$.
By Corollary \ref{cor:dT} we have
\[
\dT_w\otimes \dT_w=\Delta(\dT_w)\circ Z^{-1},\quad
Z=
\exp_{q_{i_1}}(X_1)
\cdots
\exp_{q_{i_m}}(X_m),
\]
where $X_1,\dots, X_m$ are as in Corollary \ref{cor:dT}.
Hence we have
\[
(\dT_w^{-1}\otimes\dT_w^{-1})\circ(\Delta(y))\circ(\dT_w\otimes\dT_w)
=Z\circ\Delta(\dT_w^{-1}(y))\circ Z^{-1}.
\]
Therefore, our assertion is a consequence of  $\dT_w^{-1}(y)\in U^{+}$, and $X_r\in U^-\otimes U^+$.
\end{proof}

For $w\in W$ set
\begin{align}
&U^-[\dT_w]=U^-\cap \dT_w(U^{\geqq0}),\qquad
U^+[\dT_w]=U^+\cap \dT_w(U^{\leqq0}),
\\
&U^-[\dT_w^{-1}]=U^-\cap \dT_w^{-1}(U^{\geqq0}),\qquad
U^+[\dT_w^{-1}]=U^+\cap \dT_w^{-1}(U^{\leqq0}),
\\
&
U^-[\hT_w]=U^-\cap \hT_w(U^{\geqq0}),\qquad
U^+[\hT_w]=U^+\cap \hT_w(U^{\leqq0}).
\end{align}

We can easily show the following using Lemma \ref{lem:dThT}.
\begin{lemma}
\label{lem:transfer}
For 
$w\in W$, $\Bi=(i_1,\cdots, i_m)\in\CI_w$ we have
\[
S^{-1}\dT_w(\te_{\Bi}^{\Bn})
=\hf_{\Bi}^{\Bn}.
\]
Hence we have
\[
\dT_w^{-1}S(U^-[\hT_w])=U^+[\dT_w^{-1}].
\]
\end{lemma}

The following well-known result is an easy consequence of the existence of the PBW-type base of $U^\pm$.
Here, we give another proof which works for the quantized enveloping algebra of  any symmetrizable Kac-Moody Lie algebra.
\begin{proposition}
\label{prop:decomposition}
The multiplication of $U$ induces the isomorphisms
\begin{align}
\label{eq:Ubunkai-d}
U^{\pm}\cong&
U^{\pm}[\dT_w]\otimes(U^\pm\cap\dT_w(U^\pm))
\\
\nonumber
\cong&
(U^\pm\cap\dT_w(U^\pm))\otimes U^{\pm}[\dT_w],
\\
\label{eq:Ubunkai-di}
U^{\pm}\cong&
U^{\pm}[\dT_w^{-1}]\otimes(U^\pm\cap\dT_w^{-1}(U^\pm))
\\
\nonumber
\cong&
(U^\pm\cap\dT_w^{-1}(U^\pm))\otimes U^{\pm}[\dT_w^{-1}],
\\
\label{eq:Ubunkai-h}
U^{\pm}\cong&
U^{\pm}[\hT_w]\otimes
(U^\pm\cap\hT_w(U^\pm))
\\
\nonumber
\cong&
(U^\pm\cap\hT_w(U^\pm))\otimes
U^{\pm}[\hT_w].
\end{align}
\end{proposition}
\begin{proof}
We first note that \eqref{eq:Ubunkai-h} follows easily from \eqref{eq:Ubunkai-d} and Lemma \ref{lem:dThT}.
Consider the ring involution
\[
a:U\to U\quad
(q\mapsto q^{-1},\;k_\lambda\mapsto k_\lambda^{-1},\;
e_i\mapsto -k_i^{-1}e_i,\;f_i\mapsto -f_ik_i),
\]
and the ring anti-involution
\[
b:U\to U\quad
(q\mapsto q^{-1},\;k_\lambda\mapsto k_\lambda^{-1},\;
e_i\mapsto f_i,\;f_i\mapsto e_i).
\]
By $b\dT_w=\dT_wb$ and $b(U^+)=U^-$
the statements for $U^-$ are consequences of those for $U^+$.
By $a\dT_wa=\dT_{w^{-1}}^{-1}$ and $a(U^+_\gamma)=k_\gamma^{-1}U^+_\gamma$ ($\gamma\in Q^+$) 
the statements for $\dT_{w}^{-1}$ are consequences of those for $\dT_{w^{-1}}$.
Hence we have only to show
\begin{align}
\label{eq:re1}
U^{+}\cong&
U^{+}[\dT_w]\otimes(U^+\cap\dT_w(U^+)),
\\
\label{eq:re2}
U^{+}\cong&
(U^+\cap\dT_w(U^+))\otimes U^{+}[\dT_w].
\end{align}

We first show \eqref{eq:re1}.
Note that the injectivity of 
\begin{equation}
\label{eq:map}
U^{+}[\dT_w]\otimes(U^+\cap\dT_w(U^+))\to U^+
\end{equation}
is clear from
\[
U^{+}[\dT_w]\otimes(U^+\cap\dT_w(U^+))
\subset T_w(U^{\leqq0})\otimes T_w(U^+)
\]
and $U^{\leqq0}\otimes U^+\cong U$.
Hence we have only to show the surjectivity of \eqref{eq:map}.

For $\Bi=(i_1,\dots, i_m)\in\CI_w$ denote by $U^+[\dT_w;\Bi]$ the subalgebra of $U^+$ generated by $\de_{\Bi,1}$,\dots, $\de_{\Bi,m}$.
By a standard property of $\dT_i$ we have $U^+[\dT_w;\Bi]\subset U^+[\dT_w]$.
Hence we have only to show 
\begin{equation}
\label{eq:surj}
U^{+}[\dT_w;\Bi](U^+\cap\dT_w(U^+))=U^+.
\end{equation}

We note that our assertion is already known for $w=s_i$.
Namely, we have 
\begin{align}
\label{eq:A1}
&U^+\cong U^+[\dT_i]\otimes(U^+\cap \dT_i(U^+)),\qquad
U^+[\dT_i]=\BF[e_i],
\\
\label{eq:A2}
&U^-\cong U^-[\dT_i]\otimes(U^-\cap \dT_i(U^-)),\qquad
U^-[\dT_i]=\BF[f_i]
\end{align}
(see \cite[Chapter 38]{Lbook}).

Now we are going to show \eqref{eq:surj} by induction on $\ell(w)$.
Assume that we have $xs_i>x$ for $x\in W$ and $i\in I$.
By the above argument we need to show \eqref{eq:surj} for $w=xs_i$ assuming \eqref{eq:Ubunkai-d}, \eqref{eq:Ubunkai-di}, \eqref{eq:surj} for $w\in W$ with $\ell(w)\leqq\ell(x)$.
Take $\Bi=(i_1,\dots, i_m)\in\CI_x$, and set
$\Bi'=(i_1,\cdots, i_m,i)\in\CI_{xs_i}$.
To show our assertion $U^{+}[\dT_{xs_i};\Bi'](U^+\cap\dT_{xs_i}(U^+))=U^+$, it is sufficient to show 
\begin{equation}
\label{eq:int}
U^+\cap \dT_x^{-1}(U^+)
=\BF[e_i](U^+\cap \dT_i(U^+)\cap \dT_x^{-1}(U^+)).
\end{equation}
Indeed assuming \eqref{eq:int} we have
\begin{align*}
U^+\cap \dT_x(U^+)
=&
\dT_x(U^+\cap \dT_x^{-1}(U^+))
=\BF[\dT_x(e_i)](U^+\cap \dT_x(U^+)\cap \dT_{xs_i}(U^+))
\\
\subset&
\BF[\dT_x(e_i)]
(U^+\cap \dT_{xs_i}(U^+)),
\end{align*}
and hence
\begin{align*}
U^+=&U^+[\dT_x;\Bi](U^+\cap \dT_x(U^+))
\subset
U^+[\dT_x;\Bi]
\BF[\dT_x(e_i)](U^+\cap \dT_{xs_i}(U^+))
\\
\subset&
U^+[\dT_{xs_i};\Bi'](U^+\cap \dT_{xs_i}(U^+)).
\end{align*}

To verify \eqref{eq:int} we first show the following.
\begin{align}
\label{eq:c1}
&U^+\cap \dT_x^{-1}(U^+)
\\
\nonumber
=&\{u\in U^+\mid
\tau(u, U^-(U^-[\dT_x^{-1}]\cap\Ker(\varepsilon:U^-\to\BF)))=\{0\}\},
\\
\label{eq:c2}
&U^-\cap \dT_x^{-1}(U^-)
\\
\nonumber
=&\{u\in U^-\mid
\tau(U^+(U^+[\dT_x^{-1}]\cap\Ker(\varepsilon:U^+\to\BF)),u)=\{0\}\}.
\end{align}
For simplicity set
\begin{align*}
V^+=&\{u\in U^+\mid
\tau(u, U^-(U^-[\dT_x^{-1}]\cap\Ker(\varepsilon)))=\{0\}\},
\\
V^-=&\{u\in U^-\mid
\tau(U^+(U^+[\dT_x^{-1}]\cap\Ker(\varepsilon)),u)=\{0\}\}.
\end{align*}
By \eqref{eq:Dr8} we have
\[
\tau(U^+\cap \dT_x^{-1}(U^+),U^-[\dT_x^{-1}]\cap\Ker(\varepsilon))=\{0\}.
\]
Hence by Lemma \ref{lem:delta-U} and \eqref{eq:Dr1}
We have
$U^+\cap \dT_x^{-1}(U^+)\subset V^+$.
By a similar argument we have also $U^-\cap \dT_x^{-1}(U^-)\subset V^-$.
On the other hand by the hypothesis of induction we have
$U^\pm\cong U^\pm[\dT_x^{-1}]\otimes(U^\pm\cap \dT_x^{-1}(U^\pm))$, and hence 
$U^\pm=(U^\pm\cap \dT_x^{-1}(U^\pm))\oplus U^\pm(U^\pm[\dT_x^{-1}]\cap\Ker(\varepsilon))$.
Since $\tau$ is non-degenerate, we obtain injective linear maps
$V^\pm\to (U^\mp\cap \dT_x^{-1}(U^\mp))^*$.
Comparing the dimensions of weight spaces we obtain
\begin{align*}
\dim(U_{\pm\gamma}^\pm\cap \dT_x^{-1}(U^\pm))\leqq\dim V_{\pm\gamma}^\pm\leqq\dim (U_{\mp\gamma}^\mp\cap \dT_x^{-1}(U^\mp)).
\end{align*}
for each $\gamma\in Q^+$.
This gives \eqref{eq:c1} and \eqref{eq:c2}.

By a similar argument we have also
\begin{align}
\label{eq:c3}
&U^+\cap \dT_x(U^+)
\\
\nonumber
=&\{u\in U^+\mid
\tau(u, (U^-[\dT_x]\cap\Ker(\varepsilon:U^-\to\BF))U^-)=\{0\}\},
\\
\label{eq:c4}
&U^-\cap \dT_x(U^-)
\\
\nonumber
=&\{u\in U^-\mid
\tau((U^+[\dT_x]\cap\Ker(\varepsilon:U^+\to\BF))U^+,u)=\{0\}\}.
\end{align}

Now let us show \eqref{eq:int}.
Take $u\in U^+\cap \dT_x^{-1}(U^+)$, and decompose it as
\[
u=\sum_ne_i^nu_n\qquad(u_n\in U^+\cap \dT_i(U^+))
\]
(see \eqref{eq:A1}).
Then it is sufficient to show $u_n\in \dT_x^{-1}(U^+)$, which is equivalent to
\begin{equation}
\label{eq:conc}
\tau(u_n,vz)=0
\qquad(v\in U^-,\;z\in U^-[\dT_x^{-1}]\cap\Ker(\varepsilon))
\end{equation}
by \eqref{eq:c1}.
Let $v\in U^-$, $z\in U^-[\dT_x^{-1}]\cap\Ker(\varepsilon)$.
By our assumption we have 
$\tau(u, vz)=0$.
On the other hand we have
\begin{align*}
\tau(u,vz)=&
\sum_n\tau(e_i^nu_n,vz)
=\sum_n\sum_{(z), (v)}
\tau(e_i^n,v_{(0)}z_{(0)})\tau(u_n,v_{(1)}z_{(1)})
\\
=&
\sum_n\sum_{(v)}
\tau(e_i^n,v_{(0)})\tau(u_n,v_{(1)}z)
\end{align*}
by Lemma \ref{lem:delta-U} and \eqref{eq:Dr8}.
Consider the case
\[
v=f_i^{(r)}v'\qquad(v'\in U^-\cap \dT_i(U^-)).
\]
Then we have
\begin{align*}
\tau(u,vz)
=&
\sum_n
\sum_{s=0}^r
\sum_{(v')}
q_i^{-s(r-s)}
\tau(e_i^n,f_i^{(r-s)}v'_{(0)})\tau(u_n,k_i^{-(r-s)}f_i^{(s)}v'_{(1)}z)
\\
=&
\sum_n
\sum_{s=0}^r
q_i^{-s(r-s)}
\tau(e_i^n,f_i^{(r-s)})\tau(u_n,f_i^{(s)}v'z)
\end{align*}
by Lemma \ref{lem:delta-U}, \eqref{eq:Dr7}, \eqref{eq:Dr8}.
By $u_n\in U^+\cap \dT_i(U^+)$ and \eqref{eq:c3} we have
\begin{equation}
\label{eq:bbb}
\tau(u_n,f_iU^-)=\{0\}.
\end{equation}
Hence
\begin{align*}
\tau(u,vz)
=
\sum_n
\tau(e_i^n,f_i^{(r)})\tau(u_n,v'z)
=
\tau(e_i^r,f_i^{(r)})\tau(u_r,v'z).
\end{align*}
Hence by $\tau(e_i^r,f_i^{(r)})\ne0$ we obtain
\begin{equation}
\label{eq:ccc}
\tau(u_n,v'z)=0\qquad(z\in U^-[\dT_x^{-1}]\cap\Ker(\varepsilon), v'\in U^-\cap \dT_i(U^-))
\end{equation}
for any $n$.
By \eqref{eq:bbb}, \eqref{eq:ccc}
\[
\tau(u_n,f_i^rv'z)=0\qquad
(z\in U^-[\dT_x^{-1}]\cap\Ker(\varepsilon), v'\in U^-\cap \dT_i(U^-), r\geqq0).
\]
The proof of  \eqref{eq:re1} is complete.

It remains to show \eqref{eq:re2}.
Similarly to the above  proof of \eqref{eq:re1} it is sufficient to show
\begin{equation}
\label{eq:int2}
U^+\cap \dT_x^{-1}(U^+)
=(U^+\cap \dT_i(U^+)\cap \dT_x^{-1}(U^+))\BF[e_i].
\end{equation}
This follows from \eqref{eq:int} as follows.
We can easily show 
\begin{equation}
\label{eq:xx}
\gamma\in Q^+,\; u\in U^+\cap \dT_i(U^+_\gamma)\;
\Longrightarrow\;
ue_i-q_i^{(\gamma,\alpha_i^\vee)}e_iu\in U^+\cap \dT_i(U^+)
\end{equation}
using \cite[Proposition 38.1.6]{Lbook}.
Hence in view of \eqref{eq:int} it is sufficient to show  that 
for $u\in U^+\cap \dT_i(U_\gamma^+)\cap \dT_x^{-1}(U^+)$ we have
$ue_i-q_i^{(\gamma,\alpha_i^\vee)}e_iu\in \dT_x^{-1}(U^+)$.
This is obvious since $\dT_x(e_i)\in U^+$.
\end{proof}
\begin{remark}
Proposition \ref{prop:decomposition} holds true for various $\BZ[q,q^{-1}]$-forms of $U^\pm$.
To show this it is sufficient to verify \eqref{eq:A1} over $\BZ[q,q^{-1}]$.
In the case of the De Concini-Kac form $U_{\BZ[q,q^{-1}]}^{DK,\pm}$, this follows if we can show that
$(U_{\BZ[q,q^{-1}]}^{DK,+}\cap U^+[\dT_i])(U_{\BZ[q,q^{-1}]}^{DK,+}\cap \dT_i(U^+))$ is stable under the right multiplication of $e_j$ for any $j\in I$.
If $j\ne i$, this is obvious.
If $j=i$, this follows from \eqref{eq:xx}.
The argument for the case of the Lusztig form $U_{\BZ[q,q^{-1}]}^{L,\pm}$ is similar.
Finally, \eqref{eq:A1} over $\BZ[q,q^{-1}]$ for the De Concini-Procesi form
defined by
\begin{align*}
U_{\BZ[q,q^{-1}]}^{DP,+}=&
\{u\in U^+\mid
\tau(u, U_{\BZ[q,q^{-1}]}^{L,-})\in \BZ[q,q^{-1}]\},
\\
U_{\BZ[q,q^{-1}]}^{DP,-}=&
\{u\in U^-\mid
\tau(U_{\BZ[q,q^{-1}]}^{L,+},u)\in \BZ[q,q^{-1}]\},
\end{align*}
is a consequence of that for the Lusztig form by duality.
\end{remark}

By our proof of Proposition \ref{prop:decomposition} we also obtain the following.
\begin{proposition}
\label{prop:base2}
Let $w\in W$ and $\Bi\in\CI_w$.
\begin{itemize}
\item[(i)]
The set 
$
\{\df_\Bi^{\Bn}\mid \Bn\in(\BZ_{\geqq0})^m\}
$
$($resp.\
$\{\de_\Bi^{(\Bn)}\mid \Bn\in(\BZ_{\geqq0})^m\})$
forms an $\BF$-basis of $ U^-[\dT_w]$ $($resp.\ $U^+[\dT_w])$.
\item[(ii)]
The set 
$
\{\tf_\Bi^{(\Bn)}\mid \Bn\in(\BZ_{\geqq0})^m\}
$
$($resp.\
$\{\te_\Bi^{\Bn}\mid \Bn\in(\BZ_{\geqq0})^m\})$
forms an $\BF$-basis of $ U^-[\dT_w^{-1}]$ $($resp.\ $U^+[\dT_w^{-1}])$.
\item[(iii)]
The set 
$
\{\hf_\Bi^{\Bn}\mid \Bn\in(\BZ_{\geqq0})^m\}
$
$($resp.\
$\{\he_\Bi^{(\Bn)}\mid \Bn\in(\BZ_{\geqq0})^m\})$
forms an $\BF$-basis of $ U^-[\hT_w]$ $($resp.\ $U^+[\hT_w])$.
\end{itemize}
\end{proposition}
For $w\in W$ and $\gamma\in Q^+$ we have 
\begin{align}
\label{eq:twist1}
&U^\pm[\dT_w]\cap U^{\pm}_{\pm\gamma}
=U^\pm\cap \dT_w(k_{\pm w^{-1}\gamma}U^\mp_{\pm w^{-1}\gamma}),
\\
\label{eq:twist2}
&U^\pm[\dT_w^{-1}]\cap U^{\pm}_{\pm\gamma}
=U^\pm\cap \dT_w^{-1}(k_{\mp w\gamma}U^\mp_{\pm w\gamma})
\end{align}
by Proposition \ref{prop:base2} and the explicit description of $\dT_i$.

\section{Quantized coordinate algebras}
\subsection{}
We denote by $\BC_q[G]$ the quantized coordinate algebra of $U$ (see, for example, \cite{J}, \cite{Kas}, \cite{T} for the basic facts concerning $\BC_q[G]$).
It is the $\BF$-subspace of $U^*$ spanned by the matrix coefficients of $U$-modules belonging to $\Mod_0(U)$.
Namely, for $V\in\Mod_0(U)$ define a linear map
\begin{equation}
\Phi:
V^*\otimes V\to U^*\quad(v^*\otimes v\mapsto \Phi_{v^*\otimes v})
\end{equation}
by
\[
\langle
\Phi_{v^*\otimes v},u\rangle
=\langle v^*,uv\rangle
\qquad(u\in U).
\]
Then we have 
\begin{equation}
\label{eq:defCq}
\BC_q[G]=\sum_{V\in\Mod_0(U)}\Image(V^*\otimes V\to U^*)\subset U^*.
\end{equation}
It is endowed with a Hopf algebra structure by
$m_{\BC_q[G]}={}^t\Delta_U$,
$\Delta_{\BC_q[G]}={}^tm_U$.
Moreover, \eqref{eq:defCq} induces 
\begin{equation}
\label{eq:decompCq}
\BC_q[G]\cong\bigoplus_{\lambda\in P^-}V^*(\lambda)\otimes V(\lambda).\end{equation}
We set
\[
\BC_q[B^+]=\Image(\BC_q[G]\to (U^{\geqq0})^*),\quad
\BC_q[B^-]=\Image(\BC_q[G]\to (U^{\leqq0})^*),
\]
\[
\BC_q[H]=\Image(\BC_q[G]\to (U^{0})^*).
\]
For $\lambda\in P$ we denote by 
$
\chi_\lambda:U^0\to\BF
$
the algebra homomorphism given by $\chi_\lambda(k_\gamma)=q^{(\lambda,\gamma)}$ for $\gamma\in Q$.
Then we have 
\begin{equation}
\BC_q[H]=\bigoplus_{\lambda\in P}\BF\chi_\lambda.
\end{equation}

Note that $\BC_q[G]$ is a $U$-bimodule by
\[
\langle u_1\varphi u_2,u\rangle
=\langle \varphi,u_2uu_1\rangle
\qquad
(\varphi\in\BC_q[G], u_1, u_1, u\in U).
\]
For $\varphi, \psi\in \BC_q[G]$ and $u\in U$ we have
\begin{align}
\label{eq:mult-left}
u(\varphi\psi)=&
\sum_{(u)}(u_{(0)}\varphi)(u_{(1)}\psi),
\\
\label{eq:mult-right}
(\varphi\psi)u=&
\sum_{(u)}(\varphi u_{(0)})(\psi u_{(1)}).
\end{align}
\begin{example}
\label{ex:sl2}
Consider the case where $\Gg=\Gsl_2$ and $G=SL_2$.
In this case $U=U_q(\Gsl_2)$ is the $\BF$-algebra generated by the elements $k^{\pm1}$, $e$, $f$ satisfying
\[
kk^{-1}=k^{-1}k=1,\quad
kek^{-1}=q^2e,\quad
kfk^{-1}=q^{-2}f,\quad
ef-fe=\frac{k-k^{-1}}{q-q^{-1}}.
\]
Let $V=\BF v_0\oplus\BF v_1$ be the two-dimensional $U$-module 
given by
\[
kv_0=qv_0,\quad
kv_1=q^{-1}v_1,\quad
ev_0=0,\quad
ev_1=v_0,\quad
fv_0=v_1,\quad
fv_1=0,
\]
and define $a, b, c, d\in \BC_q[SL_2]$ by
\[
uv_0=\langle a,u\rangle v_0+\langle c,u\rangle v_1,\qquad
uv_1=\langle b,u\rangle v_0+\langle d,u\rangle v_1
\qquad(u\in U).
\]
Then $\{a, b, c, d\}$ forms a generator system of the $\BF$-algebra $\BC_q[SL_2]$  satisfying the fundamental relations
\begin{align*}
&ab=qba,\quad 
cd=qdc,\quad
ac=qca, \quad
bd=qdb, \quad
bc=cb,\\1
&
ad-da=(q-q^{-1})bc,\qquad
ad-qbc=1.
\end{align*}
Its Hopf algebra structure is given by
\begin{align*}
&\Delta(a)=a\otimes a+b\otimes c,
\qquad
\Delta(b)=a\otimes b+b\otimes d,\\
&\Delta(c)=c\otimes a+d\otimes c,
\qquad
\Delta(d)=c\otimes b+d\otimes d,\\
&\varepsilon(a)=\varepsilon(d)=1,\qquad
\varepsilon(b)=\varepsilon(c)=0,\\
&
S(a)=d,\qquad S(b)=-q^{-1}b,\qquad S(c)=-qc,\qquad S(d)=a.
\end{align*}
\end{example}

\subsection{}
For $w\in W$ and $\lambda\in P^-$ set
\[
v^*_{w\lambda}=v^*_\lambda\dT_{w}^{-1}\in V^*(\lambda)_{w\lambda},
\]
and define 
$\sigma^w_\lambda\in\BC_q[G]$ by
\[
\langle\sigma^w_\lambda, u\rangle=\langle v_{w\lambda}^*,uv_{\lambda}\rangle
\qquad(u\in U).
\]
Then we have
\begin{equation}
\sigma^w_0=1,\quad\sigma^w_\lambda\sigma^w_\mu=\sigma^w_\mu\sigma^w_\lambda=\sigma^w_{\lambda+\mu}
\qquad(\lambda, \mu\in P^-).
\end{equation}
Set
\[
\CS_w=\{\sigma^w_\lambda\mid\lambda\in P^-\}\subset\BC_q[G].
\]

\begin{proposition}[\cite{J}]
\label{prop:Joseph}
The multiplicative subset $\CS_w$ of $\BC_q[G]$ satisfies the left and right Ore conditions.
\end{proposition}
It follows that we have the localization
$\CS_w^{-1}\BC_q[G]=\BC_q[G]\CS_w^{-1}$.
In the rest of this section we investigate the structure of the algebra $\CS_w^{-1}\BC_q[G]$.
In the course of the arguments we give a new proof of Proposition \ref{prop:Joseph}.

\subsection{}
In this subsection we consider the case $w=1$.

Set 
\[
(U^\pm)^\bigstar=\bigoplus_{\gamma\in Q^+}(U^\pm_{\pm\gamma})^*\subset U^*.
\]
\begin{lemma}
\label{lem:2a}
We have
\[
\BC_q[G]\subset(U^+)^\bigstar\otimes\BC_q[H]\otimes(U^-)^\bigstar
\subset
(U^+)^*\otimes(U^0)^*\otimes(U^-)^*
\subset U^*,
\]
where the embedding 
$(U^+)^*\otimes(U^0)^*\otimes(U^-)^*\subset U^*$
is given by 
\begin{multline*}
\langle\psi\otimes\chi\otimes\varphi,xty\rangle
=
\langle\psi,x\rangle
\langle\chi,t\rangle
\langle\varphi,y\rangle
\\
(\psi\in(U^+)^*,\;\chi\in(U^0)^*,\;\varphi\in(U^-)^*,\;x\in U^+,\;t\in U^0,\;y\in U^-).
\end{multline*}
\end{lemma}
\begin{proof}
It is easily seen that for any $\varphi\in\BC_q[G]$ we have
\[
\varphi|_{U^0}\in\BC_q[H],\qquad
\varphi|_{U^\pm}\in (U^\pm)^\bigstar.
\]
Hence the assertion is a consequence of 
\[
\langle\varphi,xty\rangle
=
\sum_{(\varphi)_2}
\langle\varphi_{(0)},x\rangle
\langle\varphi_{(1)},t\rangle
\langle\varphi_{(2)},y\rangle
\qquad(
x\in U^+, t\in U^0, y\in U^-)
\]
for $\varphi\in\BC_q[G]$.
\end{proof}
Note that $U^*$ is an $\BF$-algebra whose multiplication is given by the composite of 
$U^*\otimes U^*\subset(U\otimes U)^*\xrightarrow{{}^t\Delta}U^*$ and that $\BC_q[G]$ is a subalgebra of $U^*$.
We will identify 
$(U^0)^*$, $(U^{\pm})^*$ with subspaces of $U^*$ by
\begin{align*}
(U^+)^*\to U^*
\qquad
(\psi\mapsto[xty\mapsto\langle\psi,x\rangle\varepsilon(t)\varepsilon(y)]),
\\
(U^0)^*\to U^*
\qquad
(\chi\mapsto[xty\mapsto
\varepsilon(x)\langle\chi,t\rangle\varepsilon(y)]),
\\
(U^-)^*\to U^*
\qquad
(\varphi\mapsto[xty\mapsto
\varepsilon(x)\varepsilon(t)\langle\varphi,y\rangle]),
\end{align*}
where
$x\in U^+$, $t\in U^0$, $y\in U^-$.
Under this identification we have
\begin{equation}
\chi_\lambda=\sigma^1_\lambda\qquad(\lambda\in P^-).
\end{equation}
Hence
\begin{equation}
\CS_1=\{\chi_\lambda\mid\lambda\in P^-\}\subset\BC_q[H]\subset (U^0)^*\subset U^*.
\end{equation}

\begin{lemma}
\label{lem:2b}
For 
$\psi\in(U^+)^*,\;\chi\in(U^0)^*,\;\varphi\in(U^-)^*$
we have
\[
\langle\psi\chi\varphi, xty\rangle=
\langle\psi,x\rangle
\langle\chi,t\rangle
\langle\varphi,y\rangle
\qquad
(y\in U^-,\;t\in U^0,\;x\in U^+).
\]
\end{lemma}
\begin{proof}
By the definition of the comultiplication of $U$, for $x\in U^+$, $y\in U^-$ we have 
\begin{gather*}
\Delta(y)=1\otimes y+y',\qquad(\varepsilon\otimes\id)(y')=0,
\\
\Delta(x)=x\otimes 1+x',\qquad(\id\otimes\varepsilon)(x')=0.
\end{gather*}
Hence
\begin{align*}
\langle\psi\chi\varphi, xty\rangle
=&
\sum_{(x)_2,(t)_2,(y)_2}
\langle\psi,x_{(0)}t_{(0)}y_{(0)}\rangle
\langle\chi,x_{(1)}t_{(1)}y_{(1)}\rangle
\langle\varphi,x_{(2)}t_{(2)}y_{(2)}\rangle
\\
=&
\sum_{(t)_2}
\langle\psi,xt_{(0)}\rangle
\langle\chi,t_{(1)}\rangle
\langle\varphi,t_{(2)}y\rangle
=
\langle\psi,x\rangle
\langle\chi,t\rangle
\langle\varphi,y\rangle.
\end{align*}
\end{proof}
\begin{lemma}
\label{lem:2c}
For $\lambda, \mu\in P$
we have
$
\chi_\lambda\chi_\mu
=
\chi_{\lambda+\mu}
$.
\end{lemma}
\begin{proof}
We have
\begin{align*}
\langle\chi_\lambda\chi_\mu, xty\rangle
=&
\sum_{(x),(t),(y)}
\langle\chi_\lambda,x_{(0)}t_{(0)}y_{(0)}\rangle
\langle\chi_\mu,x_{(1)}t_{(1)}y_{(1)}\rangle
\\
=&
\varepsilon(x)\varepsilon(y)
\sum_{(t)}
\langle\chi_\lambda, t_{(0)}\rangle
\langle\chi_\mu,t_{(1)}\rangle
\\
=&
\varepsilon(x)\varepsilon(y)
\langle\chi_{\lambda+\mu},t\rangle
=\langle\chi_{\lambda+\mu}, xty\rangle.
\end{align*}
\end{proof}
\begin{lemma}
\label{lem:2d}
The subspaces $(U^+)^\bigstar$, $(U^-)^\bigstar$ of $U^*$ are subalgebras of $U^*$.
\end{lemma}
\begin{proof}
For $\varphi, \varphi'\in(U^-)^\bigstar$ we have
\begin{align*}
\langle\varphi\varphi',xty\rangle
=&
\sum_{(x),(t),(y)}
\langle\varphi,x_{(0)}t_{(0)}y_{(0)}\rangle
\langle\varphi',x_{(1)}t_{(1)}y_{(1)}\rangle
\\
=&
\varepsilon(x)
\sum_{(t),(y)}
\langle\varphi,t_{(0)}y_{(0)}\rangle
\langle\varphi',t_{(1)}y_{(1)}\rangle
\\
=&
\varepsilon(x)\varepsilon(t)
\sum_{(y)}
\langle\varphi,y_{(0)}\rangle
\langle\varphi',y_{(1)}\rangle
=
\varepsilon(x)\varepsilon(t)
\langle\varphi\varphi',y\rangle.
\end{align*}
The statement for $(U^+)^\bigstar$ is proved similarly.
\end{proof}

\begin{lemma}
\label{lem:2e}
\begin{itemize}
\item[(i)]
For $\psi\in(U_{\gamma}^+)^*$, $\lambda\in P$
we have
$
\chi_\lambda\psi=q^{(\lambda,\gamma)}\psi\chi_\lambda
$.
\item[(ii)]
For $\varphi\in(U_{-\gamma}^-)^*$, $\lambda\in P$
we have
$
\chi_\lambda\varphi=q^{(\lambda,\gamma)}\varphi\chi_\lambda
$.
\end{itemize}
\end{lemma}
\begin{proof}
For $x\in U^+_{\gamma'}$, $y\in U^-$, $t\in U^0$ we have
\begin{align*}
&\langle\chi_\lambda\psi, xty\rangle
=
\sum_{(x),(t),(y)}
\langle\chi_\lambda,x_{(0)}t_{(0)}y_{(0)}\rangle
\langle\psi,x_{(1)}t_{(1)}y_{(1)}\rangle
\\=&
\sum_{(t)}
\langle\chi_\lambda,k_{\gamma'} t_{(0)}\rangle
\langle\psi,xt_{(1)}y\rangle
=\varepsilon(t_{(1)})\varepsilon(y)\delta_{\gamma,\gamma'}
\sum_{(t)}
\langle\chi_\lambda,k_{\gamma'} t_{(0)}\rangle
\langle\psi,x\rangle
\\
=&
\varepsilon(y)\delta_{\gamma,\gamma'}
\langle\chi_\lambda,k_{\gamma} t\rangle
\langle\psi,x\rangle
=q^{(\lambda,\gamma)}\varepsilon(y)
\langle\psi,x\rangle
\langle\chi_\lambda,t\rangle.
\end{align*}
By a similar calculation we have
\[
\langle\psi\chi_\lambda, xty\rangle
=
\varepsilon(y)
\langle\psi,x\rangle
\langle\chi_\lambda,t\rangle.
\]
The statement (i) is proved.
The proof of (ii) is similar.
\end{proof}

\begin{lemma}
\label{lem:2f}
\begin{itemize}
\item[(i)]
Let $\varphi\in(U^-)^\bigstar$.
For sufficiently small $\lambda\in P^-$ we have
$\chi_\lambda\varphi, \varphi\chi_\lambda\in\BC_q[G]$.
\item[(ii)]
Let $\psi\in(U^+)^\bigstar$.
For sufficiently small  $\lambda\in P^-$ we have
$\chi_\lambda\psi, \psi\chi_\lambda\in\BC_q[G]$.
\end{itemize}
\end{lemma}
\begin{proof}
(i) We may assume $\varphi\in(U^-_{-\gamma})^*$.
By Proposition \ref{prop:UV} there exists $v\in V(\lambda)_{\lambda+\gamma}$ such that 
\[
\langle\varphi,y\rangle=\langle v^*_\lambda y,v\rangle\qquad(y\in U^-).
\]
Then 
\[
\langle\Phi_{v_\lambda^*\otimes v},xty\rangle=
\langle v_\lambda^*xty,v\rangle
=\varepsilon(x)
\langle\chi_\lambda,t\rangle
\langle\varphi,y\rangle
=\langle\chi_\lambda\varphi,xty\rangle.
\]
Hence
$\chi_\lambda\varphi=q^{(\lambda,\gamma)}\varphi\chi_\lambda
=
\Phi_{v_\lambda^*\otimes v}\in\BC_q[G]$.
The proof of (ii) is similar.
\end{proof}

\begin{corollary}
\label{cor:x}
Let 
$f\in(U^+)^\bigstar\BC_q[H](U^-)^\bigstar$.
Then we have
$\chi_\lambda f,\; f\chi_\lambda\in\BC_q[G]$
for sufficiently small 
$\lambda\in P^-$ .
\end{corollary}
\begin{proof}
We may assume 
$f=\psi\chi_\nu\varphi$\;($\psi\in (U^+_{\gamma})^*, \nu\in P, \varphi\in (U^-_{-\delta})^*$).
By Lemma \ref{lem:2f} we have 
$\chi_{\lambda_1}\psi, \chi_{\lambda_3}\varphi\in\BC_q[G]$
when $\lambda_1,\lambda_3\in P^-$ are sufficiently small.
Take $\lambda_2\in P^-$ such that $\lambda_2+\nu\in P^-$ and set
$\lambda=\lambda_1+\lambda_2+\lambda_3$.
Then we have
\[
\chi_\lambda f
=
q^{(\lambda_2+\lambda_3,\gamma)}
(\chi_{\lambda_1}\psi)\chi_{\lambda_2+\mu}(\chi_{\lambda_3}\varphi)
\in\BC_q[G].
\]
The proof for $f\chi_\lambda$ is similar.
\end{proof}
\begin{proposition}
\label{prop:2ga}
Let $f\in\BC_q[G]$, $\lambda\in P^-$.
\begin{itemize}
\item[(i)]
If 
$\sigma^1_\lambda f=0$, then $f=0$.
\item[(ii)]
If
$f\sigma^1_\lambda=0$, then $f=0$.
\end{itemize}
\end{proposition}
\begin{proof}
In the algebra $U^*$ the element $\sigma^1_\lambda=\chi_\lambda$ is invertible, and its inverse is given by $\chi_{-\lambda}$.
\end{proof}

We set
\begin{align}
\BC_q[G/N^-]=&
\{f\in\BC_q[G]\mid yf=\varepsilon(y)f\;(y\in U^-)\}
\\
\nonumber
=&\BC_q[G]\cap(U^+)^\bigstar\BC_q[H],
\\
\BC_q[N^+\backslash G]=&
\{f\in\BC_q[G]\mid fx=\varepsilon(x)f\;(x\in U^+)\}
\\
\nonumber
=&\BC_q[G]\cap\BC_q[H](U^-)^\bigstar.
\end{align}
They are subalgebras of $\BC_q[G]$.
\begin{proposition}
\label{prop:2g}
Assume that $\lambda\in P^-$.
\begin{itemize}
\item[(i)]
$\forall\psi\in\BC_q[G/N^-]$
$\exists\mu\in P^-$ s.t. 
$\sigma^1_\mu \psi\in\BC_q[G/N^-]\sigma^1_\lambda$, 
$\psi\sigma^1_\mu\in\sigma^1_\lambda\BC_q[G/N^-]$.
\item[(ii)]
$\forall\varphi\in\BC_q[N^+\backslash G]$
$\exists\mu\in P^-$ s.t. $\sigma^1_\mu \varphi\in\BC_q[N^+\backslash G]\sigma^1_\lambda$,
$\varphi\sigma^1_\mu\in\sigma^1_\lambda\BC_q[G/N^-]$.
\item[(iii)]
$\forall f\in\BC_q[G]$
$\exists\mu\in P^-$ s.t. $\sigma^1_\mu f\in\BC_q[G]\sigma^1_\lambda$, 
$f\sigma^1_\mu\in\sigma^1_\lambda\BC_q[G]$.
\end{itemize}
\end{proposition}
\begin{proof}
(i) By Lemma \ref{lem:2e} we have
\[
\sigma^1_\lambda \psi
\in\sigma^1_\lambda(U^+)^\bigstar\BC_q[H]
\subset(U^+)^\bigstar\BC_q[H]\sigma^1_\lambda.
\]
By Corollary \ref{cor:x} we have 
$\sigma^1_{\lambda+\nu}\psi\in\BC_q[G/N^-]\sigma^1_\lambda$
for some $\nu\in P^-$.
Similarly, we have 
$\psi\sigma^1_{\lambda+\nu'}\in\sigma^1_\lambda\BC_q[G/N^-]$
for some $\nu'\in P^-$.

The statements (ii), (iii) are proved similarly.
\end{proof}
By Proposition \ref{prop:2ga} and Proposition \ref{prop:2g} we have the following.
\begin{corollary}
\label{cor:Ore-1}
The multiplicative set $\CS_1$ satisfies the left and right Ore conditions in all of the three rings $\BC_q[G/N^-]$, $\BC_q[N^+\backslash G]$, $\BC_q[G]$.
\end{corollary}
It follows that we have the localizations
\begin{align}
\CS_1^{-1}\BC_q[G/N^-]=&\BC_q[G/N^-]\CS_1^{-1}
,
\\
\CS_1^{-1}\BC_q[N^+\backslash G]
=&
\BC_q[N^+\backslash G]\CS_1^{-1}
,\\
\CS_1^{-1}\BC_q[G]=&\BC_q[G]\CS_1^{-1}.
\end{align}

The following result is a special case of \cite[Theorem 2.6]{Y}.
\begin{proposition}
\label{prop:localization1}
\begin{itemize}
\item[(i)]
The subset 
$(U^+)^\bigstar\BC_q[H](U^-)^\bigstar$ of $U^*$
is a subalgebra of $U^*$, which is isomorphic to $\CS_1^{-1}\BC_q[G]$.
\item[(ii)]
The subset 
$(U^+)^\bigstar\BC_q[H]$ of $U^*$
is a subalgebra of $U^*$, which is isomorphic to $\CS_1^{-1}\BC_q[G/N^-]$.
\item[(iii)]
The subset 
$\BC_q[H](U^-)^\bigstar$ of $U^*$
is a subalgebra of $U^*$, which is isomorphic to $\CS_1^{-1}\BC_q[N^+\backslash G]$.
\end{itemize}
\end{proposition}
\begin{proof}
(i) Since
$\CS_1$ consists of invertible elements of $U^*$,
we have a canonical homomorphism $\Psi:\CS_1^{-1}\BC_q[G]\to U^*$ of $\BF$-algebras.
Since $\BC_q[G]\to U^*$ is injective, $\Psi$ is injective by Proposition \ref{prop:2ga}.
Hence it is sufficient to show that the image of $\Psi$ coincides with $(U^+)^\bigstar\BC_q[H](U^-)^\bigstar$.
For any $\lambda\in P$ we have
\[
\chi_\lambda\BC_q[G]\subset
\chi_\lambda
(U^+)^\bigstar\BC_q[H](U^+)^\bigstar
=
(U^+)^\bigstar\BC_q[H](U^-)^\bigstar,
\]
and hence $\Image(\Psi)\subset(U^+)^\bigstar\BC_q[H](U^-)^\bigstar$.
Another inclusion
$\Image(\Psi)\supset(U^+)^\bigstar\BC_q[H](U^-)^\bigstar$
is a consequence of Corollary \ref{cor:x}.

The proofs of (ii) and (iii) are similar.
\end{proof}
By Proposition \ref{prop:localization1} we obtain the following results.

\begin{proposition}
The multiplication of $\CS_1^{-1}\BC_q[G]$ induces the isomorphism
\[
\CS_1^{-1}\BC_q[G/N^-]
\otimes_{\BC_q[H]}
\CS_1^{-1}
\BC_q[N^+\backslash G]
\cong
\CS_1^{-1}\BC_q[G].
\]
\end{proposition}
\begin{proposition}
\label{prop:t}
For any $f\in\BC_q[G]$ there exists some $\lambda\in P^-$ such that
\[
\sigma_\lambda^1 f, f\sigma_\lambda^1\in \BC_q[G/N^-]
\BC_q[N^+\backslash G].
\]
\end{proposition}

\subsection{}
In this subsection we investigate the localization of $\BC_q[G]$ with respect to $\CS_w$ for $w\in W$.

As a left (resp.\ right) $U$-module, $\BC_q[G]$ is a sum of submodules belonging to $\Mod_0(U)$ (resp.\ $\Mod_0^r(U)$).
Hence we have a left (resp.\ right) action of $\dT_w$ on $\BC_q[G]$.
\begin{lemma}
\label{lem:Tlr1}
For $w\in W$ we have
\[
\langle\dT_w\varphi,u\rangle
=\langle\varphi\dT_w,\dT_w^{-1}(u)\rangle,\qquad
\langle\varphi\dT_w,u\rangle
=\langle\dT_w\varphi,\dT_w(u)\rangle.
\]
for $\varphi\in\BC_q[G],\;u\in U$.
\end{lemma}
\begin{proof}
We may assume that $\varphi=\Phi_{v^*\otimes v}$.
Then we have
\[
\langle\dT_w\varphi,u\rangle
=\langle v^*,u\dT_wv\rangle
=\langle v^*\dT_w,(\dT_w^{-1}(u))v\rangle
=\langle\varphi\dT_w,\dT_w^{-1}(u)\rangle.
\]
The second formula follows from the first.
\end{proof}
Setting $u=1$ in Lemma \ref{lem:Tlr1} we obtain the following.
\begin{lemma}
\label{lem:Tlr2}
For $w\in W$ we have
\[
\varepsilon(\varphi \dT_{w})
=
\varepsilon(\dT_{w}\varphi)
\qquad
(\varphi\in\BC_q[G]).
\]
\end{lemma}

In the rest of this  section we fix $w\in W$.

\begin{lemma}
\begin{itemize}
\item[(i)] $\sigma^1_\lambda \dT_w^{-1}=\sigma^w_\lambda$ for $\lambda\in P^-$.
\item[(ii)]
$\BC_q[G/N^-]\dT_w^{-1}=\BC_q[G/N^-]$.
\item[(iii)]
$\BC_q[N^+\backslash G]\dT_w^{-1}
=\{\varphi\in\BC_q[G]\mid \varphi u=\varepsilon(u)\varphi\;\;(u\in \dT_w(U^+))\}$.
\end{itemize}
\end{lemma}
\begin{proof}
The statements (i) and (ii) are obvious.
The remaining (iii) is a consequence of 
\[
(f\dT_w^{-1})(\dT_w(u))=(fu)\dT_w^{-1}
\qquad(f\in\BC_q[G], u\in U)
\]
and \eqref{eq:epsT}.
\end{proof}
\begin{lemma}
\label{lem:wa}
\begin{itemize}
\item[(i)]
$
(fh)\dT_w^{-1}=(f\dT_w^{-1})(h\dT_w^{-1})
\quad(h\in\BC_q[N^+\backslash G],\;f\in\BC_q[G])$.
\item[(ii)]
$
(f\sigma_\lambda^1)\dT_w^{-1}=
(f\dT_w^{-1})\sigma_\lambda^w
\quad(f\in\BC_q[G])$.
\item[(iii)]
$
(\sigma_\lambda^1f)\dT_{w^{-1}}\in\BF^\times
\sigma_\lambda^w(f\dT_{w^{-1}})
\quad(f\in\BC_q[G])$.
\end{itemize}
\end{lemma}
\begin{proof}
The statement (i) follows from \eqref{eq:mult-right} and Corollary \ref{cor:dT}.
The statement (ii) is a special case of (i).
Since $V^*(\lambda)_{w\lambda}$ is one-dimensional, we have
$v^*_{w\lambda}\in\BF^\times v^*_\lambda \dot{T}_{w^{-1}}$.
Hence (iii) also follows from 
Corollary \ref{cor:dT}.
\end{proof}
\begin{corollary}
\label{cor:wa}
The linear map
$\BC_q[N^+\backslash G]\ni \varphi\mapsto\varphi \dT_w^{-1}\in\BC_q[G]$ is an algebra homomorphism.
Hence 
$\BC_q[N^+\backslash G]\dT_w^{-1}$ is a subalgebra of $\BC_q[G]$.\end{corollary}
\begin{proposition}
\label{prop:2gw1}
Let $f\in\BC_q[G]$ and $\lambda\in P^-$.
\begin{itemize}
\item[(i)]
If $f\sigma^w_\lambda=0$, then $f=0$.
\item[(ii)]
If $\sigma^w_\lambda f=0$, then $f=0$.
\end{itemize}
\end{proposition}
\begin{proof}
By Lemma \ref{lem:wa} we have
\[
f\sigma^w_\lambda=
(f\dT_w\dT_w^{-1})\sigma^w_\lambda
=
((f\dT_w)\sigma_\lambda^1)\dT_w^{-1},
\]
\[
\sigma^w_\lambda f=
\sigma^w_\lambda(f\dT_{w^{-1}}^{-1}\dT_{w^{-1}})
\in\BF^\times(\sigma^1_\lambda(f\dT_{w^{-1}}^{-1}))\dT_{w^{-1}}.
\]
Hence the assertion follows from Proposition \ref{prop:2g}.
\end{proof}
By Proposition \ref{prop:2g} and Corollary \ref{cor:wa} we have the following.
\begin{proposition}
\label{prop:2gw2}
For any $\varphi\in \BC_q[N^+\backslash G]\dT_w^{-1}$ and $\lambda\in P^-$ there exists some $\mu\in P^-$ such that $\sigma^w_\mu \varphi\in(\BC_q[N^+\backslash G]\dT_w^{-1})\sigma^w_\lambda$ and
$\varphi\sigma^w_\mu \in\sigma^w_\lambda(\BC_q[N^+\backslash G]\dT_w^{-1})$.
\end{proposition}
The following result is proved similarly to \cite[Proposition 3.4]{T}.
\begin{proposition}
\label{prop:2gw3}
For any $\psi\in\BC_q[G/N^-]$ and $\lambda\in P^-$
there exists some $\mu\in P^-$ such that 
$\sigma^w_\mu \psi\in\BC_q[G/N^-]\sigma^w_\lambda$ and
$\psi\sigma^w_\mu \in\sigma^w_\lambda\BC_q[G/N^-]$.
\end{proposition}
\begin{lemma}
\label{lem:loc}
For any $f\in\BC_q[G]$ there exists some
$\lambda\in P^-$ such that $f\sigma_\lambda^w\in
\BC_q[G/N^-](\BC_q[N^+\backslash G]\dT_w^{-1})$.
\end{lemma}
\begin{proof}
By Proposition \ref{prop:t} there exists some $\lambda\in P^-$ such that 
$(f\dT_w)\sigma_\lambda^1\in
\BC_q[G/N^-]\BC_q[N^+\backslash G]$.
Hence by Lemma \ref{lem:wa} we have
\begin{align*}
f\sigma_\lambda^w
=((f\dT_w)\sigma_\lambda^1)\dT_w^{-1}
\in&
(\BC_q[G/N^-]\BC_q[N^+\backslash G])\dT_w^{-1}
\\
=&\BC_q[G/N^-](\BC_q[N^+\backslash G]\dT_w^{-1}).
\end{align*}
\end{proof}

\begin{proposition}
\label{prop:2gw}
For any $f\in\BC_q[G]$ and $\lambda\in P^-$ there exists some $\mu\in P^-$ such that
$\sigma^w_\mu f\in\BC_q[G]\sigma^w_\lambda$ and 
$f\sigma^w_\mu \in\sigma^w_\lambda\BC_q[G]$.
\end{proposition}
\begin{proof}
We can take $\nu\in P^-$ with
$f\sigma_\nu^w\in
\BC_q[G/N^-](\BC_q[N^+\backslash G]\dT_w^{-1})$
by Lemma \ref{lem:loc} .
By 
Proposition \ref{prop:2gw2} and Proposition \ref{prop:2gw3} we have 
$f\sigma_{\nu+\mu'}^w=\sigma_{\lambda}^w\BC_q[G]$
when $\mu'\in P^-$ is sufficiently small.
Similarly we have
$\sigma_{\mu''}^wf\sigma_{\nu}^w=\BC_q[G]\sigma_{\lambda+\nu}^w$.
when $\mu''\in P^-$ is sufficiently small.
Then we have
$\sigma_{\mu''}^wf=\BC_q[G]\sigma_{\lambda}^w$
by Proposition \ref{prop:2gw1}.
\end{proof}
By Proposition \ref{prop:2gw1}, Proposition \ref{prop:2gw2}, Proposition \ref{prop:2gw3}, and Proposition \ref{prop:2gw} we have the following.

\begin{corollary}
\label{cor:Ore-w}
The multiplicative set $\CS_w$ satisfies the left and right Ore conditions in all of the three rings $\BC_q[G/N^-]$, $\BC_q[N^+\backslash G]\dT_w^{-1}$, $\BC_q[G]$.
\end{corollary}
It follows that we have the localizations
\begin{align}
\CS_w^{-1}\BC_q[G/N^-]
=&
\BC_q[G/N^-]\CS_w^{-1},
\\
\CS_w^{-1}(\BC_q[N^+\backslash G]\dT_w^{-1})
=&
(\BC_q[N^+\backslash G]\dT_w^{-1})\CS_w^{-1},
\\
\CS_w^{-1}\BC_q[G]=&\BC_q[G]\CS_w^{-1}.
\end{align}

For $\lambda\in P$ define $\sigma_\lambda^w\in\CS_w^{-1}\BC_q[G]$ by 
\begin{equation}
\sigma_\lambda^w
=
(\sigma_{\lambda_2}^w)^{-1}
\sigma_{\lambda_1}^w
\qquad(\lambda_1, \lambda_2\in P^-,\;\lambda=\lambda_1-\lambda_2),
\end{equation}
and set
\begin{equation}
\tilde{\CS}_w=\{\sigma_\lambda^w\mid\lambda\in P\},
\quad
\BF[\tilde{\CS}_w]
=\bigoplus_{\lambda\in P}\BF\sigma_\lambda^w\subset
\CS_w^{-1}\BC_q[G].
\end{equation}
Note that $\tilde{\CS}_w$ is naturally isomorphic to $P$ as a group.

\begin{proposition}
\label{prop:def-fw}
We can define a bijective linear map
\begin{equation}
F_w:\CS_1^{-1}\BC_q[G]\to\CS_w^{-1}\BC_q[G]
\end{equation}
by
\[
F_w(f(\sigma^1_\lambda)^{-1})
=
(f\dT_w^{-1})(\sigma^w_\lambda)^{-1}
\qquad(\lambda\in P^-, f\in\BC_q[G]).
\]
\end{proposition}
\begin{proof}
Assume 
$
f(\sigma^1_\lambda)^{-1}
=
f'(\sigma^1_\mu)^{-1}$\;
($\lambda, \mu\in P^-,\; f, f'\in\BC_q[G]$).
Then we have 
$f\sigma^1_\mu=f'\sigma^1_\lambda$, and hence 
we have
$(f\dT_w^{-1})\sigma^w_\mu=(f'\dT_w^{-1})\sigma^w_\lambda$
by Lemma \ref{lem:wa}(ii).
It follows that
$(f\dT_w^{-1})(\sigma^w_\lambda)^{-1}=(f'\dT_w^{-1})(\sigma^w_\mu)^{-1}$.
The bijectivity is obvious.
\end{proof}
\begin{lemma}
\label{lem:f-w property}
\begin{itemize}
\item[(i)] We have
\begin{align*}
F_w(\CS_1^{-1}\BC_q[G/N^-])=&
\CS_w^{-1}\BC_q[G/N^-],
\\
F_w(\CS_1^{-1}\BC_q[N^+\backslash G])=&
\CS_w^{-1}(\BC_q[N^+\backslash G]\dT_w^{-1}).
\end{align*}
\item[(ii)]
The linear map
$\CS_1^{-1}\BC_q[N^+\backslash G]
\ni f\mapsto F_w(f)\in
\CS_w^{-1}(\BC_q[N^+\backslash G]\dT_w^{-1})$ is an algebra isomorphism.
\item[(iii)]
For 
$\varphi\in\CS_1^{-1}\BC_q[G/N^-]$ and $\psi\in\CS_1^{-1}\BC_q[N^+\backslash G]$ we  have $F_w(\varphi\psi)=F_w(\varphi)F_w(\psi)$.
\end{itemize}
\end{lemma}
\begin{proof}
The statements (i) and (ii) are obvious.
The statement (iii) is a consequence of Lemma \ref{lem:wa}.
\end{proof}
By the above arguments we obtain the following results.
\begin{proposition}
\label{prop:wisom}
The multiplication induces
\[
\CS_w^{-1}\BC_q[G/N^-]
\otimes_{\BF[\tilde{\CS}_w]}
\CS_w^{-1}(\BC_q[N^+\backslash G]\dT_w^{-1})
\cong
\CS_w^{-1}\BC_q[G].
\]
\end{proposition}
\begin{proposition}
\label{prop:wt}
For any $f\in\BC_q[G]$ there exists some $\lambda\in P^-$ such that
\[
\sigma_\lambda^wf, f\sigma_\lambda^w\in
\BC_q[G/N^-](\BC_q[N^+\backslash G]\dT_w^{-1}).
\]
\end{proposition}
\subsection{}
Set
\begin{align}
\BC_q[N_w^-\backslash G]=
\{\varphi\in\BC_q[G]\mid \varphi y=\varepsilon(y)\varphi\;\;(y\in U^-[\dT_w])\}.
\end{align}
Note that
\[
\sigma^w_\lambda\in\BC_q[G/N^-]\cap\BC_q[N_w^-\backslash G]
\qquad(\lambda\in P^+).
\]
\begin{proposition}
The subspace $\BC_q[N_w^-\backslash G]$ of $\BC_q[G]$ is a subalgebra of $\BC_q[G]$.
\end{proposition}
\begin{proof}
Let $\varphi, \psi\in\BC_q[N_w^-\backslash G]$.
For $y\in U^-[\dT_w]\cap U^-_{-\gamma}$ with $\gamma\in Q^+$
we have
\[
(\varphi\psi)y=
\sum_{(y)}(\varphi y_{(0)})(\psi y_{(1)})
=
(\varphi y)(\psi k_{-\gamma})
=\varepsilon(y)\varphi\psi.
\]
by Lemma \ref{lem:delta-U}.
Hence $\varphi\psi\in\BC_q[N_w^-\backslash G]$.
\end{proof}
\begin{lemma}
\label{lem:equiv0}
Let $\gamma\in Q^+$ and $y\in U^-[\dT_w]\cap U^-_{-\gamma}$.
Then 
for $\varphi\in\BC_q[G]$, $\lambda\in P^-$ we have
\[
(\varphi\sigma_\lambda^w)y
=q^{-(w\lambda,\gamma)}
(\varphi y)\sigma_\lambda^w.
\]
\end{lemma}
\begin{proof}
By Lemma \ref{lem:delta-U} we have
\[
(\varphi\sigma_\lambda^w)y
=
(\varphi y)(\sigma_\lambda^wk_{-\gamma})
=q^{-(w\lambda,\gamma)}
(\varphi y)\sigma_\lambda^w.
\]
\end{proof}
By Lemma \ref{lem:equiv0} we have the following.
\begin{lemma}
\label{lem:equiv}
For $\varphi\in\BC_q[G]$, $\lambda\in P^-$ we have
$
\varphi\in\BC_q[N_w^-\backslash G]\
$ 
if and only if 
$
\varphi\sigma_\lambda^w\in\BC_q[N_w^-\backslash G]
$.
\end{lemma}
\begin{proposition}
\label{prop:leftOre}
The multiplicative set $\CS_w$ satisfies the left Ore condition in $\BC_q[N_w^-\backslash G]$.
\end{proposition}
\begin{proof}
Let $f\in\BC_q[N_w^-\backslash G]$, $\lambda\in P^-$.
Then we can take $f'\in\BC_q[G]$ and  $\mu\in P^-$ satisfying
$\sigma^w_\mu f=f'\sigma^w_\lambda$.
Then by Lemma \ref{lem:equiv} we obtain $f'\in\BC_q[N_w^-\backslash G]$.
\end{proof}
We will show later that $\CS_w$ also satisfies the right Ore condition in $\BC_q[N_w^-\backslash G]$ (see Proposition \ref{prop:rightOre} below).

By Proposition \ref{prop:leftOre} we have the left localizations
\[
\CS_w^{-1}\BC_q[N_w^-\backslash G],\qquad
\CS_w^{-1}(\BC_q[G/N^-]\cap\BC_q[N_w^-\backslash G]).
\]
\begin{proposition}
\label{prop:tensor}
The multiplication of $\CS_w^{-1}\BC_q[N_w^-\backslash G]$ induces the isomorphism
\begin{align*}
&\CS_w^{-1}\BC_q[N_w^-\backslash G]
\\
\cong&
\CS_w^{-1}(\BC_q[G/N^-]\cap\BC_q[N_w^-\backslash G])
\otimes_{\BF[\tilde{\CS}_w]}
\CS_w^{-1}(\BC_q[N^+\backslash G]\dT_w^{-1}).
\end{align*}
\end{proposition}
\begin{proof}
We see easily that
$\BC_q[N^+\backslash G]\dT_w^{-1}\subset\BC_q[N^-_w\backslash G]$.
Let us show that 
$\CS_w^{-1}(\BC_q[N^+\backslash G]\dT_w^{-1})$ is a free left $\BF[\tilde{\CS}_w]$-module.
In the case $w=1$ this is a consequence of Proposition \ref{prop:localization1}.
For general $w$ this follows from the case $w=1$ and Lemma \ref{lem:f-w property}.
Take a basis $\{\psi_j\}_{j\in J}$ of the left free $\BF[\tilde{\CS}_w]$-module 
$\CS_w^{-1}(\BC_q[N^+\backslash G]\dT_w^{-1})$.
We may assume that $\psi_j\in\BC_q[N^+\backslash G]\dT_w^{-1}$.

Let
$f\in\CS_w^{-1}\BC_q[G]$.
By Proposition \ref{prop:wisom}
we can uniquely write 
\[
f=\sum_{j\in J_0}\varphi_j\psi_j\qquad(\varphi_j\in\CS_w^{-1}\BC_q[G/N^-]),
\]
where $J_0$ is a finite subset of $J$.
Then we need to show
\[
f\in
\CS_w^{-1}
\BC_q[N_w^-\backslash G]
\;\Longleftrightarrow\;
\varphi_j\in
\CS_w^{-1}
(\BC_q[G/N^-]\cap\BC_q[N_w^-\backslash G])
\quad(\forall j\in J_0).
\]

Assume that 
$\varphi_j\in
\CS_w^{-1}
(\BC_q[G/N^-]\cap\BC_q[N_w^-\backslash G])
$ for any $j\in J_0$.
We can take $\lambda\in P^-$ such that 
$\sigma^w_\lambda\varphi_j\in\BC_q[G/N^-]\cap\BC_q[N_w^-\backslash G]$ for any $j\in J_0$.
Then  from
\[
\sigma^w_\lambda f
=
\sum_{j\in J_0}(\sigma^w_\lambda\varphi_j)\psi_j
\in\BC_q[N_w^-\backslash G]
\]
we obtain $f\in
\CS_w^{-1}
\BC_q[N_w^-\backslash G]$.

Assume that $f\in
\CS_w^{-1}
\BC_q[N_w^-\backslash G]$.
Taking 
$\lambda\in P^-$
such that
$\sigma^w_\lambda\varphi_j\in\BC_q[G/N^-]$ for any $j\in J_0$, we have
\[
\sigma^w_\lambda f
=
\sum_{j\in J_0}(\sigma^w_\lambda\varphi_j)\psi_j
\qquad
(\sigma_\lambda^w\varphi_j\in\BC_q[G/N^-]).
\]
By 
$f\in
\CS_w^{-1}
\BC_q[N_w^-\backslash G]$
we may assume that 
$\sigma^w_\lambda f\in
\BC_q[N_w^-\backslash G]$.
Then by Lemma \ref{lem:delta-U}
we have
\[
\varepsilon(y)\sigma^w_\lambda f
=
(\sigma^w_\lambda f)y
=
\sum_{j\in J_0}((\sigma^w_\lambda\varphi_j)y)\psi_j
\qquad(y\in U^-[\dT_w]).
\]
By $(\sigma^w_\lambda\varphi_j)y\in\BC_q[G/N^-]$ we have
$(\sigma^w_\lambda\varphi_j)y=\varepsilon(y)(\sigma^w_\lambda\varphi_j)$ for any $j\in J_0$, and hence $\sigma^w_\lambda\varphi_j\in
\BC_q[G/N^-]\cap\BC_q[N_w^-\backslash G]$.
It follows that 
$\varphi_j\in
\CS_w^{-1}
(\BC_q[G/N^-]\cap\BC_q[N_w^-\backslash G])$ for any 
$j\in J_0$.
\end{proof} 

\subsection{}
By Proposition \ref{prop:localization1} we have
\[
\CS_1^{-1}\BC_q[G/N^-]\cong(U^+)^\bigstar\otimes\BC_q[H].
\]
Hence the linear isomorphism
$F_w:\CS_1^{-1}\BC_q[G/N^-]\to\CS_w^{-1}\BC_q[G/N^-]$
induces an isomorphism
\begin{equation}
\label{eq:FS}
\CS_w^{-1}\BC_q[G/N^-]\cong F_w((U^+)^\bigstar)\otimes_\BF\BF[\tilde{\CS}_w]
\qquad
(f\sigma_\lambda^w\leftrightarrow f\otimes \sigma_\lambda^w)
\end{equation}
of vector spaces.

In this subsection we are going to show the following.

\begin{proposition}
\label{prop:XX1}
We have
\begin{align*}
&\CS_w^{-1}(\BC_q[G/N^-]\cap\BC_q[N_w^-\backslash G])
\\
=&
\left\{F_w((U^+)^\bigstar)\cap\CS_w^{-1}\BC_q[N_w^-\backslash G]\right\}
\otimes_\BF
\BF[\tilde{\CS}_w].
\end{align*}
\end{proposition}

Let $\varphi\in(U^+)^\bigstar$.
Then for any sufficiently small $\lambda\in P^-$ there a unique
$v^*\in V^*(\lambda)$ such that
\[
\langle v^*,xv_\lambda\rangle=
\langle\varphi,x\rangle\qquad(x\in U^+)
\]
by Proposition \ref{prop:UV}.
We denote this $v^*$ by $v^*(\varphi,\lambda)$.

\begin{lemma}
\label{lem:Fw}
Let $\varphi\in(U^+)^\bigstar$.
Then  for sufficiently small $\lambda\in P^-$ we have
\[
F_w(\varphi)=\Phi_{v^*(\varphi,\lambda)\dT_w^{-1}\otimes v_\lambda}(\sigma^w_\lambda)^{-1}.
\]
\end{lemma}
\begin{proof}
For 
$x\in U^+, t\in U^0, y\in U^-$ we have
\begin{align*}
\langle
\Phi_{v^*(\varphi,\lambda)\otimes v_\lambda},
xty\rangle
=&
\langle
v^*(\varphi,\lambda), xtyv_\lambda
\rangle
=
\langle
v^*(\varphi,\lambda), xv_\lambda
\rangle
\chi_\lambda(t)\varepsilon(y)
\\
=&\langle\varphi,x\rangle\chi_\lambda(t)\varepsilon(y).
\end{align*}
Hence we obtain
$\Phi_{v^*(\varphi,\lambda)\otimes v_\lambda}=\varphi\chi_\lambda$, 
or equivalently, 
$\varphi=\Phi_{v^*(\varphi,\lambda)\otimes v_\lambda}(\chi_\lambda)^{-1}$.
It follows that
\[
F_w(\varphi)=
\{\Phi_{v^*(\varphi,\lambda)\otimes v_\lambda}\dT_w^{-1}\}(\sigma^w_\lambda)^{-1}
=
\Phi_{v^*(\varphi,\lambda)\dT_w^{-1}\otimes v_\lambda}(\sigma^w_\lambda)^{-1}.
\]
\end{proof}

\begin{corollary}
\label{cor:Fw}
$F_w((U^+)^\bigstar)=
\bigcup_{\lambda\in P^-}
\{
\Phi_{v^*\otimes v_\lambda}(\sigma^w_\lambda)^{-1}
\mid
v^*\in V^*(\lambda)\}.
$
\end{corollary}

\begin{lemma}
\label{lem:Fw2}
For 
$\mu\in P$ we have
$\sigma_\mu^wF_w((U^+)^\bigstar)
=
F_w((U^+)^\bigstar)\sigma_\mu^w$.
\end{lemma}
\begin{proof}
We may assume that $\mu\in P^-$.
For 
$\lambda\in P^-$, $v^*\in V^*(\lambda)$ we have
\[
\sigma_\mu^w\Phi_{v^*\otimes v_\lambda}
=
\Phi_{v_{w\mu}^*\otimes v_\mu}\Phi_{v^*\otimes v_\lambda}
=
\Phi_{(v_{w\mu}^*\otimes v^*)\otimes(v_\mu\otimes v_\lambda)},
\]
\[
\Phi_{v^*\otimes v_\lambda}\sigma_\mu^w
=
\Phi_{v^*\otimes v_\lambda}\Phi_{v_{w\mu}^*\otimes v_\mu}
=
\Phi_{(v^*\otimes v_{w\mu}^*)\otimes(v_\lambda\otimes v_\mu)}.
\]
Let 
\[
p:V^*(\mu)\otimes V^*(\lambda)\to V^*(\lambda+\mu),
\quad
p':V^*(\lambda)\otimes V^*(\mu)\to V^*(\lambda+\mu)
\]
be the homomorphisms of right $U$-modules
such that $p(v^*_\lambda\otimes v^*_\mu)=v^*_{\lambda+\mu}$, 
$p'(v^*_\mu\otimes v^*_\lambda)=v^*_{\lambda+\mu}$
Then by \cite[Lemma 3.5]{T} we have
\[
p(v^*_{w\mu}\otimes V^*(\lambda)_{\lambda+\gamma})=
V^*(\lambda+\mu)_{w\mu+\lambda+\gamma}=
p'(V^*(\lambda)_{\lambda+\gamma}\otimes v^*_{w\mu})
\]
for $\gamma\in Q^+$ if $\lambda\in P^-$ is sufficiently small.
Hence the assertion follows from
Corollary \ref{cor:Fw}.
\end{proof}

\begin{lemma}
\label{lem:Fy}
Let $\gamma, \delta\in Q^+$, and let
$\varphi\in(U^+_\delta)^*$, $y\in U^-[\dT_w]\cap U^-_{-\gamma}$.
Take $z\in U^+_{-w^{-1}\gamma}$ such that 
$
y=\dT_w(k_{-w^{-1}\gamma}z)
$
$($see \eqref{eq:twist1}$)$, and 
define $\varphi^z\in(U^+)^\bigstar$ by
\[
\langle \varphi^z,x\rangle
=
\langle \varphi,zx\rangle\qquad(x\in U^+).
\]
If $\lambda\in P^-$ is sufficiently small, then we have
$F_w(\varphi)\sigma^w_\lambda\in\BC_q[G/N^-]$, and 
\[
(F_w(\varphi)\sigma^w_\lambda)y
=q^{-(w^{-1}\gamma,\lambda+\delta)}
F_w(\varphi^z)\sigma^w_\lambda.
\]
\end{lemma}
\begin{proof}
If $\lambda\in P^-$ is sufficiently small, then  we have $F_w(\varphi)\sigma^w_\lambda=
\Phi_{v^*(\varphi,\lambda)\dT_w^{-1}\otimes v_\lambda}\in\BC_q[G/N^-]$.
For 
$x\in U^+$ we have
\[
\langle
v^*(\varphi,\lambda)z,xv_\lambda\rangle
=
\langle v^*(\varphi,\lambda),zxv_\lambda\rangle
=
\langle \varphi,zx\rangle
=
\langle \varphi^z,x\rangle
\]
and hence 
$
v^*(\varphi,\lambda)z=v^*(\varphi^z,\lambda)$.
It follows that
\begin{align*}
&(F_w(\varphi)\sigma^w_\lambda)y
=
\Phi_{v^*(\varphi,\lambda)\dT_w^{-1}\otimes v_\lambda}y=
\Phi_{v^*(\varphi,\lambda)k_{-w^{-1}\gamma}z\dT_w^{-1}\otimes v_\lambda}
\\
=&
q^{-(w^{-1}\gamma,\lambda+\delta)}
\Phi_{v^*(\varphi^z,\lambda)\dT_w^{-1}\otimes v_\lambda}
=
q^{-(w^{-1}\gamma,\lambda+\delta)}
F_w(\varphi^z)\sigma^w_\lambda.
\end{align*}
\end{proof}
Let us give a proof of  Proposition \ref{prop:XX1}.
By \eqref{eq:FS} any $f\in \CS_w^{-1}\BC_q[G/N^-]$ is uniquely written as
\[
f=\sum_{\lambda\in P}F_w(\varphi_\lambda)\sigma^w_\lambda
\in
\CS_w^{-1}\BC_q[G/N^-]
\qquad(\varphi_\lambda\in(U^+)^\bigstar).
\]
We need to show that 
$f\in\CS_w^{-1}(\BC_q[G/N^-]\cap\BC_q[N_w^-\backslash G])$
if and only if 
$F_w(\varphi_\lambda)\in\CS_w^{-1}\BC_q[N_w^-\backslash G]$
for any $\lambda\in P$.
By Lemma \ref{lem:equiv} we have
\begin{align*}
&f\in
\CS_w^{-1}(\BC_q[G/N^-]\cap\BC_q[N_w^-\backslash G])
\\
\Longleftrightarrow\;&
\exists\nu\in P^-\;\text{s.t.}\;
\sigma_\nu^wf\in\BC_q[G/N^-]\cap\BC_q[N_w^-\backslash G]
\\
\Longleftrightarrow\;&
\exists\nu\in P^-\;\text{s.t.}\;
\sigma_\nu^wf\in\BC_q[G/N^-],\;
\sigma_\nu^wf\sigma_\mu^w\in\BC_q[G/N^-]\cap\BC_q[N_w^-\backslash G]
\\
\Longleftrightarrow\;&
f\sigma_\mu^w\in
\CS_w^{-1}(\BC_q[G/N^-]\cap\BC_q[N_w^-\backslash G])
\end{align*}
for any $\mu\in P^-$.
Hence we may assume from the beginning that $f$ is written as
\[
f=\sum_{\lambda\in P^-}F_w(\varphi_\lambda)\sigma^w_\lambda
\qquad(\varphi_\lambda\in(U^+)^\bigstar).
\]

If $F_w(\varphi_\lambda)\in\CS_w^{-1}\BC_q[N_w^-\backslash G]$
for any $\lambda\in P^-$, there exists some $\mu\in P^-$ such that 
$\sigma_\mu^wF_w(\varphi_\lambda)\in\BC_q[G/N^-]\cap\BC_q[N^-_w\backslash G]$ for any $\lambda\in P^-$.
It follows that 
\[
\sigma_\mu^wf
=
\sum_{\lambda\in P^-}
(\sigma_\mu^wF_w(\varphi_\lambda))\sigma^w_\lambda
\in\BC_q[G/N^-]\cap\BC_q[N^-_w\backslash G]
\]
by Lemma \ref{lem:equiv}, and hence $f\in\CS_w^{-1}(\BC_q[G/N^-]\cap\BC_q[N_w^-\backslash G])$.

It remains to show that if 
$f\in\CS_w^{-1}(\BC_q[G/N^-]\cap\BC_q[N_w^-\backslash G])$, then 
$F_w(\varphi_\lambda)\in\CS_w^{-1}\BC_q[N_w^-\backslash G]$
for any $\lambda\in P^-$.
So assume that $f\in\CS_w^{-1}(\BC_q[G/N^-]\cap\BC_q[N_w^-\backslash G])$.
Take $\mu\in P^-$ which is sufficiently small.
Then 
we have
$
\sigma_\mu^wf\in
\BC_q[G/N^-]\cap\BC_q[N_w^-\backslash G]
$.
By
Lemma \ref{lem:Fw2} we can write
\[
\sigma_\mu^wF_w(\varphi_\lambda)=
F_w(\varphi'_\lambda)\sigma_\mu^w
\in\BC_q[G/N^-]
\qquad(\lambda\in P^-, \varphi'_\lambda\in(U^+)^\bigstar),
\]
and hence 
\[
\sigma_\mu^wf
=
\sum_{\lambda\in P^-}
F_w(\varphi'_\lambda)\sigma^w_{\mu+\lambda},\qquad
F_w(\varphi'_\lambda)\sigma^w_{\mu+\lambda}\in\BC_q[G/N^-]\quad(\lambda\in P^-).
\]
Let $\gamma\in Q^+\setminus\{0\}$ and
$y\in U^-[\dT_w]\cap U^-_{-\gamma}$.
By $\sigma_\mu^wf\in\BC_q[N_w^-\backslash G]$ we have
$(\sigma_\mu^wf)y=0$.
On the other hand we have 
\[
(\sigma_\mu^wf)y=
\sum_{\lambda\in P^-}(F_w(\varphi'_\lambda)\sigma^w_{\lambda+\mu})y
=
\sum_{\lambda\in P^-}q^{-(w\lambda,\gamma)}
((F_w(\varphi'_\lambda)\sigma^w_{\mu})y)
\sigma^w_{\lambda}.
\]
By Lemma \ref{lem:Fy} we have
\[
(F_w(\varphi'_\lambda)\sigma^w_{\mu})y
=F_w(\varphi''_\lambda)\sigma^w_{\mu}
\]
for some $\varphi''_\lambda\in (U^+)^\bigstar$, and hence
\[
\sum_{\lambda\in P^-}q^{-(w\lambda,\gamma)}
F_w(\varphi''_\lambda)\sigma^w_{\lambda+\mu}
=0
\]
By \eqref{eq:FS} we obtain 
$F_w(\varphi''_\lambda)=0$
for any $\lambda \in P^-$.
It follows that
\[
(\sigma^w_\mu F_w(\varphi_\lambda))y=
(F_w(\varphi'_\lambda)\sigma_\mu^w)y
=
F_w(\varphi''_\lambda)\sigma^w_{\mu}
=0.
\]
We obtain $F_w(\varphi_\lambda)\in\CS_w^{-1}\BC_q[N_w^-\backslash G]$
for any $\lambda\in P^-$.
The proof of  Proposition \ref{prop:XX1} is complete.
\subsection{}
Set
\begin{align*}
\CJ_w=&\{\psi\in(U^+)^\bigstar\mid
\psi^z=\varepsilon(z)\psi\quad
(z\in U^+[\dT_w^{-1}])\}
\\
=&
\{
\psi\in(U^+)^\bigstar\mid
\psi|_{\Ker(\varepsilon: U^+[\dT_w^{-1}]\to\BF)U^+}=0\}
\end{align*}
In this subsection we are going to show the following.
\begin{proposition}
\label{prop:intersec}
$F_w((U^+)^\bigstar)\cap\CS_w^{-1}\BC_q[N_w^-\backslash G]
=
F_w(\CJ_w)$.
\end{proposition}
We first show the following result.
\begin{lemma}
\label{lem:commm}
Let $\gamma\in Q^+$, $\psi\in\CJ_w\cap (U^+_\gamma)^*$ and $\mu\in P$.
Then we have
$q^{(\mu,\gamma)}F_w(\psi)\sigma_\mu^w=\sigma_\mu^wF_w(\psi)$.
\end{lemma}
\begin{proof}
We may assume $\mu\in P^-$.
When $\lambda\in P^-$ is sufficiently small, we have
$F_w(\psi)=\Phi_{v^*(\psi,\lambda)\dT_w^{-1}\otimes v_\lambda}(\sigma^w_\lambda)^{-1}$, and hence it is sufficient to show 
\[
q^{(\mu,\gamma)}\Phi_{v^*(\psi,\lambda)\dT_w^{-1}\otimes v_\lambda}\sigma_\mu^w
=
\sigma_\mu^w\Phi_{v^*(\psi,\lambda)\dT_w^{-1}\otimes v_\lambda}.
\]
We have
\begin{align*}
\Phi_{v^*(\psi,\lambda)\dT_w^{-1}\otimes v_\lambda}\sigma_\mu^w
=&
\Phi_{v^*(\psi,\lambda)\dT_w^{-1}\otimes v_\lambda}
\Phi_{v^*_\mu\dT_w^{-1}\otimes v_\mu}
=
\Phi_{
(v^*(\psi,\lambda)\dT_w^{-1}\otimes v^*_\mu\dT_w^{-1})
\otimes
(v_{\lambda}\otimes v_\mu)},
\\
\sigma_\mu^w\Phi_{v^*(\psi,\lambda)\dT_w^{-1}\otimes v_\lambda}
=&
\Phi_{v^*_\mu\dT_w^{-1}\otimes v_\mu}
\Phi_{v^*(\psi,\lambda)\dT_w^{-1}\otimes v_\lambda}
=
\Phi_{
(v^*_\mu\dT_w^{-1}\otimes v^*(\psi,\lambda)\dT_w^{-1})
\otimes
(v_{\mu}\otimes v_\lambda)}.
\end{align*}
Since $v^*_\mu$ is the lowest weight vector we have
\[
v^*(\psi,\lambda)\dT_w^{-1}\otimes v^*_\mu\dT_w^{-1}
=
(v^*(\psi,\lambda)\otimes v^*_\mu)\dT_w^{-1}.
\]
On the other hand by $\psi\in\CJ_w$ we have
\[
v^*(\psi,\lambda)z=\varepsilon(z)v^*(\psi,\lambda)
\qquad(z\in  U^+[\dT_w^{-1}]),
\]
and hence
\[
v^*_\mu\dT_w^{-1}
\otimes
v^*(\psi,\lambda)\dT_w^{-1}
=
(v^*_\mu\otimes v^*(\psi,\lambda))\dT_w^{-1}.
\]
Therefore, we have only to show
\[
q^{(\mu,\gamma)}\Phi_{
(v^*(\psi,\lambda)\otimes v^*_\mu)\dT_w^{-1}
\otimes
(v_{\lambda}\otimes v_\mu)}
=
\Phi_{
(v^*_\mu\otimes v^*(\psi,\lambda))\dT_w^{-1})
\otimes
(v_{\mu}\otimes v_\lambda)}.
\]
Let 
\[
p:V^*(\lambda)\otimes V^*(\mu)\to V^*(\lambda+\mu),\quad
p':V^*(\mu)\otimes V^*(\lambda)\to V^*(\lambda+\mu)
\]
be the homomorphisms of $U$-modules such that 
$p(v^*_\lambda\otimes v^*_\mu)=v^*_{\lambda+\mu}$ and 
$p'(v^*_\mu\otimes v^*_\lambda)=v^*_{\lambda+\mu}$.
The our assertion is equivalent to
\[
q^{(\mu,\gamma)}p(v^*(\psi,\lambda)\otimes v^*_\mu)
=q^{(\mu,\gamma)}v^*(\psi,\lambda+\mu)
=p'(v^*_\mu\otimes v^*(\psi,\lambda)).
\]
This follows from
\begin{align*}
\langle(v^*(\psi,\lambda)\otimes v^*_\mu)x,v_\lambda\otimes v_\mu\rangle
=&
\langle v^*(\psi,\lambda)x\otimes v^*_\mu,v_\lambda\otimes v_\mu\rangle
=\langle\psi,x\rangle,
\\
\langle(v^*_\mu\otimes v^*(\psi,\lambda))x,v_\mu\otimes v_\lambda\rangle
=&
\langle v^*_\mu k_\gamma\otimes v^*(\psi,\lambda)x,v_\mu\otimes v_\lambda\rangle
=q^{(\mu,\gamma)}\langle\psi,x\rangle
\end{align*}
for $x\in U^+$.
\end{proof}

Let us give a proof of Proposition \ref{prop:intersec}.
Assume that $\varphi\in(U^+)^\bigstar$ satisfies $F_w(\varphi)\in\CS_w^{-1}\BC_q[N_w^-\backslash G]$.
When $\mu\in P^-$ is sufficiently small, we have
$\sigma_\mu^wF_w(\varphi)\in\BC_q[N_w^-\backslash G]$.
By Lemma \ref{lem:Fw2} we have
\begin{equation}
\label{eq:commm}
\sigma_\mu^wF_w(\varphi)=F_w(\varphi')\sigma_\mu^w.
\end{equation}
By Lemma \ref{lem:Fy} we have
\[
(\varphi')^z=\varepsilon(z)\varphi'\qquad
(z\in  U^+[\dT_w^{-1}]),
\]
namely
$\varphi'\in\CJ_w$.
Hence \eqref{eq:commm} and Lemma \ref{lem:commm} implies $\varphi\in\CJ_w$.

Assume conversely that $\varphi\in\CJ_w$.
Then by Lemma \ref{lem:Fy} and Lemma \ref{lem:commm} we have
$F_w(\varphi)\in\CS_w^{-1}\BC_q[N_w^-\backslash G]$.

The proof of Proposition \ref{prop:intersec} is complete.

\subsection{}

By Proposition \ref{prop:tensor}, Corollary \ref{cor:Ore-w}, Proposition \ref{prop:XX1}, Proposition \ref{prop:intersec}, and Lemma \ref{lem:commm}
we obtain the following.
\begin{proposition}
\label{prop:rightOre}
The multiplicative set $\CS_w$ satisfies the right Ore condition in $\BC_q[N_w^-\backslash G]$.
\end{proposition}

Set
\begin{equation}
U^+[\dT_w^{-1}]^\bigstar
=\sum_{\gamma\in Q^+}
(U^+[\dT_w^{-1}]\cap U^+_\gamma)^*\subset (U^+[\dT_w^{-1}])^*.
\end{equation}
In view of \eqref{eq:Ubunkai-di} we can define
an injective linear map
\begin{equation}
i_w^+: U^+[\dT_w^{-1}]^\bigstar\to(U^+)^\bigstar
\end{equation}
by
\[
\langle i_w^+(\varphi),x_1x_2\rangle=
\langle \varphi, u_1\rangle\varepsilon(u_2)
\quad(x_1\in  U^+[\dT_w^{-1}],\;
x_2\in U^+\cap \dT_w^{-1}(U^+)).
\]
\begin{proposition}
\label{prop:Fw}
\begin{itemize}
\item[(i)]
The multiplication of 
$(U^+)^\bigstar$ $($as a subalgebra of $U^*$$)$ induces an isomorphism
\[
 i_w^+(U^+[\dT_w^{-1}]^\bigstar)\otimes\CJ_w\cong(U^+)^\bigstar
\]
of vector spaces.
\item[(ii)]
For $\varphi\in U^+[\dT_w^{-1}]^\bigstar$, $\psi\in\CJ_w$ we have
\[
F_w(i_w^+(\varphi)\psi)=F_w(i_w^+(\varphi))F_w(\psi)\qquad
(\varphi\in U^+[\dT_w^{-1}]^\bigstar,\; \psi\in\CJ_w).
\]
\end{itemize}
\end{proposition}
\begin{proof}
(i) 
For 
$\varphi\in U^+[\dT_w^{-1}]^\bigstar$, $\psi\in\CJ_w$, $x\in  U^+[\dT_w^{-1}]$, $x'\in U^+\cap \dT_w^{-1}U^{\geqq0}$ we have
\[
\langle i_w^+(\varphi)\psi,xx'\rangle
=\sum_{(x),(x')}
\langle i_w^+(\varphi),x_{(0)}x'_{(0)}\rangle
\langle \psi,x_{(1)}x'_{(1)}\rangle.
\]
Hence by Lemma \ref{lem:delta-U} we obtain
\[
\langle i_w^+(\varphi)\psi,xx'\rangle
=
\langle i_w^+(\varphi),x\rangle
\langle \psi,x'\rangle.
\]

(ii)
Take $\lambda\in P^-$ such that
$i_w^+(\varphi)\chi_\lambda\in\BC_q[G/N^-]$.
Then we have
$\chi_\lambda^{-1}\psi\chi_\lambda=\psi'\in\CJ_w$.
Take $\mu\in P^-$ such that 
$\psi'\chi_\mu\in\BC_q[G/N^-]$.
We may assume that 
$\psi'\chi_\mu=\Phi_{v^*\otimes v_\nu}$ and
\[
v^*z=\varepsilon(z)v^*\qquad(z\in  U^+[\dT_w^{-1}]).
\]
Then we have
\begin{align*}
&F_w(i_w^+(\varphi)\psi)
=
F_w((i_w^+(\varphi)\chi_\lambda)(\psi'\chi_\mu)\chi_{\lambda+\mu}^{-1})
\\
=&
\{\{
(i_w^+(\varphi)\chi_\lambda)(\psi'\chi_\mu)\}\dT_w^{-1}\}(\sigma^w_{\lambda+\mu})^{-1}
\\
=&
\{(i_w^+(\varphi)\chi_\lambda)\dT_w^{-1}\}\{(\psi'\chi_\mu)\dT_w^{-1}\}(\sigma^w_{\lambda+\mu})^{-1}
\\
=&
\{F_w(i_w^+(\varphi))\sigma^w_{\lambda}\}
\{F_w(\psi')\sigma^w_{\mu}\}
(\sigma^w_{\lambda+\mu})^{-1}
=
F_w(i_w^+(\varphi))F_w(\psi').
\end{align*}
Here, the last equality is a consequence of Lemma \ref{lem:commm}.
\end{proof}

\section{Induced modules}
\subsection{}

We fix $w\in W$ in this section.

By Proposition \ref{prop:base2}, Corollary \ref{cor:dT} and
\eqref{eq:mult-right} we have
\begin{equation}
\label{eq:actT}
(\varphi\psi)\dT_w=(\varphi\dT_w)(\psi\dT_w)
\qquad(\varphi\in\BC_q[G], \psi\in\BC_q[N^-_w\backslash G]).
\end{equation}
Define
$\eta'_{w}:\BC_q[N_w^-\backslash G]\to\BC_q[H]$
by
\[
\langle\eta'_{w}(\varphi),t\rangle=\langle\varphi\dT_{w},t\rangle
(=\langle\dT_{w}\varphi,\dT_w(t)\rangle)
\qquad(t\in U^0)
\]
(see Lemma \ref{lem:Tlr1}).
\begin{lemma}
The linear map $\eta'_{w}$ is an algebra homomorphism.
Moreover, for $\lambda\in P^-$ we have
$\eta'_{w}(\sigma^w_\lambda)=\chi_{\lambda}\in\BC_q[H]^\times$.\end{lemma}
\begin{proof}
For 
$\varphi, \psi\in\BC_q[N_w^-\backslash G], t\in U^0$ we have 
\begin{align*}
\langle\eta'_{w}(\varphi\psi),t\rangle
=&
\langle(\varphi\psi)\dT_w,t\rangle
=
\langle(\varphi\dT_w)(\psi\dT_w),t\rangle
\\
=&
\sum_{(t)}
\langle\varphi\dT_w,t_{(0)}\rangle
\langle\psi\dT_w,t_{(1)}\rangle
=
\langle\eta'_{w}(\varphi)\eta'_{w}(\psi),t\rangle.
\end{align*}
by \eqref{eq:actT}.
For $\lambda\in P^-$ and $t\in U^0$ we have
\[
\langle\eta'_{w}(\sigma^w_\lambda),t\rangle
=
\langle \sigma^w_\lambda\dT_{w},t\rangle
=
\langle v_{\lambda}^*,tv_{\lambda}\rangle
=\langle\chi_{\lambda},t\rangle.
\]
\end{proof}
Hence we obtain an algebra homomorphism
\begin{equation}
\eta_w:\CS_w^{-1}\BC_q[N_w^-\backslash G]\to\BC_q[H]
\end{equation}
by extending $\eta'_{w}$.

\begin{definition}
Define an
$(\CS_w^{-1}\BC_q[G],\BC_q[H])$-bimodule $\CM_w$ by
\begin{equation}
\CM_w=\CS_w^{-1}\BC_q[G]\otimes_{\CS_w^{-1}\BC_q[N_w^-\backslash G]}\BC_q[H],
\end{equation}
where $\CS_w^{-1}\BC_q[N_w^-\backslash G]\to\BC_q[H]$ is given by $\eta_w$.
\end{definition}
By
\[
\CS_w^{-1}\BC_q[G]=
\BC_q[G]\otimes_{\BC_q[N_w^-\backslash G]}\CS_w^{-1}\BC_q[N_w^-\backslash G]
\]
we have
\begin{equation}
\label{eq:des}
\CM_w
\cong
\BC_q[G]\otimes_{\BC_q[N_w^-\backslash G]}\BC_q[H].
\end{equation}

For $\varphi\in\CS_w^{-1}\BC_q[G]$ and $\chi\in\BC_q[H]$ we write
\begin{equation}
\varphi\star\chi
:=
\varphi\otimes\chi\in\CM_w.
\end{equation}
Then we have
\begin{equation}
\varphi\sigma^w_\lambda\star\chi
=
\varphi\star\chi_{\lambda}\chi
\qquad(\varphi\in\CS_w^{-1}\BC_q[G], \lambda\in P, \chi\in\BC_q[H]).
\end{equation}
By \eqref{eq:des} 
$\CM_w$ is generated by 
$\{\varphi\star1\mid\varphi\in\BC_q[G]\}$
as a $\BC_q[H]$-module.

\subsection{}
Set
\[
U^{\geqq0}[\dT_w^{-1}]
=(U^+[\dT_w^{-1}])U^0\subset U^{\geqq0}.
\]
Define an injective linear map
\[
U^+[\dT_w^{-1}]^\bigstar\otimes\BC_q[H]
\to
\Hom_\BF(U^{\geqq0}[\dT_w^{-1}],\BF)
\qquad(f\otimes\chi\mapsto c_{f\otimes\chi})
\]
by
\[
\langle c_{f\otimes\chi},xt\rangle
=
\langle f,x\rangle\langle\chi,t\rangle
\qquad(x\in U^+[\dT_w^{-1}], \;t\in U^0),
\]
and denote its image by $U^{\geqq0}[\dT_w^{-1}]^\bigstar$.
Then we have an identification 
\begin{equation}
\label{eq:identification}
U^{\geqq0}[\dT_w^{-1}]^\bigstar\cong U^+[\dT_w^{-1}]^\bigstar\otimes\BC_q[H]
\qquad
(c_{f\otimes\chi}\leftrightarrow f\otimes\chi)
\end{equation}
of vector spaces.
Since 
$ U^+[\dT_w^{-1}]^\bigstar\otimes\BC_q[H]$ is naturally a right $\BC_q[H]$-module by the multiplication of $\BC_q[H]$, 
$U^{\geqq0}[\dT_w^{-1}]^\bigstar$ is also endowed with a right  $\BC_q[H]$-module structure via the identification \eqref{eq:identification}.
Then we have
\begin{multline}
\langle f\chi,xt\rangle
=\sum_{(t)}\langle f,xt_{(0)}\rangle\langle\chi,t_{(1)}\rangle
\\
(f\in U^{\geqq0}[\dT_w^{-1}]^\bigstar,\;\chi\in\BC_q[H],\;x\in U^+[\dT_w^{-1}], \;t\in U^0).
\end{multline}

\subsection{}
We construct an isomorphism
\[
\Theta_w:\CM_w\to
U^{\geqq0}[\dT_w^{-1}]^\bigstar
\]
of right $\BC_q[H]$-modules.
We first define 
$\Theta_w':\BC_q[G]\to U^{\geqq0}[\dT_w^{-1}]^\bigstar$
by 
\[
\langle \Theta_w'(\varphi),u\rangle
=\langle \varphi\dT_w,u\rangle
\qquad(\varphi\in\BC_q[G],\; u\in U^{\geqq0}[\dT_w^{-1}]).
\]

\begin{lemma}
\label{lem-theta-prime}
$\Theta_w'(\varphi\psi)=\Theta_w'(\varphi)\eta'_{w}(\psi)
\qquad
(\varphi\in\BC_q[G],\; \psi\in\BC_q[N_w^-\backslash G])$.
\end{lemma}
\begin{proof}
Let $\gamma\in Q^+$, 
$x\in  U^+[\dT_w^{-1}]\cap U^+_\gamma$, $t\in U^{0}$.
Then we have
\[
\langle \Theta_w'(\varphi),xt\rangle
=
\langle \varphi\dT_w,xt\rangle
=
\langle \varphi\dT_w(x)\dT_w,t\rangle.
\]
Similarly,
\[
\langle \Theta_w'(\varphi\psi),xt\rangle
=
\langle \{(\varphi\psi)\dT_w(x)\}\dT_w,t\rangle.
\]
By \eqref{eq:twist2} we can write 
$\dT_w(x)=yk_{-w\gamma}\quad(y\in U^-[\dT_w]\cap U_{w\gamma}^-)$.
Thus by Lemma \ref{lem:delta-U} and $\psi\in\BC_q[N_w^-\backslash G]$ we have 
\begin{align*}
&(\varphi\psi)(\dT_w(x))=
\left(\sum_{(y)}(\varphi y_{(0)})(\psi y_{(1)})\right)k_{-w\gamma}
=
\left\{(\varphi y)(\psi k_{w\gamma})\right\}k_{-w\gamma}
\\
=&
(\varphi(\dT_w(x)))\psi.
\end{align*}
Hence by \eqref{eq:actT} 
we have
\begin{align*}
&\langle \Theta_w'(\varphi\psi),xt\rangle
=
\langle \{(\varphi(\dT_w(x)))\psi\}\dT_w,t\rangle
=
\langle (\varphi(\dT_w(x))\dT_w)(\psi\dT_w),t\rangle
\\
=&\sum_{(t)}
\langle \varphi(\dT_w(x))\dT_w,t_{(0)}\rangle
\langle\psi\dT_w,t_{(1)}\rangle
=\sum_{(t)}
\langle \Theta_w'(\varphi),xt_{(0)}\rangle
\langle \eta'_{w}(\psi),t_{(1)}\rangle
\\
=&
\langle \Theta_w'(\varphi)\eta'_{w}(\psi),xt\rangle.
\end{align*}
\end{proof}

Hence regarding $U^{\geqq0}[\dT_w^{-1}]^\bigstar$
as a right $\BC_q[N_w^-\backslash G]$-module via 
$\eta'_{w}:\BC_q[N_w^-\backslash G]\to\BC_q[H]$, 
$\Theta_w'$ turns out to be a homomorphism of right $\BC_q[N_w^-\backslash G]$-modules.
Moreover, the right action of the elements of $\CS_w(\subset\BC_q[N_w^-\backslash G])$ on $U^{\geqq0}[\dT_w^{-1}]^\bigstar$ is invertible.
Hence $\Theta_w'$ induces
\[
\Theta_w'':
\CS_w^{-1}\BC_q[G]=
\BC_q[G]\otimes_{\BC_q[N_w^-\backslash G]}\CS_w^{-1}\BC_q[N_w^-\backslash G]
\to
U^{\geqq0}[\dT_w^{-1}]^\bigstar.
\]
Then we have
\begin{equation}
\Theta_w''(\varphi\psi)=\Theta_w''(\varphi)\eta_{w}(\psi)
\quad
(\varphi\in\CS_w^{-1}\BC_q[G],\; \psi\in\CS_w^{-1}\BC_q[N_w^-\backslash G]).
\end{equation}
Therefore, we obtain a homomorphism
\begin{equation}
\Theta_w:
\CM_w
\to
U^{\geqq0}[\dT_w^{-1}]^\bigstar
\end{equation}
of right $\BC_q[H]$-modules by
\[
\Theta_w(\varphi\star\chi)=\Theta_w''(\varphi)\chi
\qquad(\varphi\in\CS_w^{-1}\BC_q[G], 
\chi\in\BC_q[H]).
\]

\begin{proposition}
\label{prop:Upsilon}
The linear map
\[
\Upsilon_w: U^+[\dT_w^{-1}]^\bigstar
\otimes\BC_q[H]
\to\CM_w
\qquad(
\varphi\otimes\chi\mapsto
{F_w(i_w^+(\varphi))}\star\chi)
\]
is bijective.
\end{proposition}
\begin{proof}
By 
Proposition \ref{prop:wisom}
and
\eqref{eq:FS}
we have
\[
\CS_w^{-1}\BC_q[G]
\cong
F_w((U^+)^\bigstar)\otimes
\CS_w^{-1}(\BC_q[N^+\backslash G]\dT_w^{-1}).
\]
On the other hand by 
Proposition \ref{prop:tensor},
Proposition \ref{prop:XX1}, 
Proposition \ref{prop:intersec}
we have
\[
\CS_w^{-1}\BC_q[N_w^-\backslash G]
\cong
F_w(\CJ_w)\otimes
\CS_w^{-1}(\BC_q[N^+\backslash G]\dT_w^{-1}).
\]
Hence by Proposition \ref{prop:Fw}
we have
\[
\CS_w^{-1}\BC_q[G]
\cong
F_w(i_w^+(U^+[\dT_w^{-1}]^\bigstar))
\otimes\CS_w^{-1}\BC_q[N_w^-\backslash G].
\]
It follows that
$
\CM_w\cong
U^+[\dT_w^{-1}]^\bigstar\otimes\BC_q[H].
$\end{proof}

\begin{proposition}
\label{prop:bijection}
We have $\Theta_w\circ\Upsilon_w=\id$ under the identification \eqref{eq:identification}.
Especially, 
$\Theta_w$ is an isomorphism of right $\BC_q[H]$-modules.
\end{proposition}
\begin{proof}
Let $\varphi\in U^+[\dT_w^{-1}]^\bigstar$, 
$\chi\in\BC_q[H]$, $x\in  U^+[\dT_w^{-1}]$, $t\in U^0$.
Then for $\lambda\in P^-$ which is sufficiently small we have
\begin{align*}
&\langle (\Theta_w\circ \Upsilon_w)(\varphi\otimes\chi),xt\rangle
=
\langle \Theta_w(F_w(i_w^+(\varphi))\star\chi),xt\rangle
\\
=&
\langle \Theta_w(\Phi_{v^*(i_w^+(\varphi),\lambda)\dT_w^{-1}\otimes v_\lambda}(\sigma^w_\lambda)^{-1}\star\chi),xt\rangle
\\
=&
\langle  \Theta_w(\Phi_{v^*(i_w^+(\varphi),\lambda)\dT_w^{-1}\otimes v_\lambda}\star\chi_{-\lambda}\chi),xt\rangle
\\
=&
\langle
\Theta_w'(\Phi_{v^*(i_w^+(\varphi),\lambda)\dT_w^{-1}\otimes v_\lambda})(\chi_{-\lambda}\chi),xt\rangle
\\
=&
\sum_{(t)}
\langle
\Phi_{v^*(i_w^+(\varphi),\lambda)\otimes v_\lambda},xt_{(0)}
\rangle
\langle
\chi_{-\lambda}\chi,t_{(1)}
\rangle
\\
=&
\sum_{(t)_2}
\langle i_w^+(\varphi),x\rangle
\langle
\chi_{\lambda},t_{(0)}
\rangle
\langle
\chi_{-\lambda},t_{(1)}
\rangle
\langle
\chi,t_{(2)}
\rangle
=
\langle \varphi,x\rangle
\langle\chi,t\rangle
\\
=&
\langle \varphi\otimes\chi,x\otimes t\rangle.
\end{align*}
\end{proof}
\subsection{}
\label{subsec:sl2-2}
In this subsection we consider the special case where $\Gg=\Gsl_2$ and $G=SL_2$.
We follow the notation of Example \ref{ex:sl2}.
The Weyl group consists of two elements $1$ and $s$.
We give below an explicit description of the $(\BC_q[G],\BC_q[H])$-bimodule $\CM_s$.
For $n\in\BZ_{\geqq0}$ define $m(n)\in\CM_s$ by
\[
\langle\Theta(m(n)),e^{n'}k^{i}\rangle
=
\delta_{n,n'}
\qquad
(n'\in\BZ_{\geqq0}, i\in\BZ).
\]
Then we have 
\[
\CM_s=\bigoplus_{n=0}^\infty
\BC_q[H]m(n).
\]
\begin{lemma}
\label{lem:sl2-2}
The action of $\BC_q[G]$ on $\CM_s$ is given by 
\begin{align*}
am(n)=&
\chi(q-q^{-1})q^{n-1}
m(n-1),
\qquad
bm(n)=
\chi^{-1} q^{n}
m(n),
\\
cm(n)=&
-\chi q^{n+1}
m(n),
\qquad
dm(n)=
-\chi^{-1} q[n+1]
m(n+1).
\end{align*}
\end{lemma}
\begin{proof}
By Corollary \ref{cor:dT} we have
\begin{align*}
&\langle (\psi \varphi)\dT_s,e^{n}k^{i}\rangle
\\
=&
\sum_{r=0}^n
\sum_{p=0}^\infty
q^{-2ri}
c(p,n,r)
\langle \psi\dT_s f^{(p)}
,k^{n-r+i}e^r\rangle
\langle 
\varphi\dT_s,e^{n+p-r}k^{i}\rangle
\end{align*}
for $\varphi, \psi\in\BC_q[G]$,
where
\[
c(p,n,r)
=
q^{p(p-1)/2-r(n-r)}(q-q^{-1})^p
\begin{bmatrix}n\\r\end{bmatrix}.
\]
By a direct calculation we have
\begin{align*}
a\dT_s=c,\quad
a\dT_sf=a, \quad
a\dT_sf^{(2)}=0.
\end{align*}
Hence for $\varphi\in\BC_q[G]$ we have
\begin{align*}
&\langle (a\varphi)\dT_s,e^nk^{i}\rangle
\\
=&
\sum_{r=0}^n
\sum_{p=0}^\infty
q^{-2ri}
c(p,n,r)
\langle a\dT_s f^{(p)}
,k^{n-r+i}e^r\rangle
\langle 
\varphi\dT_s,e^{n+p-r}k^{i}\rangle
\\
=&
\sum_{r=0}^n
q^{-2ri}
c(0,n,r)
\langle c
,k^{n-r+i}e^r\rangle
\langle 
\varphi\dT_s,e^{n-r}k^{i}\rangle
\\
&\qquad+
\sum_{r=0}^n
q^{-2ri}
c(1,n,r)
\langle a
,k^{n-r+i}e^r\rangle
\langle 
\varphi\dT_s,e^{n+1-r}k^{i}\rangle
\\
=&
c(1,n,0)q^{n+i}
\langle \varphi\dT_s,e^{n+1}k^{i}\rangle
\\
=&
q^{i}
(q-q^{-1})q^{n}
\langle \varphi\dT_s,e^{n+1}k^{i}\rangle.
\end{align*}
Taking $\varphi_j\in\BC_q[SL_2]$ such that $m(n)=\sum_j\varphi_j\star\chi^j$ we have
\begin{align*}
&\langle\Theta(am(n)),e^{n'}k^{i}\rangle
=
\sum_j
\langle\Theta(a\varphi_j\star\chi_j),e^{n'}k^{i}\rangle
\\
=&
\sum_j
\langle (a\varphi_j)\dT_s,e^{n'}k^{i}\rangle q^{ij}
=
q^{i}(q-q^{-1})q^{n'}
\sum_j
\langle \varphi_j\dT_s,e^{n'+1}k^{i}\rangle q^{ij}
\\
=&
(q-q^{-1})q^{n'}
\sum_j
\langle\Theta(\varphi_j\star\chi^{j+1}),e^{n'+1}k^{i}\rangle
\\
=&
(q-q^{-1})q^{n'}
\langle\Theta(m(n))\chi,e^{n'+1}k^{i}\rangle
=
(q-q^{-1})q^{n'}
\delta_{n,n'+1}\langle\chi,k^i\rangle
\\
=&(q-q^{-1})q^{n-1}
\langle\Theta(m(n-1))\chi,e^{n'}k^{i}\rangle.
\end{align*}
Hence $am(n+1)=
\chi
(q-q^{-1})q^{n}
m(n)$.
The proof of other formulas are similar.
\end{proof}

\subsection{}
Let us return to the general situation where $\Gg$ is any simple Lie algebra.
\begin{proposition}
\label{prop:MMM}
We have
\begin{align*}
\CM_w
\cong&
\CS_w^{-1}\BC_q[G/N^-]\otimes_{
\CS_w^{-1}(\BC_q[G/N^-]\cap\BC_q[N_w^-\backslash G])}
\BC_q[H]
\\
\cong&
\BC_q[G/N^-]\otimes_{
\BC_q[G/N^-]\cap\BC_q[N_w^-\backslash G]}
\BC_q[H].
\end{align*}
\end{proposition}
\begin{proof}
By 
\eqref{eq:FS}, 
Proposition \ref{prop:XX1}, 
Proposition \ref{prop:intersec}, 
Proposition \ref{prop:Fw}
we have
\[
\CS_w^{-1}\BC_q[G/N^-]
\cong
F_w(i_w^+(U^+[\dT_w^{-1}]^\bigstar))
\otimes
\CS_w^{-1}(\BC_q[G/N^-]\cap\BC_q[N_w^-\backslash G]).
\]
Hence we obtain
\begin{align*}
&\CS_w^{-1}\BC_q[G/N^-]\otimes_{
\CS_w^{-1}(\BC_q[G/N^-]\cap\BC_q[N_w^-\backslash G])}
\BC_q[H]
\\
\cong&
U^+[\dT_w^{-1}]^\bigstar\otimes\BC_q[H]
\cong
\CM_w
\end{align*}
by Proposition \ref{prop:Upsilon}.
The second isomorphism is a consequence of 
\begin{align*}
&\CS_w^{-1}\BC_q[G/N^-]
\\
\cong&
\BC_q[G/N^-]\otimes_{\BC_q[G/N^-]\cap\BC_q[N_w^-\backslash G]}
\CS_w^{-1}(\BC_q[G/N^-]\cap\BC_q[N_w^-\backslash G]).
\end{align*}
\end{proof}

We regard $\BC_q[H]$ as a subalgebra of $\BC_q[B^-]$ via the Hopf algebra homomorphism
$U^{\leqq0}\to U^0$ given by $ty\mapsto\varepsilon(y)t$ $(t\in U^0, y\in U^-)$.
Define an action of $W$ on $\BC_q[H]$ by
\[
\langle w\chi,t\rangle=
\langle\chi,\dT_w^{-1}(t)\rangle
\qquad(w\in W, \chi\in \BC_q[H], t\in U^0).
\]
For $w\in W$ 
we define a twisted right $\BC_q[H]$-module structure of 
$\BC_q[B^-]$ by
\begin{equation}
\label{eq:twist-action}
\varphi\bullet_w\chi=(Sw\chi)\varphi
\qquad(\varphi\in\BC_q[B^-],\;\chi\in\BC_q[H]).
\end{equation}
We denote by $\BC_q[B^-]^{\bullet_w}$ the $\BF$-algebra $\BC_q[B^-]$ equipped with the twisted right $\BC_q[H]$-module structure \eqref{eq:twist-action}.

We are going to construct an embedding
\[
\Xi_w:\CM_w\hookrightarrow\BC_q[B^-]^{\bullet_w}.
\]
of right $\BC_q[H]$-module.

We first define 
\[
\Xi_w':\BC_q[G/N^-]\to\BC_q[B^-]
\]
by
\[
\langle \Xi_w'(\varphi),u\rangle=
\langle
\dT_w\varphi,Su
\rangle
\qquad(\varphi\in\BC_q[G/N^-], u\in U^{\leqq0}).
\]
\begin{lemma}
\label{lem:Xi1}
The linear map
$\Xi_w'$ is an algebra anti-homomorphism. Moreover, for $\lambda\in P^-$ we have $\Xi_w'(\sigma^w_\lambda)=\chi_{-w\lambda}\in\BC_q[B^-]^\times$.
\end{lemma}
\begin{proof}
For 
$\varphi, \psi\in \BC_q[G/N^-]$, $u\in U^{\leqq0}$ we have
\begin{align*}
&\langle
\Xi_w'(\varphi\psi),u
\rangle
=
\langle
\dT_w(\varphi\psi),Su
\rangle
=
\langle
(\dT_w\varphi)(\dT_w\psi),Su
\rangle
\\
=&\sum_{(u)}
\langle
\dT_w\varphi,Su_{(1)}
\rangle
\langle
\dT_w\psi,Su_{(0)}
\rangle
=
\sum_{(u)}
\langle
\Xi_w'(\varphi),u_{(1)}
\rangle
\langle
\Xi_w'(\psi),u_{(0)}
\rangle
\\
=&\langle
\Xi_w'(\psi)\Xi_w'(\varphi),u
\rangle.
\end{align*}
Here, the second equality is a consequence of Corollary \ref{cor:dT}.
For $\lambda\in P^-$, $t\in U^0$, $y\in U^{\leqq0}$ we have
\begin{align*}
&\langle \Xi_w'(\sigma^w_\lambda),ty\rangle
=
\langle \dT_w\sigma^w_\lambda,(Sy)(St)\rangle
=
\langle v^*_\lambda\dT_w^{-1},((Sy)(St))\dT_wv_\lambda\rangle
\\
=&
\langle v^*_\lambda,\{\dT_w^{-1}(Sy)\}\{\dT_w^{-1}(St)\}v_\lambda\rangle
=
\varepsilon(\dT_w^{-1}(Sy))
\langle\chi_\lambda,\dT_w^{-1}(St)\rangle
\\
=&
\varepsilon(y)\langle\chi_{-w\lambda},t\rangle
=\langle\chi_{-w\lambda},ty\rangle.
\end{align*}
\end{proof}
Hence $\Xi_w'$ induces an algebra anti-homomorphism
\begin{equation}
\Xi_w'':\CS_w^{-1}\BC_q[G/N^-]\to\BC_q[B^-].
\end{equation}

\begin{lemma}
\label{lem:Xi2}
For 
$\varphi\in\CS_w^{-1}(\BC_q[G/N^-]\cap\BC_q[N_w^-\backslash G])$ we have
$\Xi_w''(\varphi)=Sw(\eta_w(\varphi))$.
\end{lemma}
\begin{proof}
We may assume 
$\varphi\in\BC_q[G/N^-]\cap\BC_q[N_w^-\backslash G]$.
By \eqref{eq:Ubunkai-d} the multiplication of $U^{\leqq0}$ induces an isomorphism 
\[
U^{\leqq0}\cong
S^{-1}(U^-\cap\dT_w(U^-))\otimes U^0\otimes S^{-1}(U^-[\dT_w])
\]
of vector spaces.
Let 
$y_1\in U^-[\dT_w]$, $y_2\in U^-\cap\dT_w(U^-)$, $t\in U^0$.
Then we have
\begin{align*}
&\langle\Xi_w'(\varphi),(S^{-1}y_2)t(S^{-1}y_1)\rangle
=
\langle\dT_w\varphi,y_1(St)y_2\rangle
=
\langle y_2\dT_w\varphi y_1,St\rangle
\\
=&
\langle\dT_w(\dT_w^{-1}(y_2))\varphi y_1,St\rangle
=\varepsilon(y_1)\varepsilon(y_2)
\langle\dT_w\varphi,St\rangle
\\
=&
\varepsilon(y_1)\varepsilon(y_2)
\langle\varphi\dT_w,\dT_w^{-1}S(t)\rangle
=
\langle\eta_w(\varphi),\dT_w^{-1}(St)\rangle
\varepsilon(y_1)\varepsilon(y_2)
\\
=&
\langle Sw(\eta_w(\varphi)),(S^{-1}y_2)t(S^{-1}y_1)\rangle.
\end{align*}
\end{proof}

By Proposition \ref{prop:MMM}, Lemma \ref{lem:Xi1} and Lemma \ref{lem:Xi2} 
we obtain a homomorphism 
\begin{equation}
\label{eq:Xi}
\Xi_w:\CM_w\to\BC_q[B^-]^{\bullet_w}
\end{equation}
of right $\BC_q[H]$-modules given by
\[
\Xi_w(\varphi\star\chi)
=
\Xi_w''(\varphi)\bullet_w\chi
\qquad
(\varphi\in\CS_w^{-1}\BC_q[G/N^-], \chi\in\BC_q[H]).
\]
Since $\Xi_w''$ is an algebra anti-homomorphism, we have
\begin{equation}
\Xi_w(\varphi\psi\star1)=\Xi_w(\psi\star1)\Xi_w(\varphi\star1)\qquad
(\varphi, \psi\in\CS_w^{-1}\BC_q[G/N^-]).
\end{equation}

\subsection{}
By \eqref{eq:Ubunkai-h} and Lemma \ref{lem:transfer} 
we can define an injective linear map
\begin{equation}
\Omega_w:U^{\geqq0}[\dT_w^{-1}]^\bigstar\to\BC_q[B^-]^{\bullet_w}
\end{equation}
by
\begin{multline*}
\langle\Omega_w(f),ty_2y_1\rangle
=\varepsilon(y_2)\langle f,\dT_w^{-1}S(ty_1)\rangle
\\
\qquad(f\in U^{\geqq0}[\dT_w^{-1}]^\bigstar,\; 
y_1\in U^-[\hT_w],\;
y_2\in U^-\cap\hT_wU^-,\;
t\in U^0).
\end{multline*}
\begin{lemma}
\label{lem:jmath}
The linear map $\Omega_w$ is a homomorphism of right $\BC_q[H]$-modules.
\end{lemma}
\begin{proof}
Let $f\in U^{\geqq0}[\dT_w^{-1}]^\bigstar$ and $\chi\in\BC_q[H]$.
For $y_1\in U^-[\hT_w]$, \;$y_2\in U^-\cap\hT_wU^-$,\;
$t\in U^0$ we have
\begin{align*}
&
\langle
\Omega_w(f)\bullet_w\chi,
ty_2y_1
\rangle
=
\langle
(Sw\chi)\Omega_w(f),
ty_2y_1
\rangle
\\
=&
\sum_{(y_1), (y_2), (t)}
\langle
Sw\chi,
t_{(0)}y_{2(0)}y_{1(0)}
\rangle
\langle
\Omega_w(f),
t_{(1)}y_{2(1)}y_{1(1)}
\rangle
\\
=&
\sum_{(t)}
\langle
Sw\chi,
t_{(0)}
\rangle
\langle
\Omega_w(f),
t_{(1)}y_{2}y_{1}
\rangle
\\
=&
\sum_{(t)}
\varepsilon(y_2)
\langle
Sw\chi,
t_{(0)}
\rangle
\langle
f,
\dT_w^{-1}S(t_{(1)}y_{1})
\rangle
\\
=&
\sum_{(t)}
\varepsilon(y_2)
\langle
\chi,
\dT_w^{-1}S(t_{(0)})
\rangle
\langle
f,
\{\dT_w^{-1}S(y_{1})\}
\{\dT_w^{-1}S(t_{(1)})\}
\rangle
\\
=&
\varepsilon(y_2)
\langle
f\chi,
\{\dT_w^{-1}S(y_{1})\}
\{\dT_w^{-1}S(t)\}
\rangle
\\
=&
\varepsilon(y_2)
\langle
f\chi,
\dT_w^{-1}S(ty_{1})
\rangle
\\
=&
\langle
\Omega_w(f\chi),ty_{2}ty_{1}
\rangle.
\end{align*}
\end{proof}
\begin{lemma}
\label{lem:XiTheta}
$\Omega_w\circ\Theta_w=\Xi_w$.
\end{lemma}
\begin{proof}
By Proposition \ref{prop:MMM} we have only to show 
\[
(\Omega_w\circ\Theta_w)(\varphi\star\chi)=\Xi_w(\varphi\star\chi)
\qquad
(\varphi\in\BC_q[G/N^-],\;\chi\in\BC_q[H]).
\]
By the definitions of 
$\Theta_w$, $\Xi_w$ and Lemma \ref{lem:jmath}
it is sufficient to show
\[
(\Omega_w\circ\Theta_w')(\varphi)=\Xi_w'(\varphi)
\qquad
(\varphi\in\BC_q[G/N^-]).
\]
Let
$y_1\in U^-[\hT_w]$,\;
$y_2\in U^-_{-\gamma}\cap\hT_wU^-$,\;
$\delta\in Q$.
The we have
\begin{align*}
&\langle
(\Omega_w\circ\Theta_w')(\varphi),
k_\delta y_2y_1
\rangle
=
\varepsilon(y_2)\langle \Theta_w'(\varphi),\dT_w^{-1}S(k_\delta y_1)\rangle
\\
=&
\varepsilon(y_2)\langle \varphi\dT_w,\dT_w^{-1}S(k_\delta y_1)\rangle
=
\varepsilon(y_2)\langle \dT_w\varphi,
S(k_\delta y_1)\rangle,
\end{align*}
\begin{align*}
&
\langle
\Xi_w'(\varphi),
k_\delta y_2y_1
\rangle
=
\langle
\dT_w\varphi,
(Sy_1)(Sy_2)k_{-\delta}
\rangle
\\
=&
q^{-(\gamma,\delta)}
\langle
\dT_w\varphi,
(Sy_1)k_{-\delta}(Sy_2)
\rangle
=
q^{-(\gamma,\delta)}
\langle
(Sy_2)\dT_w\varphi,
(Sy_1)k_{-\delta}
\rangle
\\
=&
q^{-(\gamma,\delta)}
\langle
\dT_w(\dT_w^{-1}S(y_2))\varphi,
S(k_\delta y_1)
\rangle
\\
=&
q^{-(\gamma,\delta)}
\langle
\dT_w(S\hT_w^{-1}(y_2))\varphi,
S(k_\delta y_1)
\rangle
=
\varepsilon(y_2)
\langle
\dT_w\varphi,
S(k_\delta y_1)
\rangle.
\end{align*}
\end{proof}
We define a $\BC_q[H]$-submodule $\CA_w$ of $\BC_q[B^-]^{\bullet_w}$
by
\[
\CA_w=
\{\varphi\in\BC_q[B^-]\mid \varphi y=0\quad(y\in U^-\cap\hT_w(U^-))\}.
\]

Note that we have an isomorphism
\begin{equation}
\label{eq:isomd}
(U^-)^\bigstar\otimes\BC_q[H]\cong\BC_q[B^-]^{\bullet_w}
\qquad(f\otimes\chi\leftrightarrow d_{f\otimes\chi})
\end{equation}
of right $\BC_q[H]$-modules 
given by
\[
\langle d_{f\otimes\chi},ty\rangle
=
\langle Sw\chi,t\rangle
\langle f,y\rangle
\qquad
(t\in U^0, y\in U^-)).
\]
Set
\begin{equation}
U^-[\hT_w]^\bigstar
=
\bigoplus_{\gamma\in Q^+}
(U^-[\hT_w]\cap U^-_{-\gamma})^*
\subset (U^-[\hT_w])^*.
\end{equation}
By \eqref{eq:Ubunkai-h} we have an injective linear map
\begin{equation}
i_w^-:U^-[\hT_w]^\bigstar\to (U^-)^\bigstar
\end{equation}
given by
\[
\langle i_w^-(f),y'y\rangle
=\varepsilon(y')\langle f,y\rangle
\qquad(y\in U^-[\hT_w], y'\in U^-\cap\hT_w(U^-)
\]
Under the identification \eqref{eq:isomd} we have
\[
i_w^-(U^-[\hT_w]^\bigstar)\otimes\BC_q[H]\cong\CA_w.
\]
\begin{proposition}
The linear map 
$\Xi_w$ is injective and its image coincides with $\CA_w$.
\end{proposition}
\begin{proof}
Note that 
$\Theta_w$ is bijective and $\Omega_w$ is injective.
Hence by Lemma \ref{lem:XiTheta} we see that 
$\Xi_w$ is injective and its image coincides with $\Image(\Omega_w)$.
Moreover, by the definition of $\Omega_w$ the image of $\Omega_w$ coincides with $\CA_w$.
\end{proof}

\section{The decomposition into tensor product}
\subsection{}
For $i\in I$ define a Hopf subalgebra $U(i)$ of $U$ by
\[
U(i)
=\langle k_i^{\pm1},e_i,f_i\rangle\cong \BF\otimes_{\BQ(q_i)}U_{q_i}(\Gsl_2)\subset U.
\]
Define subalgebras $U(i)^\flat$ ($\flat=0, \pm, \geqq0, \leqq0$) by
\begin{gather*}
U(i)^0=\langle k_i^{\pm1}\rangle,
\qquad
U(i)^+=\langle e_i\rangle,
\qquad
U(i)^-=\langle f_i\rangle,
\\
U(i)^{\geqq0}=\langle k_i^{\pm1}, e_i\rangle,
\qquad
U(i)^{\leqq0}=\langle k_i^{\pm1}, f_i\rangle.
\end{gather*}

We denote the quantized coordinate algebra of $U(i)$ by 
$\BC_{q}[G(i)]\;(\cong\BF\otimes_{\BQ(q_i)}\BC_{q_i}[SL_2])$.
As an algebra it is generated by elements $a_i, b_i, c_i, d_i$
satisfying the fundamental relations 
\begin{gather*}
a_ib_i=q_ib_ia_i,\quad 
c_id_i=q_id_ic_i,\quad
a_ic_i=q_ic_ia_i, \quad
b_id_i=q_id_ib_i, \\
b_ic_i=c_ib_i,\quad
a_id_i-d_ia_i=(q_i-q_i^{-1})b_ic_i,\quad
a_id_i-q_ib_ic_i=1
\end{gather*}
(see Example \ref{ex:sl2}).

We have a quotient Hopf algebra
$\BC_{q}[H(i)]$ of $\BC_q[G(i)]$ corresponding to
$U(i)^0$.
Then we have
\[
\BC_{q}[H(i)]
=\BF[\chi_i^{\pm1}],\qquad
\chi_i(k_i)=q_i.
\]

Denote by 
\begin{equation}
r^G_{G(i)}:\BC_q[G]\to\BC_{q}[G(i)]
\end{equation}
the Hopf algebra homomorphism corresponding to $U(i)\subset U$.

\subsection{}
Consider the $(\BC_q[G(i)],\BC_q[H(i)])$-bimodule
\begin{equation}
\CM_i=\BF\otimes_{\BQ(q_i)}\CM^{SL_2}_{s_i},
\end{equation}
where $s_i$ is the generator of the Weyl group of $U(i)$, 
and $\CM^{SL_2}_{s_i}$ is the $\BC_q[G]$-module $\CM_w$ for $G=SL_2$, $q=q_i$,   $w=s_i$.
We have an isomorphism
\[
\Theta_i:\CM_i\to(U(i)^{\geqq0})^\bigstar
\]
of right $\BC_q[H(i)]$-modules given by
\begin{multline*}
\langle\Theta_i({\varphi}\star\chi),xt\rangle
=\sum_{(t)}\langle\varphi\dT_i,xt_{(0)}\rangle
\langle\chi,t_{(1)}\rangle
\\
(\varphi\in\BC_{q}[G(i)], \chi\in\BC_q[H(i)], x\in U(i)^+, t\in U(i)^0).
\end{multline*}
Define $p_i(n)\in\CM_i$ by
\[
\langle\Theta_i(p_i(n)),e_i^{n'}k_i^j\rangle
=\delta_{nn'}(-1)^{n}q_{i}^{n}[n]_{q_i}!.
\]
By Lemma \ref{lem:sl2-2} we obtain the following.
\begin{proposition}
The set $\{p_i(n)\mid n\in\BZ_{\geqq0}\}$ forms a basis of the $\BC_q[H(i)]$-module $\CM_i$.
Moreover, we have
\begin{align*}
a_ip_i(n)=&(1-q_i^{2n})\chi_ip_i(n-1),
\qquad
b_ip_i(n)=\chi_i^{-1}q_i^{n}p_i(n),
\\
c_ip_i(n)=&-\chi_iq_i^{n+1}p_i(n),
\qquad
d_ip_i(n)=\chi_i^{-1}p_i(n+1).
\end{align*}
\end{proposition}
We will regard $\CM_i$ as a $(\BC_q[G],\BC_q[H(i)])$-module via  $r^G_{G(i)}$.

\subsection{}
Let $w\in W$ with $\ell(w)=m$.
We set
\begin{equation}
\label{eq:zr}
z_{\Bi,r}=s_{i_{r+1}}s_{i_{r+2}}\cdots s_{i_{m}}
\qquad
(r=0,\cdots, m),
\end{equation}
\begin{equation}
\BC_q[H(\Bi)]=
\BC_q[H(i_1)]\otimes\cdots\otimes\BC_q[H(i_t)].
\end{equation}

For $\Bi\in \CI_w$
consider the 
$(\BC_q[G]^{\otimes t},
\BC_q[H(\Bi)]
)$-bimodule 
$\CM_{i_1}\otimes\cdots\otimes\CM_{i_m}$.
Via the iterated comultiplication 
$\Delta_{t-1}:\BC_q[G]\to\BC_q[G]^{\otimes t}$
and the algebra homomorphism
\begin{equation}
\label{eq:base-change}
\Delta_\Bi:\BC_q[H]\to \BC_q[H(\Bi)]
\end{equation}
given by 
\[
\Delta_{\Bi}(\chi)=
\sum_{(\chi)_{m-1}}
z_{\Bi,1}\chi_{(0)}|_{U(i_1)^0}
\otimes\cdots\otimes
z_{\Bi,m}\chi_{(m-1)}|_{U(i_m)^0},
\]
we can regard $\CM_{i_1}\otimes\cdots\otimes\CM_{i_m}$ as a $(\BC_q[G],\BC_q[H(\Bi)])$-bimodule or a $(\BC_q[G],\BC_q[H])$-bimodule.

Define a linear map
\[
F'_\Bi:\CM_w\to\CM_{i_1}\otimes\cdots\otimes\CM_{i_m}
\]
by
\begin{multline*}
\langle
(\Theta_{i_1}\otimes\cdots\otimes\Theta_{i_m})(F'_\Bi(m)),
u_1\otimes\cdots \otimes u_m
\rangle
\\
=
\langle\Theta_w(m),
(\dT_{z_{\Bi,1}}^{-1}(u_1))\cdots(\dT_{z_{\Bi,m}}^{-1}(u_m))
\rangle
\\
(m\in\CM_w, u_j\in U(i_j)^{\geqq0}).
\end{multline*}

In this subsection we will show the following.
\begin{theorem}
\label{thm:isomorphismmm}
The linear map 
$F'_\Bi$ is a homomorphism of $(\BC_q[G],\BC_q[H])$-bimodules, and 
it induces an isomorphism
\begin{equation}
\label{eq:main-isom}
{F_\Bi}:\CM_w\otimes_{\BC_q[H]}\BC_q[H(\Bi)]
\to
\CM_{i_1}\otimes\cdots\otimes\CM_{i_m}
\end{equation}
of $(\BC_q[G],\BC_q[H(\Bi)])$-bimodules, where $\BC_q[H]\to \BC_q[H(\Bi)]$ is given by $\Delta_\Bi$.
\end{theorem}
We first note the following.
\begin{lemma}
\label{lem:iteration-comult}
Let $\varphi\in\BC_q[G]$, $w, w_1,\dots,w_{k}\in W$.
Then we have
\begin{align*}
&\Delta_k(\varphi\dT_{w})
\\
=&
\sum_{(\varphi)_{k}}
(\dT_{w_{1}}^{-1}\varphi_{(0)}\dT_w)\otimes
(\dT_{w_{2}}^{-1}\varphi_{(1)}\dT_{w_{1}})\otimes
\cdots
\otimes
(\dT_{w_{k}}^{-1}\varphi_{(k-1)}\dT_{w_{k-1}})\otimes
(\varphi_{(k)}\dT_{w_{k}}).
\end{align*}
\end{lemma}
\begin{proof}
By induction we may assume that $k=1$.
Set $x=w_1$.
We may also assume that
$\varphi=\Phi_{v^*\otimes v}\;(v\in V, v^*\in V^*)$ for some $V\in\Mod_0(U)$.
Let $\{v_j\}$ be a basis of $V$ and let $\{v_j^*\}$ be its dual basis.
Then we have
\[
\Delta(\Phi_{v^*\otimes v})
=\sum_j\Phi_{v^*\otimes v_j}\otimes \Phi_{v_j^*\otimes v}.
\]
Since the dual basis of $\{\dT_x^{-1}v_j\}$ is $\{v_j^*\dT_x\}$, we have for
$u_0, u_1\in U$ that
\begin{align*}
&\langle
\Delta(\Phi_{v^*\otimes v}\dT_w),u_0\otimes u_1\rangle
=
\langle v^*\dT_w,u_0u_1v\rangle
\\
=&
\sum_j
\langle v_j^*\dT_x,u_1v\rangle
\langle v^*\dT_w,u_0\dT_x^{-1}v_j\rangle
=
\sum_j\langle\Phi_{v^*\dT_w\otimes \dT_x^{-1}v_j}\otimes\Phi_{v_j^*\dT_x\otimes v},
u_0\otimes u_1\rangle
\\
=&
\sum_j\langle(\dT_x^{-1}\Phi_{v^*\otimes v_j}\dT_w)\otimes
(\Phi_{v_j^*\otimes v}\dT_x),
u_0\otimes u_1\rangle
\end{align*}
\end{proof}

\begin{lemma}
For $\varphi\in\BC_q[G]$ we have
\begin{multline*}
\langle\Theta_w(\varphi\star1),
(\dT_{z_{\Bi,1}}^{-1}(u_1))\cdots(\dT_{z_{\Bi,m}}^{-1}(u_m))
\rangle
\\
=
\sum_{(\varphi)_{m-1}}
\prod_{r=1}^m
\langle
\Theta_{i_r}(
r^G_{G(i_r)}(\varphi_{(r-1)})\star1),
u_r
\rangle
\\
(u_r\in U(i_r)^{\geqq0}).
\end{multline*}
\end{lemma}
\begin{proof}
By Lemma \ref{lem:iteration-comult} we have
\begin{align*}
&\Delta_{m-1}(\varphi\dT_{w})
\\
=&
\sum_{(\varphi)_{m}}
(\dT_{z_{\Bi,1}}^{-1}\varphi_{(0)}\dT_{z_{\Bi,0}})\otimes
(\dT_{z_{\Bi,2}}^{-1}\varphi_{(1)}\dT_{z_{\Bi,1}})\otimes
\cdots
\otimes
(\dT_{z_{\Bi,m}}^{-1}\varphi_{(m-1)}\dT_{z_{\Bi,m-1}})
.
\end{align*}
Hence
\begin{align*}
&
\langle
\Theta_w(\varphi\star1),
(\dT_{z_{\Bi,1}}^{-1}(u_1))\cdots(\dT_{z_{\Bi,m}}^{-1}(u_m))
\rangle
\\
=&
\langle\varphi\dT_w,
(\dT_{z_{\Bi,1}}^{-1}(u_1))\cdots(\dT_{z_{\Bi,m}}^{-1}(u_m))
\rangle
\\
=&
\langle
\Delta_{m-1}(\varphi\dT_w),
(\dT_{z_{\Bi,1}}^{-1}(u_1))\otimes\cdots\otimes(\dT_{z_{\Bi,m}}^{-1}(u_m))
\rangle
\\
=&
\sum_{(\varphi)_{m-1}}
\prod_{r=1}^m
\langle\dT_{z_{\Bi,r}}^{-1}\varphi_{(r-1)}\dT_{z_{\Bi,r-1}},
\dT_{z_{\Bi,r}}^{-1}(u_r)
\rangle
\\
=&
\sum_{(\varphi)_{m-1}}
\langle
\prod_{r=1}^m
\varphi_{(r-1)}\dT_{i_r},
u_r
\rangle
\\
=&
\sum_{(\varphi)_{m-1}}
\prod_{r=1}^m
\langle
\Theta_{i_r}(
r^G_{G(i_r)}(\varphi_{(r-1)})\star1),
u_r
\rangle
\end{align*}
by Lemma \ref{lem:Tlr1}.
\end{proof}

Now we give a proof of Theorem \ref{thm:isomorphismmm}.

We first show that $F'_\Bi$ is a homomorphism of right $\BC_q[H]$-modules.
For $m\in\CM_w$, 
$x_j\in U(i_j)_{c_j\alpha_{i_j}}^{+}$, 
$\chi\in\BC_q[H]$ we have
\begin{align*}
&
\langle
(\Theta_{i_1}\otimes\cdots\otimes\Theta_{i_m})(F'_\Bi(m\chi)),
x_1k_{i_1}^{p_1}\otimes\cdots \otimes x_mk_{i_m}^{p_m}
\rangle
\\
=&
\langle\Theta_w(m\chi),
(\dT_{z_{\Bi,1}}^{-1}(x_1k_{i_1}^{p_1}))\cdots(\dT_{z_{\Bi,m}}^{-1}(x_mk_{i_m}^{p_m}))
\rangle
\\
=&
q^{A}
\langle\Theta_w(m\chi),
(\dT_{z_{\Bi,1}}^{-1}(x_1))\cdots(\dT_{z_{\Bi,m}}^{-1}(x_m))
k_{z_{\Bi,1}^{-1}\alpha_{i_1}}^{p_1}\cdots k_{z_{\Bi,m}^{-1}\alpha_{i_m}}^{p_m}
\rangle
\\
=&
q^{A}
\langle\Theta_w(m),
(\dT_{z_{\Bi,1}}^{-1}(x_1))\cdots(\dT_{z_{\Bi,m}}^{-1}(x_m))
k_{z_{\Bi,1}^{-1}\alpha_{i_1}}^{p_1}\cdots k_{z_{\Bi,m}^{-1}\alpha_{i_m}}^{p_m}
\rangle
\\
&\hspace{7cm}\times
\langle\chi, k_{z_{\Bi,1}^{-1}\alpha_{i_1}}^{p_1}\cdots k_{z_{\Bi,m}^{-1}\alpha_{i_m}}^{p_m}\rangle
\\
=&
\langle\Theta_w(m),
(\dT_{z_{\Bi,1}}^{-1}(x_1k_{i_1}^{p_1}))\cdots(\dT_{z_{\Bi,m}}^{-1}(x_mk_{i_m}^{p_m}))
\rangle
\langle\chi, k_{z_{\Bi,1}^{-1}\alpha_{i_1}}^{p_1}\cdots k_{z_{\Bi,m}^{-1}\alpha_{i_m}}^{p_m}\rangle
\\
=&
\langle
(\Theta_{i_1}\otimes\cdots\otimes\Theta_{i_m})(F'_\Bi(m)),
x_1k_{i_1}^{p_1}\otimes\cdots \otimes x_mk_{i_m}^{p_m}
\rangle
\\
&\hspace{7cm}\times
\langle\chi, k_{z_{\Bi,1}^{-1}\alpha_{i_1}}^{p_1}\cdots k_{z_{\Bi,m}^{-1}\alpha_{i_m}}^{p_m}\rangle
\\
=&
\langle
\{(\Theta_{i_1}\otimes\cdots\otimes\Theta_{i_m})(F'_\Bi(m))\}\Delta_\Bi(\chi),
x_1k_{i_1}^{p_1}\otimes\cdots \otimes x_mk_{i_m}^{p_m}
\rangle
\\
=&
\langle
(\Theta_{i_1}\otimes\cdots\otimes\Theta_{i_m})(F'_\Bi(m)\Delta_\Bi(\chi)),
x_1k_{i_1}^{p_1}\otimes\cdots \otimes x_mk_{i_m}^{p_m}
\rangle,
\end{align*}
where 
\[
A=
\sum_{r=1}^{m-1}p_r(z_{\Bi,r}^{-1}\alpha_{i_r},c_{r+1}z_{\Bi,r+1}^{-1}\alpha_{i_{r+1}}+\cdots+c_{m}z_{\Bi,m}^{-1}\alpha_{i_{m}}).
\]
Hence 
$F'_\Bi$ is a homomorphism of right $\BC_q[H]$-modules.

We next show that 
$F'_\Bi$ is a homomorphism of left $\BC_q[G]$-modules.
It is sufficient to show $F'_\Bi(\varphi m)=\varphi F'_\Bi(m)$ for
$\varphi\in\BC_q[G]$, $m\in\CM_w$.
Since $F'_\Bi$  is a homomorphism of right $\BC_q[H]$-module,
we may assume that 
$m=\psi\star1$\;($\psi\in\BC_q[G]$).
Then we have
\begin{align*}
&
\langle
(\Theta_{i_1}\otimes\cdots\otimes\Theta_{i_m})(F'_\Bi(\psi\star1)),
u_1\otimes\cdots \otimes u_m
\rangle
\\
=&
\langle\Theta_w(\psi\star1),
(\dT_{z_{\Bi,1}}^{-1}(u_1))\cdots(\dT_{z_{\Bi,m}}^{-1}(u_m))
\rangle
\\
=&
\sum_{(\psi)_{m-1}}
\prod_{r=1}^m
\langle
\Theta_{i_r}(
r^G_{G(i_r)}(\psi_{(r-1)})\star1),
u_r
\rangle
\\
=&
\sum_{(\psi)_{m-1}}
\langle
(\Theta_{i_1}\otimes\cdots\otimes\Theta_{i_m})
(
(r^G_{G(i_1)}(\psi_{(0)})\star1)\otimes\cdots\otimes
(r^G_{G(i_{m-1})}(\psi_{(m-1)})\star1)
)
\\
&\hspace{9cm}
,u_1\otimes\cdots \otimes u_m
\rangle.
\end{align*}
Hence 
\[
F'_\Bi(\psi\star1)=\sum_{(\psi)_{m-1}}
(r^G_{G(i_1)}(\psi_{(0)})\star1)\otimes\cdots\otimes
(r^G_{G(i_m)}(\psi_{(m-1)})\star1).
\]
It follows that
\begin{align*}
&F'_\Bi(\varphi m)=F'_\Bi(\varphi\psi\star1)
\\
=&
\sum_{(\varphi)_{m-1}, (\psi)_{m-1}}
(r^G_{G(i_1)}(\varphi_{(0)}\psi_{(0)})\star1)\otimes\cdots\otimes
(r^G_{G(i_m)}(\varphi_{(m-1)}\psi_{(m-1)})\star1)
\\
=&
\varphi\sum_{(\psi)_{m-1}}
(r^G_{G(i_1)}(\psi_{(0)})\star1)\otimes\cdots\otimes
(r^G_{G(i_m)}(\psi_{(m-1)})\star1)
=\varphi F'_\Bi(m).
\end{align*}
Therefore,
$F'_\Bi$ is a homomorphism of left $\BC_q[G]$-modules.

Since $F'_\Bi$ is a homomorphism of $(\BC_q[G],\BC_q[H])$-bimodules, it induces a homomorphism
\[
{F}_\Bi:\CM_w\otimes_{\BC_q[H]}\BC_q[H(\Bi)]
\to
\CM_{i_1}\otimes\cdots\otimes\CM_{i_m}
\qquad(a\otimes\chi\mapsto F'_\Bi(a)\chi)
\]
of $(\BC_q[G],\BC_q[H(\Bi)])$-bimodules.
It remains to show that ${F}_\Bi$ is bijective.
Via 
$\Theta_w$ we have
\begin{align*}
\CM_w\otimes_{\BC_q[H]}\BC_q[H(\Bi)]
\cong& ( U^+[\dT_w^{-1}]^\bigstar\otimes\BC_q[H])\otimes_{\BC_q[H]}\BC_q[H(\Bi)]
\\
\cong&  U^+[\dT_w^{-1}]^\bigstar\otimes \BC_q[H(\Bi)],
\end{align*}
and via
$\Theta_{i_1}\otimes\cdots\otimes\Theta_{i_m}$ we have
\begin{align*}
&\CM_{i_1}\otimes\cdots\otimes\CM_{i_m}
\\
\cong&
\{(U(i_1)^+)^\bigstar\otimes\BC_q[H(i_1)]\}
\otimes\cdots\otimes
\{(U(i_m)^+)^\bigstar\otimes\BC_q[H(i_m)]\}
\\
\cong&
\{(U(i_1)^+)^\bigstar
\otimes\cdots\otimes
(U(i_m)^+)^\bigstar\}
\otimes \BC_q[H(\Bi)].
\end{align*}
Hence the assertion follows from
\begin{multline*}
(U(i_1)^+)^\bigstar
\otimes\cdots\otimes
(U(i_m)^+)^\bigstar
\cong
U^+[\dT_w^{-1}]^\bigstar
\\
(x_1\otimes\cdots x_m\leftrightarrow
\dT_{z_{\Bi,1}}^{-1}(x_1)\otimes\cdots\otimes
\dT_{z_{\Bi,m}}^{-1}(x_m)).
\end{multline*}

The proof of Theorem \ref{thm:isomorphismmm} is complete

\section{Basis elements}
\subsection{}
Let $w\in W$ with $\ell(w)=m$, and fix $\Bi=(i_1,\cdots, i_m)\in\CI_w$.

For $\Bn=(n_1,\dots, n_m)\in(\BZ_{\geqq0})^m$ we denote by 
$p_\Bi(\Bn)$ the element of $\CM_w\otimes_{\BC_q[H]}\BC_q[H(\Bi)]$ corresponding to 
\[
p_{i_1}(n_1)\otimes\cdots\otimes p_{i_m}(n_m)
\in
\CM_{i_1}\otimes\cdots\otimes\CM_{i_m}
\]
under the isomorphism \eqref{eq:main-isom}.

By \eqref{eq:identification} we have
\begin{align*}
U^{\geqq0}[\dT_w^{-1}]^\bigstar\otimes_{\BC_q[H]}\BC_q[H(\Bi)]
\cong &U^{+}[\dT_w^{-1}]^\bigstar\otimes \BC_q[H(\Bi)].
\end{align*}
Hence $\Theta_w$ induces an isomorphism 
\begin{equation}
\Theta_{w,{\Bi}}:\CM_w\otimes_{\BC_q[H]}\BC_q[H(\Bi)]\to
U^{+}[\dT_w^{-1}]^\bigstar\otimes \BC_q[H(\Bi)]
\end{equation}
of right $\BC_q[H(\Bi)]$-modules.
We will regard $U^{+}[\dT_w^{-1}]^\bigstar\otimes \BC_q[H(\Bi)]$ as a subset of $\Hom_\BF(U^{+}[\dT_w^{-1}],\BC_q[H(\Bi)])$ in the following.

For $r=1,\cdots, m$ and $\Bn\in(\BZ_{\geqq0})^m$ set
\begin{equation}
\beta_{\Bi,r}=z_{\Bi,r}^{-1}\alpha_{i_r}, \qquad
\gamma_{\Bi,\Bn,r}=n_{r+1}\beta_{\Bi,r+1}+\cdots +n_m\beta_{\Bi,m}.
\end{equation}

\begin{proposition}
\label{prop:isomorphismmm2}
\begin{align*}
&\langle \Theta_{w, \Bi}(p_\Bi(\Bn)),\te_\Bi^{\Bn'}\rangle
\\
=&\delta_{\Bn,\Bn'}
\left\{
\prod_{r=1}^m
(-1)^{n_r}q_{i_r}^{n_r}[n_r]_{q_{i_r}}!
\right\}
\chi_{i_1}^{\langle\beta^\vee_{\Bi,1},\gamma_{\Bi,\Bn,1}\rangle}
\otimes\cdots\otimes
\chi_{i_m}^{\langle\beta^\vee_{\Bi,m},\gamma_{\Bi,\Bn,m}\rangle}.
\end{align*}
\end{proposition}
\begin{proof}
Define $a\in\CM_w$ by
\[
\langle\Theta_w(a),\te_\Bi^{\Bn'}t\rangle
=\delta_{\Bn\Bn'}\varepsilon(t)\qquad(\Bn'\in(\BZ_{\geqq0})^m, t\in U^0).
\]
Then we have
\begin{align*}
&\langle
(\Theta_{i_1}\otimes\cdots\otimes\Theta_{i_m})({F}_\Bi(a\otimes1)),
e_{i_1}^{n_1'}k_{i_1}^{j_1}\otimes\cdots\otimes
e_{i_m}^{n_m'}k_{i_m}^{j_m}
\rangle
\\
=&
\langle
(\Theta_{i_1}\otimes\cdots\otimes\Theta_{i_m})(F'_\Bi(a))),
e_{i_1}^{n_1'}k_{i_1}^{j_1}\otimes\cdots\otimes
e_{i_m}^{n_m'}k_{i_m}^{j_m}
\rangle
\\
=&
\langle
\Theta_w(a),
T_{z_{\Bi,1}}^{-1}(e_{i_1}^{n_1'}k_{i_1}^{j_1})\cdots
T_{z_{\Bi,m}}^{-1}(e_{i_m}^{n_m'}k_{i_m}^{j_m})
\rangle
\\
=&
q^{A'}
\langle
\Theta_w(a),
\te^{\Bn'}_\Bi
k_{z_{\Bi,1}^{-1}\alpha_{i_1}}^{j_1}\cdots
k_{z_{\Bi,m}^{-1}\alpha_{i_m}}^{j_m}
\rangle
\\
=&\delta_{\Bn\Bn'}q^{A}
\\
=&\delta_{\Bn\Bn'}\prod_{r=1}^m
 (q_{i_r}^{j_r})^{\langle\beta^\vee_{\Bi,r},\gamma_{\Bi,\Bn,r}\rangle},
\end{align*}
where
\begin{align*}
A'=\sum_{r=1}^{m-1}
\langle j_r\beta_{\Bi,r},\gamma_{\Bi,\Bn,r+1}
\rangle,
\qquad
A=\sum_{r=1}^{m-1}
\langle j_r\beta_{\Bi,r},\gamma_{\Bi,\Bn',r+1}
\rangle.
\end{align*}
On the other hand we have
\begin{align*}
&\langle
(\Theta_{i_1}\otimes\cdots\otimes\Theta_{i_m})(
p_{i_1}(n_1)\otimes\cdots\otimes p_{i_m}(n_m)
),
e_{i_1}^{n_1'}k_{i_1}^{j_1}\otimes\cdots\otimes
e_{i_m}^{n_m'}k_{i_m}^{j_m}
\rangle
\\
=&
\prod_{r=1}^m
\delta_{n_rn'_r}(-1)^{n_r}q_{i_r}^{n_r}[n_r]_{q_{i_r}}!
=
\delta_{\Bn\Bn'}
\prod_{r=1}^m
(-1)^{n_r}q_{i_r}^{n_r}[n_r]_{q_{i_r}}!.
\end{align*}
\end{proof}
\subsection{}
We rewrite Proposition \ref{prop:isomorphismmm2} using $\Xi_w$ instead of $\Theta_w$.
Note that the isomorphism \eqref{eq:isomd}
induces
\[
\BC_q[B^-]^{\bullet_w}\otimes_{\BC_q[H]}\BC_q[H(\Bi)]
\cong (U^-)^\bigstar\otimes \BC_q[H(\Bi)]\;
\left(
\subset\Hom_\BF(U^-,\BC_q[H(\Bi)])
\right).
\]
Hence $\Xi_w$ induces an injection
\[
\Xi_{w,\Bi}:\CM_w\otimes_{\BC_q[H]}\BC_q[H(\Bi)]\to
(U^{-})^\bigstar\otimes \BC_q[H(\Bi)]\;
\left(
\subset\Hom_\BF(U^-,\BC_q[H(\Bi)])
\right).
\]
Recall that
$\{\hf_\Bi^{\Bn}\}_{\Bn}$ forms a basis $U^-[\hT_w]$ and the multiplication induces an isomorphism
$
(U^-\cap\hT_wU^-)\otimes U^-[\hT_w]\cong U^-.
$ (see Proposition \ref{prop:base2}, \eqref{eq:Ubunkai-h}).

\begin{proposition}
\label{prop:isomorphismmm3}
For $y\in U^-\cap\hT_wU^-$ we have
\begin{align*}
&\langle \Xi_{w,\Bi}(p_\Bi(\Bn)),y\hf_\Bi^{\Bn'}\rangle
\\
=&\varepsilon(y)\delta_{\Bn,\Bn'}
\left\{
\prod_{r=1}^m
(-1)^{n_r}q_{i_r}^{n_r}[n_r]_{q_{i_r}}!
\right\}
\chi_{i_1}^{\langle\beta^\vee_{\Bi,1},\gamma_1\rangle}
\otimes\cdots\otimes
\chi_{i_m}^{\langle\beta^\vee_{\Bi,m},\gamma_m\rangle}.
\end{align*}
\end{proposition}
\begin{proof}
Let
\[
\Omega_{w,\Bi}: U^+[\dT_w^{-1}]^\bigstar\otimes \BC_q[H(\Bi)]\to(U^-)^\bigstar\otimes \BC_q[H(\Bi)]
\]
be the homomorphism of right $\BC_q[H(\Bi)]$-modules induced by $\Omega_w$.
For 
$f\in U^+[\dT_w^{-1}]^\bigstar$
the element of $U^{\geqq0}[\dT_w^{-1}]^\bigstar\otimes_{\BC_q[H]} \BC_q[H(\Bi)]$ corresponding to 
$f\otimes1\in U^+[\dT_w^{-1}]^\bigstar\otimes \BC_q[H(\Bi)]$
is written as 
$\tilde{f}\otimes1$, where 
$\tilde{f}\in U^{\geqq0}[\dT_w^{-1}]^\bigstar$ is given by
\[
\langle\tilde{f},xt\rangle
=
\langle f,x\rangle\varepsilon(t)
\qquad(x\in U^+[\dT_w^{-1}], t\in U^0).
\]
Then for 
$y_1\in U^-[\hT_w]$, $y_2\in U^-\cap\hT_wU^-$, $t\in U^0$ we have
\begin{multline*}
\langle\Omega_w(\tilde{f}),ty_2y_1\rangle
=
\varepsilon(y_2)\langle\tilde{f},\dT_w^{-1}S(ty_1)\rangle
=
\varepsilon(y_2t)\langle{f},\dT_w^{-1}S(y_1)\rangle
\\
(\dT_w^{-1}S(y_1)\in U^+[\dT_w^{-1}]).
\end{multline*}
Namely, the element of $(U^-)^\bigstar\otimes \BC_q[H(\Bi)]$ corresponding to $f\otimes1$ is written as $\hat{f}\otimes1$, where $\hat{f}\in(U^-)^\bigstar$ is given by 
\[
\langle\hat{f},y_2y_1\rangle=
\varepsilon(y_2)\langle{f},\dT_w^{-1}S(y_1)\rangle
\qquad(
y_1\in U^-[\hT_w], y_2\in U^-\cap\hT_wU^-).
\]
Hence for $y\in U^-\cap\hT_wU^-$ we have
\[
\langle
\Xi_{w,\Bi}(p_\Bi(\Bn)),y\hf_\Bi^{\Bn'}
\rangle
=
\varepsilon(y)
\langle
\Theta_{w,\Bi}(p_\Bi(\Bn)),\te_\Bi^{\Bn'}
\rangle.
\]
\end{proof}
\subsection{}
Set 
\begin{equation}
U^{\geqq0}[\hT_w]=
U^+[\hT_w]U^0
\subset U^{\geqq0}
\end{equation}
and define 
\begin{equation}
\Psi_w:U^{\geqq0}[\hT_w]\to\BC_q[B^-]
\end{equation}
by
\[
\langle\Psi_w(x),u\rangle
=
\tau(x,u)
\qquad(x\in U^{\geqq0}[\hT_w], u\in U^{\leqq0}).
\]
By Proposition \ref{prop:base2a} $\Psi_w$ is an injective algebra homomorphism and its image is contained in $\CA_w$.
Hence there exists a unique injective linear map
\begin{equation}
\Gamma_w:U^{\geqq0}[\hT_w]\to\CM_w
\end{equation}
such that 
$\Xi_w\circ\Gamma_w=\Psi_w$.

\begin{theorem}
\label{thm:pd}
We have
\[
p_\Bi(\Bn)
=d_\Bi(\Bn)
\Gamma_w(\he_\Bi^{(\Bn)})\otimes
\left\{
\chi_{i_1}^{\langle\beta^\vee_{\Bi,1},\gamma_1\rangle}
\otimes\cdots\otimes
\chi_{i_m}^{\langle\beta^\vee_{\Bi,m},\gamma_m\rangle}
\right\},
\]
where
\[
d_\Bi(\Bn)=\prod_{r=1}^md_{i_r}(n_r),\qquad
d_i(n)=q^{n(n+1)/2}(q^{-1}-q)^n.
\]
\end{theorem}
\begin{proof}
For $y\in U^-\cap\hT_wU^-$ we have
\begin{align*}
\langle\Xi_w(\Gamma_w(\he_\Bi^{(\Bn)}))
\otimes1,y\hf_\Bi^{\Bn'}\rangle
=&\varepsilon(y)
\langle\Psi_w(\he_\Bi^{(\Bn)}),\hf_\Bi^{\Bn'}\rangle
=\varepsilon(y)\tau(\he_\Bi^{(\Bn)},\hf_\Bi^{\Bn'})
\\
=&\varepsilon(y)\delta_{\Bn\Bn'}
\prod_{t=1}^mc_{q_{i_t}}(n_t),
\end{align*}
where
\[
c_q(n)=[n]!q^{-n(n-1)/2}(q-q^{-1})^{-n}.
\]
\end{proof}
\subsection{}
Set $m_0=\ell(w_0)$.
In this subsection we consider the case $w=w_0$.
\begin{lemma}
\label{lem:jmath-e}
Let $i\in I$ and define $i'\in I$ by $w_0\alpha_i=-\alpha_{i'}$.
Then we have
\[
\Gamma_{w_0}(e_i)=
\frac1{1-q_i^2}
(\sigma^{w_0}_{-\varpi_{i'}}e_i)(\sigma^{w_0}_{-\varpi_{i'}})^{-1}\star1.
\]
\end{lemma}
\begin{proof}
It is sufficient to show 
\[
\Psi_{w_0}(e_i)=
\frac1{1-q_i^2}
\Xi_{w_0}
((\sigma^{w_0}_{-\varpi_{i'}}e_i)(\sigma^{w_0}_{-\varpi_{i'}})^{-1})\star1).
\]
Set $v^*=v^*_{-\varpi_{i'}}\dT_{w_0}^{-1}$,
$v=\dT_{w_0}v_{-\varpi_{i'}}$,
so that 
$v^*\in V^*(-\varpi_{i'})_{\varpi_i}$,
$v\in V(-\varpi_{i'})_{\varpi_i}$ with $\langle v^*,v\rangle=1$.
For $t\in U^0$, $y\in U^-\cap \hT_{s_i}U^-$, $p\geqq0$
we have
\begin{align*}
&\langle
\Xi_{w_0}'(\sigma^{w_0}_{-\varpi_{i'}}e_i),tyf_i^p
\rangle
=\langle
\dT_{w_0}(\sigma^{w_0}_{-\varpi_{i'}}e_i),S(tyf_i^p)
\rangle
\\
=&
\langle
v^*_{-\varpi_{i'}}\dT_{w_0}^{-1}e_i,
S(tyf_i^p)\dT_{w_0}v_{-\varpi_{i'}}
\rangle
=
\langle
v^*e_i,
S(tyf_i^p)v
\rangle
\\
=&
\chi_{-\varpi_i}(t)
\langle
v^*e_i,
S(f_i^p)S(y)v
\rangle
=
-\chi_{-\varpi_i}(t)
\varepsilon(y)\delta_{p1}
\langle
v^*e_i,
f_ik_iv
\rangle
\\
=&
-\chi_{-\varpi_i}(t)
\varepsilon(y)\delta_{p1}q_i
\langle
v^*e_i,
f_iv
\rangle
\\
=&
-\chi_{-\varpi_i}(t)
\varepsilon(y)\delta_{p1}q_i
\langle
v^*,
\frac{k_i-k_i^{-1}}{q_i-q_i^{-1}}
v
\rangle
=
-\chi_{-\varpi_i}(t)
\varepsilon(y)\delta_{p1}q_i.
\end{align*}
Hence
\begin{align*}
&\langle
\Xi_{w_0}
((\sigma^{w_0}_{-\varpi_{i'}}e_i)(\sigma^{w_0}_{-\varpi_{i'}})^{-1})\star1)
,tyf_i^p
\rangle
=\langle\chi_{\varpi_i}\Xi_{w_0}'(\sigma^{w_0}_{-\varpi_{i'}}e_i),tyf_i^p
\rangle
\\
=&\sum_{(t)}
\chi_{\varpi_i}(t_{(0)})(-\chi_{-\varpi_i}(t_{(1)})
\varepsilon(y)\delta_{p1}q_i)
=-\varepsilon(t)
\varepsilon(y)\delta_{p1}q_i.
\end{align*}
On the other hand by \eqref{eq:Dr7} and Proposition \ref{prop:base2a} we have
\[
\langle\Psi_{w_0}(e_i), ,tyf_i^p\rangle
=\tau(e_i,tyf_i^p)
=\frac1{q_i-q_i^{-1}}\varepsilon(t)\varepsilon(y)\delta_{p1}
\]
for $t\in U^0$, $y\in U^-\cap \hT_{s_i}U^-$, $p\geqq0$.
\end{proof}

\begin{proposition}
\label{prop:KOY-conj}
For $\Bi\in\CI_{w_0}$,  $i\in I$, $\Bn\in(\BZ_{\geqq0})^{m_0}$ write 
\[
\he_\Bi^{(\Bn)}e_i=\sum_{\Bn'}c_{\Bn\Bn'}\he_\Bi^{(\Bn')}.
\]
Then we have
\begin{align*}
&\left\{
\frac1{1-q_i^2}
(\sigma^{w_0}_{-\varpi_{i'}}e_i)(\sigma^{w_0}_{-\varpi_{i'}})^{-1}
\right\}
p_\Bi(\Bn)
\\
=&
\sum_{\Bn'}c_{\Bn\Bn'}
\frac{d_\Bi(\Bn)}{d_\Bi(\Bn')}
p_{\Bi}(\Bn')
\left(
\chi_{i_1}^{\langle\beta^\vee_{\Bi,1},\gamma_{\Bi,\Bn,1}-\gamma_{\Bi,\Bn',1}\rangle}
\otimes\cdots\otimes
\chi_{i_{m_0}}^{\langle\beta^\vee_{\Bi,m_0},\gamma_{\Bi,\Bn,m_0}-\gamma_{\Bi,\Bn',m_0}\rangle}
\right),
\end{align*}
where $i'$ is as in Lemma \ref{lem:jmath-e}.
\end{proposition}
\begin{proof}
Set $\varphi=\frac1{1-q_i^2}
(\sigma^{w_0}_{-\varpi_{i'}}e_i)(\sigma^{w_0}_{-\varpi_{i'}})^{-1}
\in\CS_{w_0}^{-1}\BC_q[G/N^-]$
so that $\Gamma_{w_0}(e_i)=\varphi\star1$.
For $\Bn'\in(\BZ_{\geqq0})^{m_0}$ take
$\varphi_{\Bn'}\in\CS_{w_0}^{-1}\BC_q[G/N^-]$ such that
$
\Gamma_{w_0}(\he_\Bi^{(\Bn')})=\varphi_{\Bn'}\star1$.
Then we have
\begin{align*}
&\Xi_{w_0}(\varphi\varphi_\Bn\star1)
=\Xi_{w_0}(\varphi_n\star1)\Xi_{w_0}(\varphi\star1)
=\Xi_{w_0}(\Gamma_{w_0}(\he_\Bi^{(\Bn)}))\Xi_{w_0}(\Gamma_{w_0}(e_i))
\\
=&\Psi_{w_0}(\he_\Bi^{(\Bn)})\Psi_{w_0}(e_i)
=\Psi_{w_0}(\he_\Bi^{(\Bn)}e_i)
=
\sum_{\Bn'}c_{\Bn\Bn'}\Psi_{w_0}(\he_\Bi^{(\Bn')})
\\
=&
\sum_{\Bn'}c_{\Bn\Bn'}\Xi_{w_0}(\Gamma_{w_0}(\he_\Bi^{(\Bn')}))
=
\sum_{\Bn'}c_{\Bn\Bn'}\Xi_{w_0}(\varphi_{\Bn'}\star1).
\end{align*}
Hence
\[
\varphi\varphi_\Bn\star1=\sum_{\Bn'}c_{\Bn\Bn'}\varphi_{\Bn'}\star1.
\]
It follows that 
\[
\varphi\Gamma_{w_0}(\he_\Bi^{(\Bn)})
=
\sum_{\Bn'}c_{\Bn\Bn'}\Gamma_{w_0}(\he_\Bi^{(\Bn')}).
\]
Therefore, the assertion is a consequence of Theorem \ref{thm:pd}.
\end{proof}

\section{Specialization}
\subsection{}
We denote by $\Hom_{\alg}(\BC_q[H],\BF)$ the set of algebra homomorphisms from $\BC_q[H]$ to $\BF$.
It is endowed with a structure of commutative group via the multiplication
\begin{multline*}
(\theta_1\theta_2)(\chi)=\sum_{(\chi)}
\theta_1(\chi_{(0)})\theta_2(\chi_{(1)})
\\
(\theta_1, \theta_2\in\Hom_{\alg}(\BC_q[H],\BF),\quad
\chi\in\BC_q[H]).
\end{multline*}
The identity element is given by $\varepsilon$, and the inverse of $\theta$ is given by $\theta\circ S$.

For $\theta\in \Hom_{\alg}(\BC_q[H],\BF)$ 
we
denote by $\BF_\theta=\BF1_\theta$ the corresponding left $\BC_q[H]$-module.
For $\theta\in \Hom_{\alg}(\BC_q[H],\BF)$ 
and $w\in W$
we define an $\CS_w^{-1}\BC_q[G]$-module $\CM_w^\theta$ by
\[
\CM_w^\theta=\CM_w\otimes_{\BC_q[H]}\BF_\theta.
\]
Set $1^\theta_w=(1\star1)\otimes 1_\theta\in\CM_w^\theta$.
We have
\[
\CM_w^\theta
\cong
\CS_w^{-1}\BC_q[G]\otimes_{\CS_w^{-1}\BC_q[N_w^-\backslash G]}\BF
\cong
\BC_q[G]\otimes_{\BC_q[N_w^-\backslash G]}\BF,
\]
where $\CS_w^{-1}\BC_q[N_w^-\backslash G]\to\BF$ is given by $\theta\circ{\eta}_w$.

Note that we have a decomposition
\[
\BC_q[N_w^-\backslash G]
=
\bigoplus_{\lambda\in P}\BC_q[N_w^-\backslash G]_\lambda
\]
with
\[
\BC_q[N_w^-\backslash G]_\lambda
=
\{\varphi\in\BC_q[N_w^-\backslash G]
\mid
t\varphi=\chi_\lambda(t)\varphi\quad(t\in U^0)\}.
\]
We have
\begin{equation}
(\theta\circ{\eta}_w)(\varphi)=\varepsilon(\varphi\dT_w)\theta(\chi_\lambda)
\qquad(\lambda\in P, \varphi\in\BC_q[N_w^-\backslash G]_\lambda).
\end{equation}
Indeed, for
$t\in U^0$ we have
\[
\langle{\eta}_w(\varphi),t\rangle=
\langle\varphi\dT_w,t\rangle=
\langle(t\varphi)\dT_w,1\rangle=
\chi_\lambda(t)\varepsilon(\varphi\dT_w),
\]
and hence ${\eta}_w(\varphi)=\varepsilon(\varphi\dT_w)\chi_\lambda$.
Therefore, $(\theta\circ{\eta}_w)(\varphi)=\varepsilon(\varphi\dT_w)\theta(\chi_\lambda)$.

The $\BC_q[H]$-module $\BF_\theta$ can also be regarded as a  $\BC_q[G]$-module
via the canonical Hopf algebra homomorphism $r^G_H:\BC_q[G]\to\BC_q[H]$. 
We denote this $\BC_q[G]$-module by $\BF_\theta^G=\BF1_\theta^G$.

\begin{proposition}
\label{prop:parameter-shift}
Assume that we are given two algebra homomorphisms $\theta_i:\BC_q[H]\to\BF$ $(i=1, 2)$.
Then as a $\BC_q[G]$-module we have
\[
\CM_w^{\theta_1}\cong{\CM}_w^{\theta_2}\otimes_\BF\BF^G_{\theta_1(\theta_2\circ S)}.
\]
Here, the right side is regarded as a $\BC_q[G]$-module via the comultiplication $\Delta:\BC_q[G]\to\BC_q[G]\otimes\BC_q[G]$.
\end{proposition}
\begin{proof}
Let $\lambda\in P$, $\varphi\in\BC_q[N_w^-\backslash G]_\lambda$.
For 
$u\in U$, $t\in U^0$ we have
\[
\langle\Delta(\varphi), u\otimes t\rangle=\langle\varphi,ut\rangle
=
\langle t\varphi,u\rangle
=
\chi_\lambda(t)\langle\varphi,u\rangle,
\]
and hence 
$(\id\otimes r^G_H)(\varphi)=\varphi\otimes\chi_\lambda$.
It follows that
\begin{align*}
&\varphi(1^{\theta_2}_w\otimes1^G_{\theta_1(\theta_2\circ S)})
=
(\varphi1^{\theta_2}_w)\otimes(\chi_\lambda1^G_{\theta_1(\theta_2\circ S)})
\\
=&
\varepsilon(\varphi\dT_w)
\theta_2(\chi_\lambda)
(\theta_1(\theta_2\circ S))(\chi_\lambda)1^{\theta_2}_w\otimes1^G_{\theta_1(\theta_2\circ S)}
\\
=&
\varepsilon(\varphi\dT_w)
(\theta_2\theta_1(\theta_2\circ S))(\chi_\lambda)
1^{\theta_2}_w\otimes1^G_{\theta_1(\theta_2\circ S)}
\\
=&
\varepsilon(\varphi\dT_w)\theta_1(\chi_\lambda)
1^{\theta_2}_w\otimes1^G_{\theta_1(\theta_2\circ S)}
=(\theta_1\circ\eta_w)(\varphi)
1^{\theta_2}_w\otimes1^G_{\theta_1(\theta_2\circ S)}.
\end{align*}
Hence there exists uniquely a homomorphism $F^{\theta_1}_{\theta_2}:\CM_w^{\theta_1}
\to
{\CM}_w^{\theta_2}\otimes_\BF\BF^G_{\theta_1(\theta_2\circ S)}$
of $\BC_q[G]$-modules sending $1^{\theta_1}_w$
to
$1^{\theta_2}_w\otimes1^G_{\theta_1(\theta_2\circ S)}$.
Similarly, we have 
a homomorphism $F^{\theta_2}_{\theta_1}:\CM_w^{\theta_2}
\to
{\CM}_w^{\theta_1}\otimes_\BF\BF^G_{\theta_2(\theta_1\circ S)}$
of $\BC_q[G]$-modules sending $1^{\theta_2}_w$
to
$1^{\theta_1}_w\otimes1^G_{\theta_2(\theta_1\circ S)}$.
Applying $(\bullet)\otimes_\BF\BF^G_{\theta_1(\theta_2\circ S)}$ to
$F^{\theta_2}_{\theta_1}$ we obtain a homomorphism
\[
\tilde{F}^{\theta_2}_{\theta_1}:=F^{\theta_2}_{\theta_1}\otimes_\BF\BF^G_{\theta_1(\theta_2\circ S)}:\CM_w^{\theta_2}\otimes_\BF\BF^G_{\theta_1(\theta_2\circ S)}
\to
{\CM}_w^{\theta_1}
\]
of $\BC_q[G]$-modules sending $1^{\theta_2}_w\otimes 1^G_{\theta_1(\theta_2\circ S)}$
to
$1^{\theta_1}_w$.
It remains to show $\tilde{F}^{\theta_2}_{\theta_1}\circ F^{\theta_1}_{\theta_2}=\id$ and $F^{\theta_1}_{\theta_2}\circ\tilde{F}^{\theta_2}_{\theta_1}=\id$.
The first identity is a consequence of $(\tilde{F}^{\theta_2}_{\theta_1}\circ F^{\theta_1}_{\theta_2})(1^{\theta_1}_w)=1^{\theta_1}_w$.
The second one follows by applying 
$(\bullet)\otimes_\BF\BF^G_{\theta_1(\theta_2\circ S)}$
to
$\tilde{F}^{\theta_1}_{\theta_2}\circ F^{\theta_2}_{\theta_1}=\id$.
\end{proof}

\subsection{}
In view of Proposition \ref{prop:parameter-shift} we only consider the $\CS_w^{-1}\BC_q[G]$-module
\begin{equation}
\overline{\CM}_w=\CM_w^\varepsilon
\end{equation}
in the following.
For $\varphi\in\CS_w^{-1}\BC_q[G]$ we denote by $\overline{\varphi}\in\overline{\CM}_w$ the image of $\varphi\star1$ under  ${\CM}_w\to \overline{\CM}_w$.
Define $\overline{\eta}_w:\CS_w^{-1}\BC_q[N_w^-\backslash G]\to\BF$ as the composite $\varepsilon\circ\eta_w$.
Then we have
\[
\overline{\CM}_w
\cong
\CS_w^{-1}\BC_q[G]\otimes_{\CS_w^{-1}\BC_q[N_w^-\backslash G]}\BF
\cong
\BC_q[G]\otimes_{\BC_q[N_w^-\backslash G]}\BF,
\]
where $\CS_w^{-1}\BC_q[N_w^-\backslash G]\to\BF$ is given by $\overline{\eta}_w$.
Moreover, $\Theta_w$ induces  a linear isomorphism
\[
\overline{\Theta}_w:\overline{\CM}_w\to U^+[\dT_w^{-1}]^\bigstar
\]
given by
\[
\langle\overline{\Theta}_w(\overline{\varphi}),x\rangle
=\langle\varphi\dT_w,x\rangle
\qquad(\varphi\in\BC_q[G], x\in U^+[\dT_w^{-1}]).
\]
By the direct sum decomposition
\[
U^+[\dT_w^{-1}]^\bigstar
=\bigoplus_{\gamma\in Q^+\cap (-w^{-1}Q^+)}
(U^+[\dT_w^{-1}]\cap U^+_\gamma)^*
\]
we have a direct sum decomposition
\begin{equation}
\label{eq:ddecomp}
\overline{\CM}_w
=
\bigoplus_{\gamma\in Q^+\cap (-w^{-1}Q^+)}
\overline{\CM}_{w,\gamma},
\end{equation}
where
\[
\overline{\CM}_{w,\gamma}
=(\overline{\Theta}_w)^{-1}((U^+[\dT_w^{-1}]\cap U^+_\gamma)^*)
\qquad(\gamma\in Q^+\cap (-w^{-1}Q^+)).
\]
Note that 
\[
\overline{\CM}_{w,0}=\BF\overline{1}.
\]
\begin{lemma}
\label{weight-decomp}
For $m\in\overline{\CM}_{w,\gamma}$ and $\lambda\in P$
we have
$
\sigma^w_\lambda m=q^{-(\lambda,\gamma)}m.
$
\end{lemma}
\begin{proof}
We may assume $\lambda\in P^-$.
Take $\varphi\in\BC_q[G]$ such that $\overline{\varphi}=m$.
By Corollary \ref{cor:dT} we have
$\dT_w(\sigma^w_\lambda\varphi)=
(\dT_w\sigma^w_\lambda)(\dT_w\varphi)
$, and hence for $x\in U^+[\dT_w^{-1}]^\bigstar$ we have
\[
\langle\overline{\Theta}_w(\sigma^w_\lambda m),x\rangle
=\langle\dT_w(\sigma^w_\lambda\varphi),\dT_w(x)\rangle
=\langle(\dT_w\sigma^w_\lambda)(\dT_w\varphi),\dT_w(x)\rangle.
\]
Assume $x\in U^+[\dT_w^{-1}]\cap U^+_\delta$ with $\delta\in Q^+\cap(-w^{-1}Q^+)$, and set $y=\dT_w(x)$.
Then we have 
$y\in(U^-_{w\delta})k_{-w\delta}$
by \eqref{eq:twist2}.
Hence
\begin{align*}
\langle\overline{\Theta}_w(\sigma^w_\lambda m),x\rangle
=&
\sum_{(y)}
\langle\dT_w\sigma^w_\lambda, y_{(0)}\rangle
\langle\dT_w\varphi,y_{(1)}\rangle
=
\langle\dT_w\sigma^w_\lambda, k_{-w\delta}\rangle
\langle\dT_w\varphi,y\rangle
\\
=&
q^{-(\lambda,\delta)}
\langle\overline{\Theta}_w(m),x\rangle.
\end{align*}
\end{proof}
\begin{theorem}
\label{thm:irred}
The $\BC_q[G]$-module $\overline{\CM}_w$ is irreducible.
\end{theorem}
\begin{proof}
For any $\BC_q[G]$-submodule $N$ of $\overline{\CM}_w$ we have
\[
N=
\bigoplus_{\gamma\in Q^+\cap (-w^{-1}Q^+)}
(N\cap\overline{\CM}_{w,\gamma})
\]
by Lemma \ref{weight-decomp}.
Since $\overline{\CM}_{w,0}$ is one-dimensional and generates the $\BC_q[G]$-module $\overline{\CM}_{w}$, it is sufficient to show
$N\cap\overline{\CM}_{w,0}\ne\{0\}$ for any non-zero $\BC_q[G]$-submodule $N$ of $\overline{\CM}_w$.
By definition the projection 
$\overline{\CM}_{w}\to \overline{\CM}_{w,0}$ 
with respect to \eqref{eq:ddecomp} is given by
$m\mapsto\langle\overline{\Theta}_w(m),1\rangle\overline{1}$,
and hence it is sufficient to show that for any $m\in\overline{\CM}_{w}\setminus\{0\}$ there exists some $\psi\in\BC_q[G]$ such that 
$\langle\overline{\Theta}_w(\psi m),1\rangle\ne0$.
Take $\varphi\in\BC_q[G]$ such that $\overline{\varphi}=m$.
For $z\in U^+$ and $\lambda\in P^-$ we have
\[
\langle\overline{\Theta}_w((z\sigma^w_\lambda) m),1\rangle
=\langle\{(z\sigma^w_\lambda)\varphi\}\dT_w,1\rangle.
\]
Write
\[
\Delta\dT_w=(\dT_w\otimes \dT_w)\sum_ju^-_j\otimes u^+_j
\]
(see Corollary \ref{cor:dT}).
Then we have
\begin{align*}
&\langle\overline{\Theta}_w((z\sigma^w_\lambda) m),1\rangle
=
\sum_j
\langle z\sigma^w_\lambda\dT_wu_j^-,1\rangle
\langle\varphi\dT_wu_j^+,1\rangle
\\
=&
\sum_j
\langle v^*_\lambda u_j^-,zv_\lambda\rangle
\langle\overline{\Theta}_w(m),u_j^+\rangle
=
\langle v^*_\lambda yz,v_\lambda\rangle
\end{align*}
with $y=\sum_j
\langle\overline{\Theta}_w(m),u_j^+\rangle u_j^-\in U^-\setminus\{0\}$.
Hence it is sufficient to show that for any $y\in U^-$ there exists some $\lambda\in P^-$ and $z\in U^+$ such that $\langle v^*_\lambda yz,v_\lambda\rangle\ne0$.
If $\lambda\in P^-$ is sufficiently small, then we have $v^*_\lambda y\ne0$.
Then the assertion is a consequence of the irreducibility of $V^*(\lambda)$ as a right $U$-module.
\end{proof}
\subsection{}
For $i\in I$ we define a $\BC_q[G]$-module $\overline{\CM}_i$ by
\begin{equation}
\overline{\CM}_i=\CM_i\otimes_{\BC_q[H(i)]}\BF_\varepsilon.
\end{equation}
It is an irreducible $\BC_q[G]$-module with basis 
$\{\overline{p}_i(n)\mid n\in\BZ_{\geqq0}\}$
satisfying
\begin{align*}
a_i\overline{p}_i(n)=&(1-q_i^{2n})\overline{p}_i(n-1),
\qquad
b_i\overline{p}_i(n)=q_i^{n}\overline{p}_i(n),
\\
c_i\overline{p}_i(n)=&-q_i^{n+1}\overline{p}_i(n)i,
\qquad
d_i\overline{p}_i(n)=\overline{p}_i(n+1).
\end{align*}

Fix $w\in W$, and set $\ell(w)=m$.
For $\Bi=(i_1,\cdots, i_m)\in \CI_w$ the isomorphism
\eqref{eq:main-isom} induces an isomorphism
\begin{equation}
\label{eq:main-isom2}
\overline{\CM}_w
\cong
\overline{\CM}_{i_1}\otimes\cdots\otimes\overline{\CM}_{i_m}
\end{equation}
of $\BC_q[G]$-modules.
For $\Bn=(n_1,\dots, n_m)\in(\BZ_{\geqq0})^m$ we denote by 
$\overline{p}_\Bi(\Bn)$ the element of $\overline{\CM}_w$ corresponding to 
\[
\overline{p}_{i_1}(n_1)\otimes\cdots\otimes \overline{p}_{i_m}(n_m)
\in
\overline{\CM}_{i_1}\otimes\cdots\otimes\overline{\CM}_{i_m}
\]
via the isomorphism \eqref{eq:main-isom2}.

By Theorem \ref{thm:pd} we have the following.
\begin{theorem}
\label{thm:KOY}
For  $\Bi, \Bj\in\CI_w$ we have
\begin{align*}
&\he_\Bj^{(\Bn)}=\sum_{\Bn'}a_{\Bn'}\he_\Bi^{(\Bn')}
\Longrightarrow\;
\overline{p}_\Bj(\Bn)
=
\sum_{\Bn'}a_{\Bn'}\frac{d_{\Bj,\Bn}}{d_{\Bi,\Bn'}}
\overline{p}_\Bi(\Bn'),
\end{align*}
where $d_{\Bi,\Bn}$ is as in Theorem \ref{thm:pd}.
\end{theorem}
\begin{remark}
Theorem \ref{thm:KOY} for $w=w_0$ is the main result of 
Kuniba, Okado, Yamada (\cite[Theorem 5]{KOY})．
\end{remark}

By Proposition \ref{prop:KOY-conj} we have the following.
\begin{proposition}
\label{prop:KOY-conj2}
For $\Bi\in\CI_{w_0}$,  $i\in I$, $\Bn\in(\BZ_{\geqq0})^{m_0}$ write 
\[
\he_\Bi^{(\Bn)}e_i=\sum_{\Bn'}c_{\Bn\Bn'}\he_\Bi^{(\Bn')}.
\]
Then we have
\begin{align*}
&\left\{
\frac1{1-q_i^2}
(\sigma^{w_0}_{-\varpi_{i'}}e_i)(\sigma^{w_0}_{-\varpi_{i'}})^{-1}
\right\}
\overline{p}_\Bi(\Bn)
=
\sum_{\Bn'}c_{\Bn\Bn'}
\frac{d_\Bi(\Bn)}{d_\Bi(\Bn')}
\overline{p}_{\Bi}(\Bn'),
\end{align*}
where $i'$ is as in Lemma \ref{lem:jmath-e}.
\end{proposition}
\begin{remark}
Proposition \ref{prop:KOY-conj2} is a conjecture of Kuniba, Okado, Yamada (\cite[Conjecture 1]{KOY}).
\end{remark}

\section{Comments}
\label{sec:comment}

\subsection{}
In this paper we worked over the base field $\BF=\BQ(q)$; however, almost all of the arguments work equally well after minor modifications even when $\BF$ is an arbitrary field of characteristic zero and $q_i^2\ne1$ for any $i\in I$.
The only exception is Theorem \ref{thm:irred}, which states that $\overline{\CM}_w$ is irreducible.
For this result we need to assume that $q$ is not a root of 1.
\subsection{}
Let us consider generalization of our results to the case where $\Gg$ is a symmetrizable Kac-Moody Lie algebra.
We take $\BC_q[G]$ to be the subspace of $U_q(\Gg)^*$ spanned by the matrix coefficients of integrable lowest weight modules (see \cite{Kas}).
Then $\BC_q[G]$ is naturally endowed with an algebra structure.
A problem is that the comultiplication $\Delta:\BC_q[G]\to\BC_q[G]\otimes\BC_q[G]$ is not defined.
Indeed $\Delta(\varphi)$ for $\varphi\in\BC_q[G]$ turns out to be an infinite sum which belongs to a completion of 
$\BC_q[G]\otimes\BC_q[G]$.
However, since we only consider the tensor product modules of type 
$\CM_{i_1}\otimes\cdots\otimes\CM_{i_m}$, what we actually need is the homomorphism of of the form
\begin{equation}
\label{eq:KM}
(r^G_{G(i_1)}\otimes\cdots \otimes r^G_{G(i_m)})\circ\Delta_{m-1}:\BC_q[G]\to
\BC_q[G(i_1)]\otimes\cdots\otimes\BC_q[G(i_{m})].
\end{equation}
We can easily check that \eqref{eq:KM} is well-defined even in the Kac-Moody setting by showing that $(r^G_{G(i_1)}\otimes\cdots \otimes r^G_{G(i_m)})\circ\Delta_{m-1}$ sends any element of $\BC_q[G]$ to a finite sum inside  $\BC_q[G(i_1)]\otimes\cdots\otimes\BC_q[G(i_{m})]$.
It is easily seen that  all of the arguments in this paper also work in the setting where $\Gg$ is a symmetrizable Kac-Moody Lie algebra.

\bibliographystyle{unsrt}

\end{document}